\documentclass[11pt]{article}
\usepackage{amsfonts}

\usepackage{graphics}
\usepackage{indentfirst}
\usepackage{cite}
\usepackage{latexsym}
\usepackage{amsmath,amsthm}
\usepackage{amssymb}
\usepackage{amscd}
\usepackage{mathrsfs}

\usepackage{tocloft}

\usepackage{color}

\newtheorem{theorem}{Theorem}[section]
\newtheorem{remark}{Remark}[section]

\newtheorem{definition}{Definition}[section]
\newtheorem{lemma}[theorem]{Lemma}

\newtheorem{proposition}[theorem]{Proposition}

\newcommand{\n}{\rho}

\renewcommand{\div}{ {\rm div }  }

\newcommand{\na}{\nabla }

\newcommand{\bt}{\begin{theorem}}
\newcommand{\bl}{\begin{lemma}}
\newcommand{\el}{\end{lemma}}
\newcommand{\et}{\end{theorem}}
\newcommand{\ga}{\gamma}

\newcommand{\OM}{\Omega}

\newcommand{\de}{\delta}

\newcommand{\la}{\label}
\newcommand{\si}{\sigma}
\newcommand{\ka}{\kappa}
\newcommand{\om}{\Omega}
\newcommand{\ol}{\overline}

\newcommand{\bn}{\begin{eqnarray}}
\newcommand{\en}{\end{eqnarray}}
\newcommand{\bnn}{\begin{eqnarray*}}
\newcommand{\enn}{\end{eqnarray*}}

\newcommand{\bnnn}{\begin{eqnarray*}}
\newcommand{\ennn}{\end{eqnarray*}}
\newcommand{\ben}{\begin{enumerate}}
\newcommand{\een}{\end{enumerate}}

\newcommand{\ba}{\begin{aligned}}
\newcommand{\ea}{\end{aligned}}
\newcommand{\be}{\begin{equation}}
\newcommand{\ee}{\end{equation}}

\def\p{\partial}
\def\norm[#1]#2{\|#2\|_{#1}}

\def\lam{\lambda}
\def\ep{\varepsilon}

\makeatletter      
\@addtoreset{equation}{section}
\makeatother

\title{Finite-Time Vanishing of Vacuum and Large-Time Behavior of the Multi-Dimensional Degenerate Compressible Navier-Stokes Equations with Large Spherically Symmetric Initial Data}

\date{}

\author{$\text{Qinghao L{\small EI}}^{a,b}, \text{Zhilei L{\small{IANG}}}^{c}\thanks{Email addresses:  leiqinghao22@mails.ucas.ac.cn (Q. H. Lei), liangzl@swufe.edu.cn (Z. L. Liang) }$\\
a. School of Mathematical Sciences,\\ University of Chinese Academy of Sciences,
Beijing 100049, P. R. China;\\
b. Institute of Applied Mathematics,\\ Academy of Mathematics and Systems Science, \\
Chinese Academy of Sciences, Beijing 100190, P. R. China;\\
c. School of  Mathematics, \\
Southwestern University of Finance and Economics, Chengdu 611130, P. R. China}

\begin{document}
\maketitle

\begin{abstract}
In this paper, we study the two- and three-dimensional barotropic compressible Navier-Stokes equations with degenerate density-dependent viscosity coefficients in the whole space or in a ball for arbitrarily large spherically symmetric initial data.
For initial data allowing vacuum, we establish the global existence of weak solutions and derive uniform-in-time a priori estimates.
As a consequence, we prove that the vacuum state of weak solutions will vanish in finite time.
The key ingredients include a treatment of the pressure terms based on separate estimates in the low- and high-density regions, the Bresch-Desjardins entropy estimates, weighted radial estimates, and a coupled control of the velocity and effective velocity.
In the three-dimensional case, this approach also allows us to treat the endpoint case of the adiabatic exponent.
For sufficiently regular initial data with strictly positive density, we establish the global existence and large-time behavior of classical solutions. 
\end{abstract}


\section{Introduction and main results}
We consider the multi-dimensional barotropic compressible
Navier-Stokes equations with density-dependent viscosity coefficients, which read as follows:
\be\ba\la{ns}
\begin{cases}
\rho_t + \div(\rho \mathbf{u}) = 0, \\
(\n \mathbf{u})_t + \div(\n \mathbf{u} \otimes \mathbf{u}) - \div( \mu(\n) \mathbb{D}\mathbf{u}) - \na( \lambda(\n) \div \mathbf{u} ) + \na P = 0,
\end{cases}
\ea\ee
where $t \ge 0$ is time, $x \in \OM \subset \mathbb{R}^N (N=2,3)$ is the spatial coordinate,
$\n=\n(x,t)$ and $\mathbf{u}(x,t)=(u^1(x,t),\dots,u^N(x,t))$ represent the 
density and velocity of the compressible flow, respectively.
The pressure $P$ is given by
\be\la{i1}
P=a \rho^\ga,
\ee
with constants $a>0$ and $\ga > 1$.
Without loss of generality, we assume that $a=1$.
The viscosity coefficients $\mu(\n)$ and $\lam(\n)$ satisfy the physical restrictions
\be\la{i2}
\mu(\n) \ge 0, \quad \mu(\n) + N \lam(\n) \ge 0,
\ee
and the deformation tensor $\mathbb{D}\mathbf{u}$ is defined by
\be\nonumber
\mathbb{D}\mathbf{u} = \frac{1}{2} ( \na \mathbf{u} + \na \mathbf{u}^T ).
\ee
The system is supplemented with the spherically symmetric initial data
\be\la{cz}
\n(x,0)=\n_0(x) = \n_0(r), \quad \n \mathbf{u}(x,0)=\mathbf{m}_0(x) = m_0(r) \frac{x}{r}, \quad x\in \OM,
\ee
where $r=|x|$.

In this paper, we investigate two types of boundary conditions:

(1) Cauchy problem: $\OM = \mathbb{R}^N$ and $(\n,\mathbf{u})$ satisfies the far-field behavior
\be\la{qkjbjtj}\ba
(\n,\mathbf{u})(x,t) \to (1,0) \ \text{ as } |x| \to \infty, \  t>0.
\ea\ee

(2) Bounded domain with boundary condition:
\be\la{yjybjtj}
\OM = B_R \triangleq \{ x \in \mathbb{R}^N \mid |x| < R \}, \quad \n \mathbf{u} = 0 \ \text{ on } \ \p \OM.
\ee

There is an extensive literature concerning the global existence of solutions to (\ref{ns}) when both viscosity coefficients $\mu(\n)$ and $\lambda(\n)$ are constants.
The one-dimensional problem has been studied extensively; see \cite{H4,Ka,KN,KS,S1,S2} and the references therein.
In the multi-dimensional setting, the local existence and uniqueness of classical solutions were established in \cite{N,S} under the assumption that the initial density is strictly positive.
For strong solutions, further results were obtained in \cite{CCK,CK,CK2,SS,LLL}, where the initial density is allowed to vanish.
The first global classical solutions were established by Matsumura-Nishida \cite{MN1} for initial data sufficiently close to a non-vacuum equilibrium in some Sobolev space $H^s$.
Later, Hoff \cite{H1,H2,H3} studied the problem with discontinuous initial data and introduced a new type of a priori estimates on the material derivative $\dot{\mathbf{u}}$.
For large initial data, a major breakthrough in the global existence theory of weak solutions was due to Lions \cite{L2}, who proved the global existence of finite-energy weak solutions for the pressure $P = a\n^\ga \ (a>0,\ga>1)$, under the assumption that $\ga$ is suitably large; for example, $\ga \ge \frac{9}{5}$ in the three-dimensional case.
This result was subsequently improved by Feireisl-Novotn\'y-Petzeltov\'a \cite{FNP} to $\ga>\frac{3}{2}$, and by Jiang-Zhang \cite{JZ} to $\ga>1$ in the spherically symmetric setting.
For initial data containing vacuum, Huang-Li-Xin \cite{HLX2} established the global existence and uniqueness of classical solutions to the three-dimensional Cauchy problem for smooth initial data with small energy but possibly large oscillations.
Subsequently, Li-Xin \cite{LX2} further derived some a priori decay rates for both the pressure and the spatial gradient of velocity, provided that the initial total energy is sufficiently small, and proved the global existence and large-time asymptotic behavior of strong and classical solutions to the Cauchy problem in two and three dimensions with far-field vacuum.
Recently, Cai-Li \cite{CL} established the global existence and exponential decay of classical solutions in three-dimensional bounded domains subject to slip boundary conditions, assuming that the initial total energy is sufficiently small.

It is worth noting that the case where the viscosity coefficients depend on the density and degenerate at vacuum has received extensive attention in recent years.
Indeed, as pointed out by Liu-Xin-Yang \cite{LXY}, in the derivation of the compressible Navier-Stokes equations from the Boltzmann equation through the Chapman-Enskog expansion (see \cite{CC}), the viscosity coefficients depend on the temperature, which, in the barotropic setting, leads to a dependence on the density.
When $\mu$ is a positive constant and $\lambda(\rho) = b \n^\beta$ with $b>0$ and $\beta>3$, Vaigant-Kazhikhov \cite{VK} first showed that the two-dimensional system (\ref{ns}) has a unique global strong solution for large initial data with density away from vacuum in rectangular domains subject to slip boundary conditions.
Subsequently, for periodic domains, Jiu-Wang-Xin \cite{JWX1} generalized this result by removing the assumption that the initial density is away from vacuum.
More recently, in two-dimensional periodic domains and in the whole space, Huang-Li \cite{HL2,HL3} (see also \cite{JWX2}) relaxed the restriction $\beta>3$ to $\beta>\frac{4}{3}$.
This problem was further investigated by Fan-Li-Li \cite{FLL} in general two-dimensional bounded simply connected domains with Navier-slip boundary conditions, where the global existence of strong and weak solutions was established under the condition $\beta>\frac{4}{3}$.
In addition, Huang-Su-Yan-Yu \cite{HSYY} proved the global existence of radially symmetric strong solutions in two-dimensional balls under the assumption $\beta>1$.

For density-dependent viscosity coefficients satisfying (\ref{i2}) and
\be\la{bde}
\lambda(\rho) = \mu'(\rho) \rho - \mu(\rho),
\ee
a remarkable framework was developed by Bresch-Desjardins \cite{BD1,BD2,BDL}, which provides an additional estimate for the density gradient.
Subsequently, Mellet-Vasseur \cite{MV} investigated the stability of (\ref{ns}) by deriving a new a priori estimate for smooth approximate solutions.
In the one-dimensional case, Li-Li-Xin \cite{LLX} established the global existence of weak solutions to (\ref{ns}) with $\mu(\n) = \n^\alpha (\alpha>\frac{1}{2})$ and $\lam(\n)=0$, and showed that the vacuum state of weak solutions must vanish within finite time.
Later, Jiu-Xin \cite{JX} generalized the result of \cite{LLX} to the one-dimensional Cauchy problem.
For spherically symmetric initial data, Guo-Jiu-Xin \cite{GJX} proved the global existence of weak solutions for large initial data allowing vacuum.
Furthermore, Guo-Li-Xin \cite{GLX} extended this result to the free boundary problem and studied the Lagrangian structure and dynamics.
For general initial data, Li-Xin \cite{LX1} and Vasseur-Yu \cite{VY} independently proved the global existence of weak solutions for large initial data with vacuum, while Bresch-Vasseur-Yu \cite{BVY} further extended these results to more general viscosity coefficients.
In the exterior domain of a ball in $\mathbb{R}^{N} \ (N=2,3)$, Cao-Li-Zhu \cite{CLZ} established the global existence and uniqueness of spherically symmetric classical solutions for large spherically symmetric initial data with far-field vacuum.
More recently, for spherically symmetric initial data, Zhang \cite{Z} first proved the global existence and uniqueness of classical solutions in a ball in $\mathbb{R}^{N}$ $(N=2,3)$.
Later, Huang-Meng-Zhang \cite{HMZ} generalized this result and obtained the large-time asymptotic behavior of the solutions.
For the Cauchy problem in two and three dimensions, Chen-Zhang-Zhu \cite{CZZ1,CZZ2} established the global existence and uniqueness of classical solutions for large spherically symmetric initial data with either vacuum or non-vacuum far-field density.

The main aim of the present paper is to investigate the global existence and finite-time vacuum vanishing of spherically symmetric weak solutions, as well as the global existence and large-time behavior of spherically symmetric classical solutions to system (\ref{ns})--(\ref{cz}) subject to the boundary conditions (\ref{qkjbjtj}) or (\ref{yjybjtj}) with large initial data.
More precisely, we consider viscosity coefficients of the form
\be\la{nxxs}
\mu(\n) = \n^\alpha, \quad \lambda(\n) = (\alpha-1) \n^\alpha,
\ee
which satisfy the BD entropy structure (\ref{bde}).

We seek spherically symmetric solutions of the form:
\be\la{qdc}
\n(x,t) = \n(r,t), \quad \mathbf{u}(x,t) = u(r,t) \frac{x}{r},
\ee
so that system (\ref{ns}) is transformed into
\be\la{nsqdc}\ba
\begin{cases}
\n_t + (\n u)_r + \frac{(N-1)}{r} \n u = 0, \\
\n(u_t + u u_r) + (\n^\ga)_r - \alpha \left( r^{-(N-1)} \n^\alpha (r^{N-1}u)_r \right)_r + \frac{N-1}{r}(\n^\alpha)_r u = 0.
\end{cases}
\ea\ee

Before stating the main results, we first explain the notation
and conventions used throughout this paper.

For any $1\leq s \leq \infty$ and a positive integer $l$, we define the standard Lebesgue and Sobolev spaces as follows:
\be\ba\nonumber
L^s =L^s(\OM),\quad W^{l,s} =W^{l,s}(\OM),\quad H^l = W^{l,2}.
\ea\ee
For the Cauchy problem, we define the potential energy density by
\be\la{pe}\ba
K(\n) \triangleq \n \int_{1}^{\n} \frac{P(s)-1}{s^2} ds.
\ea\ee
It is easy to verify that, for every $\hat{\n}>0$, there exists a positive constant $C=C(\ga,\hat{\n})$ such that, for all $0 \le \n \le \hat{\n}$,
\be\la{qkjsn}\ba
\frac{1}{C} (\n - 1)^2 \le K(\n)
\le C (\n - 1)^2.
\ea\ee

For later use, we introduce several auxiliary functions.
\begin{definition}
For $p \in (2,\infty)$, define
\be\la{cs1}
f_{1}(p) \triangleq \frac{ p^2 - 2p \sqrt{p-1} }{(p-2)^2}, \quad f_{2}(p) \triangleq \frac{ p^2 + 2p \sqrt{p-1} }{(p-2)^2}.
\ee
A direct calculation shows that $f_{1}(p)$ is strictly increasing on $(2,\infty)$, whereas $f_{2}(p)$ is strictly decreasing on $(2,\infty)$.

For $z \in (\frac{1}{2},1) \cup (1,\infty)$, define
\be\la{cs2}\ba
\varphi(z) \triangleq \frac{ 2z^2 + 2z\sqrt{2z-1} }{(z-1)^2}, \quad \varphi(1) = \infty.
\ea\ee
Then $\varphi(z)|_{(\frac{1}{2},1)}$ is the inverse function of $f_{1}(p)$, whereas $\varphi(z)|_{(1,\infty)}$ is the inverse function of $f_{2}(p)$.
Moreover, $\varphi(z)$ is strictly increasing on $(\frac{1}{2},1)$ and strictly decreasing on $(1,\infty)$, and
\be\nonumber
\lim_{z \to \frac{1}{2}^+} \varphi(z) = 2, \quad \lim_{z \to 1} \varphi(z) = \infty, \quad \lim_{z \to \infty} \varphi(z) = 2.
\ee

For $p \in (2,\infty)$, set
\be\la{cs3}
g_{1}(p) \triangleq \frac{p^2-p+1-\sqrt{2p^3-p^2-2p+1} }{(p-2)^2}, \quad g_{2}(p) \triangleq \frac{p^2-p+1+\sqrt{2p^3-p^2-2p+1} }{(p-2)^2}.
\ee
A direct computation shows that $g_{1}(p)$ is strictly increasing on $(2,\infty)$, whereas $g_{2}(p)$ is strictly decreasing on $(2,\infty)$.

For $z \in (\frac{2}{3},1) \cup (1,\infty)$, define
\be\la{cs4}\ba
\psi(z) \triangleq \frac{ z (2z-1) + \sqrt{ z(2z-1)(3z-2) } }{(z-1)^2}, \quad \psi(1) = \infty.
\ea\ee
Then $\psi(z)|_{(\frac{2}{3},1)}$ is the inverse function of $g_{1}(p)$, whereas $\psi(z)|_{(1,\infty)}$ is the inverse function of $g_{2}(p)$.
Furthermore, $\psi(z)$ is strictly increasing on $(\frac{2}{3},1)$ and strictly decreasing on $(1,\infty)$, and
\be\nonumber
\lim_{z \to \frac{2}{3}^+} \psi(z) = 2, \quad \lim_{z \to 1} \psi(z) = \infty, \quad \lim_{z \to \infty} \psi(z) = 2.
\ee
For notational convenience, we define
\be\la{pna}\ba
\mathcal{P}_{N}(\alpha) \triangleq
\begin{cases}
\varphi(\alpha) \quad \textnormal{ if } N=2, \\
\psi(\alpha) \quad \textnormal{ if } N=3.
\end{cases}
\ea\ee
We adopt the convention that $\mathcal{P}_N(1)=\infty$.
\end{definition}

Throughout the paper, we assume that the initial data $(\n_0,\mathbf{m}_0)$ satisfy
\be\la{czre}\ba
\begin{cases}
0 \le \n_0 \in L^\infty(\om), \quad \n_0 \not\equiv 0, \quad \n_0 - 1 \in L^2(\om), \quad \na \n_0^{\alpha-\frac{1}{2}} \in L^2(\om), \\
\n_0^{-1}|\mathbf{m}_0|^2 \in L^1(\om), \quad \mathbf{m}_0 = 0 \text{ a.e. on } \om_0,
\end{cases}
\ea\ee
where
\be\nonumber\ba
\om_0 \triangleq \{x \in \om \mid \n_0(x)=0 \},
\ea\ee
denotes the vacuum set of $\n_0$.
We also use the convention that $\n_0^{-1}|\mathbf{m}_0|^2 = 0$ a.e. on $\om_0$.

We now give the definition of weak solutions to (\ref{ns})--(\ref{nxxs}).
\begin{definition}
For $N=2,3$, let $\om = \mathbb{R}^N$ or $\om = B_R$.
$(\n,\mathbf{u})$ is said to be a weak solution to \eqref{ns}-- \eqref{nxxs} if
\be\la{rjdf1}\ba
\begin{cases}
\n - 1 \in L^\infty(0,T;L^2(\om) \cap L^\infty(\om) ), \ \n \in C(\ol{\om} \times (0,T)), \\
\na \n^{\alpha-\frac{1}{2}}, \sqrt{\n} \mathbf{u} \in L^\infty(0,T;L^2(\om)), \\
\na \n^{ \frac{\ga+\alpha-1}{2} } \in L^2(\om \times (0,T)), \\
\n^\alpha \na \mathbf{u} \in L^2(0,T;W^{-1,1}(B_R)) \quad \textnormal{ if } \om = B_R, \\
\n^\alpha \na \mathbf{u} \in L^2(0,T;W_{ \textnormal{loc} }^{-1,1}(\mathbb{R}^N)) \quad \textnormal{ if } \om = \mathbb{R}^N, \\
\end{cases}
\ea\ee
with $(\sqrt{\n}, \sqrt{\n} \mathbf{u})$ satisfying
\be\la{rjdf2}\ba
\begin{cases}
\n_t + \div(\sqrt{\n} \sqrt{\n} \mathbf{u}) = 0, \text{ in } \mathcal{D}'(\om \times (0,T)), \\
\n(x,t=0) = \n_0(x),
\end{cases}
\ea\ee
and if the following equality holds for every smooth test function $\phi(x,t) = (\phi(x,t)^1,\dots,\phi(x,t)^N)$ with compact support such that $\phi(x,T) = 0$:
\be\la{rjdf3}\ba
& \int_{\om} \mathbf{m}_0 \cdot \phi(x,0) dx
+ \int_0^T \int_{\om} \left( \sqrt{\n} \sqrt{\n} \mathbf{u} \cdot \phi_t + (\sqrt{\n} \mathbf{u} \otimes \sqrt{\n} \mathbf{u}) : \na \phi + \n^\ga \div \phi \right) dx dt \\
& - \left< \n^\alpha \na \mathbf{u}, \na \phi \right> - (\alpha-1) \left< \n^\alpha \div \mathbf{u}, \div \phi \right> = 0,
\ea\ee
where
\be\la{rjdf4}\ba
\left< \n^\alpha \na \mathbf{u}, \na \phi \right> \triangleq
- \int_0^T \int_{\om} \left( \n^{\alpha-\frac{1}{2}} \sqrt{\n} \mathbf{u} \cdot \Delta \phi + \frac{2\alpha}{2\alpha-1} \na \n^{\alpha-\frac{1}{2}} \cdot \na \phi \cdot \sqrt{\n} \mathbf{u} \right) dx dt,
\ea\ee
\be\la{rjdf5}\ba
\left< \n^\alpha \div \mathbf{u}, \div \phi \right> \triangleq
- \int_0^T \int_{\om} \left( \n^{\alpha-\frac{1}{2}} \sqrt{\n} \mathbf{u} \cdot \na \div \phi + \frac{2\alpha}{2\alpha-1} \na \n^{\alpha-\frac{1}{2}} \cdot \sqrt{\n} \mathbf{u} \cdot \div \phi \right) dx dt.
\ea\ee
\end{definition}

Our first result concerns the global existence of weak solutions.

\begin{theorem}\la{th1}
Let $N=2$ or $3$, and let $\om=\mathbb{R}^N$ or $\om = B_R$.
Suppose that the initial data $(\n_0,\mathbf{m}_0)$ are spherically symmetric and satisfy \eqref{czre} and
\be\la{th12}\ba
\na \n_0^{\alpha-1+\frac{1}{s_1}} \in L^{s_1}(\om), \quad \n_0^{-p_1+1}|\mathbf{m}_0|^{p_1} \in L^1(\om),
\ea\ee
for some $s_1$ and $p_1$ satisfying
\be\la{th13}\ba
N < s_1 \le p_1 < \mathcal{P}_{N}(\alpha), \quad \alpha-1+\frac{1}{s_1}>0,
\ea\ee
where we use the convention that $\n_0^{-{p_1}+1}|\mathbf{m}_0|^{p_1} = 0$ a.e. on $\om_0$.
Assume that $\alpha$ and $\ga$ satisfy
\be\la{th110}\ba
\ga > \frac{\alpha+1}{2} - \frac{1}{N},
\ea\ee
and
\be\la{th11}\ba
\begin{cases}
\alpha > \frac{1}{2}, \quad 1<\ga<\infty & \textnormal{ if } N=2, \\
7-2\sqrt{10} < \alpha < 7+2\sqrt{10}, \quad 1 < \ga \le 6\alpha-3+\eta_0 & \textnormal{ if } N=3,
\end{cases}
\ea\ee
for some constant $\eta_0 \in (0,1]$ depending only on $s_1$, $\alpha$, $\mu$, and the initial data.
Then the problem \eqref{ns}, \eqref{i1}, \eqref{cz}, \eqref{nxxs} with the boundary condition \eqref{qkjbjtj} or \eqref{yjybjtj} has a global spherically symmetric weak solution.
\end{theorem}

\begin{remark}\la{rk1}
It is worth noting that, in Theorem \ref{th1}, when $\alpha \ge 1$, the condition
\be\nonumber\ba
\alpha-1+\frac{1}{s_1}>0,
\ea\ee
in \eqref{th13} is automatically satisfied.
When $\alpha<1$, the condition \eqref{th13} is equivalent to
\be\nonumber\ba
N < s_1 \le p_1 < \mathcal{P}_{N}(\alpha), \quad s_1 < \frac{1}{1-\alpha}.
\ea\ee
Such exponents $s_1$ and $p_1$ exist under the assumption \eqref{th11}.
Indeed, $\alpha > \frac{N-1}{N}$ implies
\be\nonumber\ba
\frac{1}{1-\alpha} > N,
\ea\ee
while the admissible range of $\alpha$ guarantees $\mathcal{P}_{N}(\alpha) > N$.
\end{remark}

The following result shows that the vacuum state of weak solutions in Theorem \ref{th1} will vanish in finite time.

\begin{theorem}\la{th3}
Let $N=2$ or $3$, and let $\om=\mathbb{R}^N$ or $\om = B_R$.
Suppose that the initial data $(\n_0,\mathbf{m}_0)$ are spherically symmetric and satisfy \eqref{czre} and
\be\la{th33}\ba
\na \n_0^{\alpha-1+\frac{1}{s_1}} \in L^{s_1}(\om), \quad \n_0^{-p_1+1}|\mathbf{m}_0|^{p_1} \in L^1(\om),
\ea\ee
for some $s_1$ and $p_1$ satisfying
\be\la{th34}\ba
N < s_1 \le p_1 < \mathcal{P}_{N}(\alpha), \quad \alpha-1+\frac{1}{s_1}>0,
\ea\ee
where we use the convention that $\n_0^{-p_1+1}|\mathbf{m}_0|^{p_1} = 0$ a.e. on $\om_0$.
Assume that $\alpha$ and $\gamma$ satisfy \eqref{th110} and \eqref{th11}.
Then there exist positive constants $T_0$ and $\n_{-}$ depending on the initial data such that the global spherically symmetric weak solution obtained in Theorem \ref{th1} satisfies
\be\la{th36}\ba
\inf_{x \in \ol{\om}} \n(x,t) \ge \n_{-}>0, \quad \textnormal{ for } t \in [T_0,\infty).
\ea\ee
\end{theorem}

The following result establishes the global existence and large-time behavior of classical solutions.

\begin{theorem}\la{th4}
Let $N=2$ or $3$, and let $\om=\mathbb{R}^N$ or $\om = B_R$.
Suppose that the initial data $(\n_0,\mathbf{u}_0)$ are spherically symmetric and satisfy
\be\la{th42}\ba
0 < \underline{\n_0} \le \n_0 \le \hat{\n_0}, \quad
\n_0 - 1 \in H^3(\om), \quad \mathbf{u}_0 \in H^3(\om),
\ea\ee
where $\underline{\n_0}$ and $\hat{\n_0}$ are positive constants.
In the case $\OM = B_R$, we further assume that $\mathbf{u}_0 \in H_0^1(B_R)$.
Assume that $\alpha$ and $\ga$ satisfy
\be\la{th41}\ba
\begin{cases}
\frac{1}{2} < \alpha \le 1, \quad 1 < \ga <\infty & \textnormal{ if } N=2, \\
7-2\sqrt{10} < \alpha \le 1, \quad 1< \ga \le 6\alpha-3+\eta_1 & \textnormal{ if } N=3,
\end{cases}
\ea\ee
for some constant $\eta_1 \in (0,1]$ depending only on $\alpha$, $\mu$, and the initial data.
Then the problem \eqref{ns}, \eqref{i1}, \eqref{cz}, \eqref{nxxs} with the boundary condition \eqref{qkjbjtj} or \eqref{yjybjtj} has a unique global spherically symmetric classical solution $(\n,\mathbf{u})$ in $\om \times (0,\infty)$ satisfying
\be\la{th43}\ba
C^{-1} \le \n(x,t) \le C, \quad \textnormal{ for all } (x,t)\in \om \times[0,\infty),
\ea\ee
and for any $0 < T <\infty$,
\be\la{th44}\ba
\begin{cases}
\n - 1 \in C([0,T];H^3(\om)), \n_t \in C([0,T];H^2(\om)), \\
\mathbf{u} \in C([0,T];H^3(\om)) \cap L^2(0,T;H^4(\om)), \\
\mathbf{u}_t \in C([0,T];H^1(\om)) \cap L^2(0,T;H^2(\om)), \mathbf{u}_{tt} \in L^2(\om \times (0,T)),
\\
\sqrt{t} \na^4 \mathbf{u}, \sqrt{t} \na^2 \mathbf{u}_{t}, \sqrt{t} \mathbf{u}_{tt} \in L^\infty(0,T;L^2(\om)), \sqrt{t} \na \mathbf{u}_{tt} \in L^2(\om \times (0,T)),
\end{cases}
\ea\ee
where $C>0$ is a constant depending on $\alpha$, $\ga$, $\underline{\n_0}$, $\| \n_0 - 1 \|_{H^3(\om)}$, and $\| \mathbf{u}_0 \|_{H^2(\om)}$ but is independent of $T$.

Moreover, there exists $q \in (N,4]$ such that the following large-time behavior holds.
If $\om=\mathbb{R}^N$, then for any $s \in (2,\infty)$ and any $q_0 \in (2,q)$,
\be\la{th45}\ba
\lim_{t \to \infty} \left( \| \n(\cdot,t) - 1 \|_{L^s(\mathbb{R}^N)} \right.
& + \| \na \n(\cdot,t) \|_{L^2(\mathbb{R}^N) \cap L^{q_0}(\mathbb{R}^N)} + \| \n_t(\cdot,t) \|_{L^2(\mathbb{R}^N)}  \\
& \left. + \| \na \mathbf{u}(\cdot,t) \|_{H^1(\mathbb{R}^N)} + \| \mathbf{u}_t(\cdot,t) \|_{L^2(\mathbb{R}^N)} \right) = 0.
\ea\ee
If $\om=B_R$, there exist positive constants $C$ and $\hat{\de}$ independent of $t$ such that for any $t\ge 0$,
\be\la{th46}\ba
\| \na \n(\cdot,t) \|^2_{L^2(B_R)} + \| \n_t(\cdot,t) \|^2_{L^2(B_R)} + \| \mathbf{u}(\cdot,t) \|^2_{H^2(B_R)} + \| \mathbf{u}_t(\cdot,t) \|^2_{L^2(B_R)}
\le C e^{- \hat{\de} t},
\ea\ee
and for any $q_0 \in (2,q)$,
\be\la{th46a}\ba
\| \n(\cdot,t) - \ol{\n_0} \|_{W^{1,q_0}(B_R)} \le C e^{- \tilde{\de} t},
\ea\ee
where $\tilde{\de}>0$ depends on $q_0$ but is independent of $t$, and $\ol{\n_0}$ denotes
\be\nonumber\ba
\ol{\n_0} = \frac{1}{|B_R|} \int_{B_R} \n_0(x) dx.
\ea\ee
\end{theorem}

\begin{remark}\la{rk2}
We clarify the meaning of the boundary condition \eqref{yjybjtj} in the bounded-domain case.
Since $\mathbf{u}(x,t) = u(r,t) \frac{x}{r}$, the continuity equation can be written in the weak form
\be\la{rk21}\ba
\int_0^R \n \tilde{\varphi} r^{N-1} dr |_{t_1}^{t_2} = \int_{t_1}^{t_2} \int_0^R (\n \tilde{\varphi}_t + \n u \tilde{\varphi}_r ) r^{N-1} dr dt,
\ea\ee
for any $\tilde{\varphi} \in C^1( [0,R] \times [t_1,t_2] )$.
By approximation, \eqref{rk21} also holds for Lipschitz continuous test functions.
Let $\tilde{\varphi}_1(t)$ and $\tilde{\varphi}_2(r)$ be Lipschitz continuous functions satisfying $\tilde{\varphi}_1(t) \equiv 1$ on $[t_1,t_2]$ and
\be\nonumber\ba
\tilde{\varphi}_2(r) =
\begin{cases}
1 \quad & r \in [0,R-\de], \\
1 - \frac{1}{\de}(r - (R-\de)) & r \in [R-\de,R].
\end{cases}
\ea\ee
Choosing $\tilde{\varphi}(r,t) = \tilde{\varphi}_1(t) \tilde{\varphi}_2(r)$ in \eqref{rk21} and using conservation of mass, we obtain
\be\nonumber\ba
\frac{1}{\de} \left| \int_{t_1}^{t_2} \int_{R-\de}^R \n u r^{N-1} dr dt \right|
& \le \left| \int_{R-\de}^R \n(r,t_2) (\tilde{\varphi}_2(r) - 1) r^{N-1} dr \right. \\
& \qquad \qquad - \left. \int_{R-\de}^R \n(r,t_1) (\tilde{\varphi}_2(r) - 1) r^{N-1} dr \right| \to 0,
\ea\ee
as $\de \to 0$.
This shows that $(\n u)(R,t)=0$ in the sense of trace.
\end{remark}

\begin{remark}\la{rk3}
For the one-dimensional problem \eqref{ns} with $\mu(\n) = \n^\alpha$, $\alpha>\frac{1}{2}$, and $\lam(\n)=0$, Li-Li-Xin \cite{LLX} and Jiu-Xin \cite{JX} established the global existence of weak solutions and proved that the vacuum state of such solutions will vanish in finite time in bounded domains and in the whole space, respectively.

Subsequently, Huang-Meng-Zhang \cite{HMZ} proved that the vacuum state of spherically symmetric weak solutions vanishes in finite time in two- and three-dimensional balls under the assumptions
\be\nonumber\ba
& N=2, \quad 1 \le \alpha < 7.46, \quad \ga>1, \quad \ga \ge 2\alpha-1; \\
& N=3, \quad 1 \le \alpha < 5.81, \quad \ga>1, \quad 2\alpha-1 \le \ga < 3\alpha-1.
\ea\ee
Therefore, Theorem \ref{th3} generalizes and extends the results in \cite{HMZ,LLX,JX}.
\end{remark}

\begin{remark}\la{rk4}
In \cite{HMZ}, the authors established the large-time behavior of classical solutions in two- and three-dimensional balls under the conditions
\be\nonumber\ba
& N=2, \quad 0.54 \le \alpha \le 1, \quad \ga>1; \\
& N=3, \quad 0.689 \le \alpha \le 1, \quad 1 < \ga < 3\alpha-1.
\ea\ee
Observe that
\be\nonumber\ba
7-2\sqrt{10} \approx 0.6754 < 0.689,
\ea\ee
and that, for $\alpha>\frac{2}{3}$,
\be\nonumber\ba
3\alpha-1<6\alpha-3.
\ea\ee
Consequently, Theorem \ref{th4} generalizes and extends the large-time behavior result in \cite{HMZ}.
\end{remark}

We now make some comments on the analysis of this paper.
For brevity, we focus on the whole-space case and indicate only the necessary modifications for the problem in the ball $B_R$.
We first discuss the global existence of weak solutions and finite-time vanishing of vacuum.
The proof is based on the construction of a family of global smooth approximate solutions and the derivation of global a priori estimates.
To this end, for any $0<\ep<1$, we consider the approximate system:
\be\ba\nonumber
\begin{cases}
(\rho_\ep)_t + \div(\rho_\ep \mathbf{u}_\ep) = 0, \\
(\n_\ep \mathbf{u}_\ep)_t + \div(\n_\ep \mathbf{u}_\ep \otimes \mathbf{u}_\ep) - \div( \mu_\ep(\n_\ep) \mathbb{D}\mathbf{u}_\ep) - \na( \lambda_\ep(\n_\ep) \div \mathbf{u}_\ep ) + \na \n_\ep^\ga = 0,
\end{cases}
\ea\ee
in the exterior domain
\be\nonumber\ba
\om_\ep \triangleq \{ x \in \mathbb{R}^N \mid \ep<|x|<\infty \}.
\ea\ee
The viscosity coefficients are given by
\be\nonumber\ba
\mu_\ep(\n_\ep) = \n_\ep^\alpha + \ep \n_\ep^\theta, \quad \lam_\ep(\n_\ep) = (\alpha-1) \n_\ep^\alpha + \ep (\theta-1) \n_\ep^\theta,
\ea\ee
where $\theta$ is defined in (\ref{rgylt1}) and satisfies
\be\nonumber\ba
\frac{N-1}{N} < \theta \le 1, \quad \mathcal{P}_{N}(\alpha) \le \mathcal{P}_{N}(\theta).
\ea\ee
For smooth spherically symmetric initial data with density away from vacuum, the local existence and uniqueness of spherically symmetric classical solutions to the approximate system can be established by standard arguments similar to those in \cite{CZZ1,CZZ2,JR,LPZ,ZZ}.
To extend the local classical solution globally in time and prove that the vacuum state of weak solutions will vanish in finite time, we need to derive global a priori estimates.
A central step is to obtain the upper and lower bounds for the density.
Using the fact that $\alpha,\theta>\frac{N-1}{N}$, we first derive the standard energy estimates and BD entropy estimates.
Together with Sobolev inequalities, these estimates yield higher integrability of the density.
However, the approximate radial profiles $(\n_\ep,u_\ep)$ are defined only on $(\ep,\infty)$.
A direct application of Sobolev inequalities on this interval may therefore produce constants depending on $\ep$.
To avoid this difficulty, we extend the approximate density to $(0,\infty)$ by reflection across $r=\ep$.
Then, by a direct calculation, we obtain for any $2 \le p <\infty$,
\be\la{ca1}\ba
\sup_{0\le t \le T} \int_\ep^\infty |\n_\ep^{\alpha-\frac{1}{2}}-1|^p r dr \le C(p) \quad \text{ if } N=2,
\ea\ee
and
\be\la{ca2}\ba
\sup_{0\le t \le T} \int_\ep^\infty \left( |\n_\ep^{\alpha-\frac{1}{2}}-1|^2 + |\n_\ep^{\alpha-\frac{1}{2}}-1|^6 \right) r^2 dr \le C \quad \text{ if } N=3.
\ea\ee
These estimates, combined with the Gagliardo-Nirenberg inequality and the one-dimensional Sobolev inequality, imply that there exist positive constants $R_0 \ge 2$, $c_1$, and $c_2$ independent of $\ep$ and $T$ such that
\be\la{ca3}\ba
0<c_1 \le \n_\ep(r,t) \le c_2 \quad \textnormal{ for all } (r,t) \in [R_0,\infty) \times (0,T).
\ea\ee
Thus, the density is uniformly bounded from above and away from zero outside a fixed neighborhood of the origin.
Using (\ref{ca3}), we further obtain for any $2 \le p \le 6$,
\be\la{ca30}\ba
\left( \int_\ep^{R_0} \chi_{(\n_\ep \le \frac{c_1}{2})} r^{N-1} dr \right)^{\frac{2}{p}}
\le C \int_\ep^{R_0} |\p_r \n_\ep^{\frac{\ga+\alpha-1}{2}}|^{2} r^{N-1} dr.
\ea\ee
These estimates play a crucial role in the subsequent analysis.
In addition, by combining the BD entropy estimates, the one-dimensional Sobolev inequality, the reflected extension of the approximate density, and the Hardy-type inequality in Lemma \ref{hdi}, we can derive the following estimates:
\be\la{ca4}\ba
\sup_{0 \le t \le T} \| \n_\ep r^{\xi} \|_{L^\infty(\ep,R_0)} \le C(\xi) \quad \text{ for any } \xi>0, \quad \text{ if } N=2,
\ea\ee
\be\la{ca6}\ba
\sup_{0 \le t \le T} \| \n_\ep^{2\alpha-1} r \|_{L^\infty(\ep,R_0)} \le C, \quad \text{ if } N=3,
\ea\ee
and for any $t \in [0,T]$,
\be\la{ca61}\ba
\int_\ep^{R_0} \chi_{(\n_\ep > 2c_2)} \n_\ep^{\ga+\alpha-1}(r,t) dr
\le C \int_\ep^{R_0} |\p_r \n_\ep^{\frac{\ga+\alpha-1}{2}}(r,t)|^2 r^2 dr.
\ea\ee
We next estimate the density in the interior region $(\ep,R_0)$.
The key point is to obtain an $L^\infty(0,T;L^{p}(\ep,R_0))$ estimate, with $p>N$, for
\be\nonumber\ba
r^{\frac{N-1}{p}} \p_r \n_\ep^{\alpha-1+\frac{1}{p}}.
\ea\ee
This estimate follows from the corresponding estimates for $(\n_\ep r^{N-1})^{\frac{1}{p}} u_\ep$ and $(\n_\ep r^{N-1})^{\frac{1}{p}} w_\ep$, where $w_\ep$ is the effective velocity defined in (\ref{bjyxsd}).
Let $s_1$ and $p_1$ be the initial-data exponents in Theorem \ref{th1}.
Under the assumptions of Theorem \ref{th1}, we may choose $s_2>N$ sufficiently close to $N$ such that
\be\la{ca60}\ba
N < s_2 \le \min\left\{s_1,4\right\}, \quad \frac{1}{s_2} \ge \frac{\alpha+1}{2} - \ga.
\ea\ee
Multiplying $(\ref{bjnsqdc})_2$ by $u_\ep (1+|u_\ep|^2)^{\frac{s_2-2}{2}} r^{N-1}$, integrating by parts over $(\ep,\infty)$, and using Young's inequality and the fact that $N<s_2 \le s_1<\mathcal{P}_{N}(\alpha)\le\mathcal{P}_{N}(\theta)$, we obtain
\be\la{ca8}\ba
& \frac{d}{dt} \int_\ep^\infty \frac{1}{s_2} \n_\ep \left( (1+|u_\ep|^2)^{\frac{s_2}{2}}-1 \right) r^{N-1} dr + \mathcal{D}_{\ep,s_2}(t) \\
& \le - \int_\ep^\infty \p_r (\n_\ep^\ga) u_\ep (1+|u_\ep|^2)^{\frac{s_2-2}{2}} r^{N-1} dr,
\ea\ee
where $\mathcal{D}_{\ep,s_2}(t)$ denotes the corresponding nonnegative viscous dissipation.

In addition, the potential energy $K(\n_\ep)$ satisfies
\be\la{ca9}\ba
\frac{d}{dt} \int_\ep^\infty K(\n_\ep) r^{N-1} dr = \int_\ep^\infty \p_r(\n_\ep^\ga) u_\ep r^{N-1} dr.
\ea\ee
Adding (\ref{ca9}) to (\ref{ca8}) yields
\be\la{ca10}\ba
& \frac{d}{dt} \int_\ep^\infty \left[ \frac{1}{s_2} \n_\ep \left( (1+|u_\ep|^2)^{\frac{s_2}{2}}-1 \right) + K(\n_\ep) \right] r^{N-1} dr + \mathcal{D}_{\ep,s_2}(t) \\
& \le - \int_\ep^\infty \p_r (\n_\ep^\ga) \left[ u_\ep (1+|u_\ep|^2)^{\frac{s_2-2}{2}} - u_\ep \right] r^{N-1} dr.
\ea\ee
We next estimate the pressure term in (\ref{ca10}).
We need to handle the singularities at the origin and at infinity simultaneously.
Choose a smooth cutoff function $\eta(r) \in C^\infty([0,\infty))$ satisfying
\be\nonumber\ba
0 \le \eta(r) \le 1, \quad
\eta(r) =
\begin{cases}
0 & \text{ if } 0 \le r \le R_0, \\
1 & \text{ if } r \ge 2 R_0.
\end{cases}
\ea\ee
Then
\be\la{ca11}\ba
& - \int_\ep^\infty \p_r (\n_\ep^\ga) \left[ u_\ep (1+|u_\ep|^2)^{\frac{s_2-2}{2}} - u_\ep \right] r^{N-1} dr \\
& = - \int_\ep^\infty \p_r \left(\n_\ep^\ga - \eta(r)\right) \left[ u_\ep (1+|u_\ep|^2)^{\frac{s_2-2}{2}} - u_\ep \right] r^{N-1} dr \\
& \quad - \int_\ep^\infty \p_r \left( \eta(r) \right) \left[ u_\ep (1+|u_\ep|^2)^{\frac{s_2-2}{2}} - u_\ep \right] r^{N-1} dr.
\ea\ee
A direct calculation shows that, for $N < s_2 \le 4$,
\be\la{ca12}\ba
(1+|u_\ep|^2)^{\frac{s_2-2}{2}} - 1 \le \min\{ |u_\ep|^2,|u_\ep|^{s_2-2}\}.
\ea\ee
After integrating by parts in (\ref{ca11}), we estimate the resulting terms using (\ref{ca3})--(\ref{ca61}), (\ref{ca12}), and Young's inequality.
Substituting the resulting estimate into (\ref{ca10}) and integrating over $(0,T)$, we obtain
\be\la{ca120}\ba
& \sup_{0 \le t \le T} \int_\ep^\infty \n_\ep |u_\ep|^{s_2} r^{N-1} dr
+ \int_0^T \mathcal{D}_{\ep,s_2}(t) dt \\
& \le C + C \int_0^T \left( \int_\ep^\infty \left( \n_\ep |u_\ep|^{s_2} + \n_\ep |w_\ep|^{s_2} \right) r^{N-1} dr \right)^{1+\tilde{b}} \left( \int_\ep^{R_0} |\p_r \n_\ep^{\frac{\ga+\alpha-1}{2}}|^{2} r^{N-1} dr \right) dt,
\ea\ee
for some $\tilde{b} \ge 0$.

On the other hand, multiplying $(\ref{bjnsqdc})_2$ and (\ref{bjyxsdfc1}) by $|u_\ep|^{s_2-2} u_\ep r^{N-1}$ and $|w_\ep|^{s_2-2} w_\ep r^{N-1}$, respectively, integrating by parts over $(\ep,\infty)$, and subtracting the momentum identity from the effective-velocity identity, we arrive at
\be\la{ca13}\ba
& \frac{1}{s_2} \frac{d}{dt} \int_\ep^\infty \left( \n_\ep |w_\ep|^{s_2} - \n_\ep |u_\ep|^{s_2} \right) r^{N-1} dr \\
& = - \int_\ep^\infty \frac{\ga \n_\ep^{\ga}}{ \alpha \n_\ep^{\alpha-1} + \ep \theta \n_\ep^{\theta-1} } (w_\ep - u_\ep) \left( w_\ep |w_\ep|^{s_2-2} - u_\ep |u_\ep|^{s_2-2} \right) r^{N-1} dr + \tilde{\mathcal{D}}_{\ep,s_2}(t),
\ea\ee
where $\tilde{\mathcal{D}}_{\ep,s_2}(t)$ consists of viscous terms that are controlled by the dissipation term $\mathcal{D}_{\ep,s_2}(t)$.
Since the function $f(z) = z |z|^{s_2-2}$ is strictly increasing and $\n_\ep>0$, we have
\be\la{ca14}\ba
- \int_\ep^\infty \frac{\ga \n_\ep^{\ga}}{ \alpha \n_\ep^{\alpha-1} + \ep \theta \n_\ep^{\theta-1} } (w_\ep - u_\ep) \left( w_\ep |w_\ep|^{s_2-2} - u_\ep |u_\ep|^{s_2-2} \right) r^{N-1} dr \le 0.
\ea\ee
Integrating (\ref{ca13}) over $(0,T)$ and using (\ref{ca14}), we combine the resulting estimate with (\ref{ca120}).
Applying a Gr\"onwall-type inequality together with the BD entropy estimates and using the definition of $w_\ep$, we can get the time-uniform $L^\infty(0,T;L^{s_2}(\ep,\infty))$ bound for $r^{\frac{N-1}{s_2}} \p_r \n_\ep^{\alpha-1+\frac{1}{s_2}}$, provided that $\tilde{b}$ is suitably small.
This, together with the fundamental theorem of calculus, gives the desired time-uniform upper bound for the approximate density:
\be\la{ca7}
\sup_{0 \le t \le T} \| \n_\ep \|_{L^\infty(\ep,\infty)} \le C.
\ee
Moreover, using (\ref{ca3}), (\ref{ca7}), the fact that $\theta \le 1$, and the Lagrangian coordinates, and adapting the arguments in \cite{CZZ1,CZZ2,HMZ,Lei}, we can obtain a strictly positive lower bound for $\n_\ep$.
Once these upper and lower bounds are available, the standard higher-order estimates can be derived by adapting the arguments in \cite{CZZ1,CZZ2,HMZ,Lei}.
Consequently, the local classical solution $(\n_\ep,\mathbf{u}_\ep)$ can be extended globally in time.
Using the $\ep$-uniform estimates and adapting the compactness arguments in \cite{GJX,LX1,MV}, we pass to the limit as $\ep \to 0$ and obtain the global existence of weak solutions.
Furthermore, combining the time-uniform estimates with the standard arguments, we prove that the vacuum state of weak solutions will vanish in finite time.

In addition, for initial data away from vacuum, the same method also yields time-uniform upper and lower bounds for the density and, in turn, the large-time behavior of the corresponding classical solution.

The rest of this paper is organized as follows:
In Section 2, we recall some fundamental inequalities and known facts.
In Section 3, we introduce the approximate system and derive the basic a priori estimates.
In Section 4, we derive a time-uniform upper bound for the density.
In Section 5, we derive an $\ep$-dependent positive lower bound for the density of the approximate solutions.
In Section 6, we prove the global existence of smooth solutions to the approximate system and establish the compactness results needed for passing to the limit.
In Section 7, we present the proofs of Theorems \ref{th1} and \ref{th3}.
In Section 8, we prove Theorem \ref{th4}.
Finally, in Section 9, we give a detailed construction of the smooth approximate initial data.

We treat only the Cauchy problem in the main text.
The necessary modifications for the bounded domain case are sketched in Sections 7 and 8.

\section{Preliminaries}
In this section, we collect several known results and elementary inequalities that will be used frequently in subsequent analysis.

We first recall the following Gagliardo-Nirenberg inequalities (see \cite{NI}).
\begin{lemma}\la{gn1}
Let $p,q,r\in [1,\infty]$, and let $j$ and $m$ be integers satisfying $0\le j < m$.
Assume that
\be\nonumber\ba
\frac{1}{p} = \frac{j}{N} + a \left( \frac{1}{r} - \frac{m}{N} \right) + \frac{1-a}{q},
\ea\ee
where $a \in [\frac{j}{m},1]$.
If $m-j-\frac{N}{r}$ is a non-negative integer, we assume in addition that $a<1$.
If $j=0$, $mr<N$, and $q=\infty$, we further assume that either $f$ tends to zero at infinity or that $f \in L^k$ with $1 \le k <\infty$.
Then for $f \in L^q(\mathbb{R}^N) \cap D^{m,r}(\mathbb{R}^N)$, there exists a positive constant $C$ depending only on
$p$, $q$, $r$, $N$, $j$, $m$, and $a$ such that
\be\ba\la{gn11}
\| \na^j f \|_{L^p(\mathbb{R}^N)} \le C \| f \|^{1-a}_{L^q(\mathbb{R}^N)} \| \na^m f\|^{a}_{L^r(\mathbb{R}^N)}.
\ea\ee
\end{lemma}

The following one-dimensional Sobolev inequality will be used frequently.
\begin{lemma}\la{ywsi}
Let $\ep \in (0,1)$ and $f \in W^{1,1}(\ep,\infty)$.
Then for any $\ep \le r < \infty$,
\be\la{ywsi1}\ba
|f(r)| \le \int_{r}^\infty |f'(s)| ds.
\ea\ee
Moreover, for any $\ep \le r_1 < r_2 < \infty$, we have
\be\la{ywsi2}\ba
|f(r)| \le \frac{1}{r_2-r_1} \int_{r_1}^{r_2} |f(s)| ds
+ \int_{r_1}^{r_2} |f'(s)| ds \quad \textnormal{ for all } r \in [r_1,r_2].
\ea\ee
\end{lemma}
\begin{proof}
First, since $f \in W^{1,1}(\ep,\infty)$, it follows from \cite{AF,EL} that $f \in C([\ep,\infty))$ and
\be\la{ywp1}\ba
f(r) = f(M) - \int_r^M f'(s) ds \quad \text{ for all } \ep \le r < M < \infty .
\ea\ee
Thus, for any $\ep \le R_1<R_2<\infty$,
\be\la{ywp2}\ba
|f(R_1) - f(R_2)| = \left| \int_{R_1}^{R_2} f'(s) ds \right|
\le \int_{R_1}^{R_2} |f'(s)| ds.
\ea\ee
Since $f' \in L^1(\ep,\infty)$, the right-hand side of (\ref{ywp2}) tends to zero as $R_1, R_2 \to \infty$.
Hence, $\lim_{M \to \infty} f(M)$ exists.
Combining this with $f \in L^1(\ep,\infty)$ leads to
\be\la{ywp3}\ba
\lim_{M \to \infty} f(M) = 0.
\ea\ee
Letting $M \to \infty$ in (\ref{ywp1}), we obtain
\be\la{ywp4}\ba
f(r) = - \int_r^\infty f'(s) ds,
\ea\ee
which immediately yields (\ref{ywsi1}).

Next, for any $r, \hat{r} \in (r_1,r_2]$, we deduce from (\ref{ywp1}) that
\be\la{ywp6}\ba
|f(r)| = \left| f(\hat{r}) - \int_{r}^{\hat{r}} f'(s) ds \right|
\le |f(\hat{r})| + \int_{r_1}^{r_2} |f'(s)| ds.
\ea\ee
Integrating (\ref{ywp6}) with respect to $\hat{r}$ over $(r_1,r_2)$ and dividing by $r_2-r_1$, we arrive at (\ref{ywsi2}).
This completes the proof of Lemma \ref{ywsi}.
\end{proof}

Motivated by an idea from \cite{CZZ1}, we establish the following weighted estimates.
\begin{lemma}\la{hdi}
Let $\ep \in (0,1)$, $M \in (2,\infty)$, and let $g(r)$ be a function defined on $(\ep,M)$ such that $rg(r), rg'(r) \in L^{2}(\ep,M)$.
Then there exists a positive constant $C$ depending only on $M$ such that
\be\la{hdi1}\ba
\int_\ep^M |g(r)|^2 dr + \| r^{\frac{1}{2}} g(r) \|^2_{L^\infty(\ep,M)}
\le C \int_\ep^M \left( r^2 |g(r)|^2 + r^2 |g'(r)|^2 \right) dr.
\ea\ee
\end{lemma}
\begin{proof}
First, integration by parts together with Young's inequality leads to
\be\nonumber\ba
\int_\ep^M |g(r)-g(M)|^2 dr & = -2 \int_\ep^M (r-\ep) (g(r)-g(M)) g'(r) dr \\
& \le \frac{1}{2} \int_\ep^M |g(r)-g(M)|^2 dr
+ C \int_\ep^M r^2 |g'(r)|^2 dr,
\ea\ee
which gives
\be\la{hd2}\ba
\int_\ep^M |g(r)-g(M)|^2 dr & \le C \int_\ep^M r^2 |g'(r)|^2 dr.
\ea\ee
Moreover, by (\ref{ywsi2}) and H\"older's inequality, we have
\be\nonumber\ba
|g(M)| & \le \frac{2}{M} \int_{\frac{M}{2}}^M |g(r)| dr + \int_{\frac{M}{2}}^M |g'(r)| dr \\
& \le \frac{4}{M^2} \int_{\frac{M}{2}}^M r |g(r)| dr + \frac{2}{M} \int_{\frac{M}{2}}^M r |g'(r)| dr \\
& \le C \int_\ep^M \left( r |g(r)| + r |g'(r)| \right) dr \\
& \le C \left( \int_\ep^M \left( r^2 |g(r)|^2 + r^2 |g'(r)|^2 \right) dr \right)^{\frac{1}{2}},
\ea\ee
which together with (\ref{hd2}) yields
\be\la{hd3}\ba
\int_\ep^M |g(r)|^2 dr & \le C \int_\ep^M \left( r^2 |g(r)|^2 + r^2 |g'(r)|^2 \right) dr.
\ea\ee
It follows from (\ref{ywsi2}), (\ref{hd3}), and Young's inequality that, for any $\ep \le r \le M$,
\be\la{hd1}\ba
r |g(r)|^2
& \le \frac{1}{M-\ep} \int_\ep^M r |g(r)|^2 dr
+ \int_\ep^M \left( |g(r)|^2 + 2r |g(r)| |g'(r)| \right) dr \\
& \le C \int_\ep^M \left( |g(r)|^2 + r^2 |g'(r)|^2 \right) dr \\
& \le C \int_\ep^M \left( r^2 |g(r)|^2 + r^2 |g'(r)|^2 \right) dr.
\ea\ee
Combining (\ref{hd1}) and (\ref{hd3}), we obtain (\ref{hdi1}) and complete the proof.
\end{proof}

The following lemma plays an important role in the proof of vacuum vanishing.
\begin{lemma}\la{zkxsyl}
Let $f(t) \in L^1(0,\infty) \cap W^{1,\infty}(0,\infty)$.
Then
\be\la{zy01}\ba
|f(t)| \to 0 \ \textnormal{ as } t \to \infty.
\ea\ee
\end{lemma}
\begin{proof}
First, since $f(t) \in W^{1,\infty}(0,\infty)$, $f(t)$ is Lipschitz and for any $t_1,t_2>0$,
\be\la{zy1}\ba
|f(t_1) - f(t_2)| \le \| f'(t) \|_{L^\infty(0,\infty)} |t_1-t_2|.
\ea\ee
Set
\be\la{zy2}\ba
M \triangleq \| f'(t) \|_{L^\infty(0,\infty)}.
\ea\ee
If $M=0$, then $f(t)$ is constant on $(0,\infty)$.
Since $f(t) \in L^1(0,\infty)$, this constant must be zero, and the conclusion follows.

Suppose now that $M>0$.
We argue by contradiction. Assume that
\be\nonumber\ba
|f(t)| \not\to 0 \quad \text{ as } t \to \infty.
\ea\ee
Then there exist $\ep_0>0$ and a sequence $\{t_n\}_{n \ge 1}$ with $t_n \to \infty$ such that
\be\la{zy3}\ba
|f(t_n)| \ge \ep_0 \quad \text{ for all } n \ge 1.
\ea\ee
Let
\be\la{zy4}\ba
\delta \triangleq \frac{\ep_0}{2M}.
\ea\ee
From (\ref{zy1}), (\ref{zy2}), and (\ref{zy3}), we deduce that for any $t \in (t_n-\delta,t_n+\delta)$,
\be\la{zy5}\ba
|f(t)| \ge |f(t_n)| - |f(t)-f(t_n)| \ge \ep_0 - M |t-t_n| \ge \frac{\ep_0}{2}.
\ea\ee
Since $t_n \to \infty$, after passing to a subsequence, if necessary, we may assume that the intervals $(t_n-\delta,t_n+\delta)$ are pairwise disjoint and contained in $(0,\infty)$.

By (\ref{zy5}), we have
\be\la{zy6}\ba
\int_0^\infty |f(t)| dt \ge \sum_{n=1}^\infty \int_{t_n-\delta}^{t_n+\delta} |f(t)| dt
\ge \sum_{n=1}^\infty \ep_0 \delta = + \infty,
\ea\ee
which contradicts the assumption that $f(t) \in L^1(0,\infty)$.
Thus,
\be\nonumber\ba
|f(t)| \to 0 \ \textnormal{ as } t \to \infty,
\ea\ee
which completes the proof of Lemma \ref{zkxsyl}.
\end{proof}

Next, we prove a Gr\"onwall-type inequality.
\begin{lemma}\la{gnti}
Let $f \in L^\infty(0,T)$ and $g \in L^1(0,T)$ be nonnegative functions on $[0,T]$ satisfying, for a.e. $t \in [0,T]$,
\be\la{gnt01}\ba
f(t) \le \int_0^t f(s)^{1+b} g(s) ds + A_0,
\ea\ee
for some constants $b \ge 0$ and $A_0>0$.
Assume that
\be\la{gnt02}\ba
\int_0^T g(s) ds \le B_0.
\ea\ee
If
\be\la{gnt03}\ba
b A_0^b B_0 < 1,
\ea\ee
then for a.e. $t \in [0,T]$,
\be\la{gnt04}\ba
f(t) \le
\begin{cases}
A_0 e^{B_0} \quad & \textnormal{ if } b=0, \\
A_0 \left( 1 - b A_0^b B_0 \right)^{-\frac{1}{b}} \quad & \textnormal{ if } b>0.
\end{cases}
\ea\ee
\end{lemma}
\begin{proof}
First, set
\be\la{gnt11}\ba
h(t) \triangleq \int_0^t f(s)^{1+b} g(s) ds + A_0,
\ea\ee
which together with (\ref{gnt01}) implies
\be\la{gnt12}\ba
h'(t) = f(t)^{1+b} g(t) \le h(t)^{1+b} g(t).
\ea\ee
If $b=0$, then
\be\la{gntb01}\ba
\frac{d}{dt} \left( h(t) e^{-\int_0^t g(s) ds} \right) \le 0.
\ea\ee
Integrating (\ref{gntb01}) over $(0,t)$ and applying (\ref{gnt11}), we obtain
\be\la{gntb02}\ba
h(t) e^{-\int_0^t g(s) ds} \le h(0) = A_0,
\ea\ee
which together with (\ref{gnt02}) and (\ref{gnt11}) gives
\be\la{gntb03}\ba
f(t) \le h(t) \le A_0 e^{\int_0^t g(s) ds} \le A_0 e^{B_0}.
\ea\ee
If $b>0$, since $h \ge A_0 > 0$, we deduce from (\ref{gnt12}) that
\be\la{gnt13}\ba
\frac{d}{dt} (h^{-b}(t)) = - b h^{-(1+b)}(t) h'(t) \ge - b g(t).
\ea\ee
Integrating (\ref{gnt13}) over $(0,t)$ and using (\ref{gnt11}), we arrive at
\be\la{gnt14}\ba
h(t)^{-b} \ge A_0^{-b} - b \int_0^t g(s) ds = A_0^{-b} \left( 1 - b A_0^b \int_0^t g(s) ds \right).
\ea\ee
It follows from (\ref{gnt02}) and (\ref{gnt03}) that
\be\la{gnt15}\ba
1 - b A_0^b \int_0^t g(s) ds \ge 1 - b A_0^b B_0 > 0,
\ea\ee
which together with (\ref{gnt01}), (\ref{gnt11}), and (\ref{gnt14}) yields
\be\la{gnt16}\ba
f(t) \le h(t) \le A_0 \left( 1 - b A_0^b B_0 \right)^{-\frac{1}{b}}.
\ea\ee

Combining (\ref{gntb03}) and (\ref{gnt16}), we obtain (\ref{gnt04}), thereby completing the proof of Lemma \ref{gnti}.
\end{proof}

Finally, we recall the following ``inversion'' operator for the divergence on bounded domains (see \cite{CL} for details).
\begin{lemma}\la{iod}
Let $\OM$ be a bounded smooth domain.
For $1<p<\infty$, there exists a bounded linear operator $\mathcal{B}$ given by
\be\ba\nonumber
\mathcal{B}:\left\{f \in L^p(\OM) :  \int_\Omega fdx=0 \right\}&\rightarrow (W^{1,p}_0(\OM))^N,
\ea\ee
such that $v=\mathcal{B}(f)$ solves
\be\la{iod1}\ba
\begin{cases}
\mathrm{div}v=f&\ \textnormal{ in }\Omega,\\
v=0&\ \textnormal{ on }\partial\Omega.
\end{cases}
\ea\ee
Moreover, the operator satisfies the following properties:

(1) For $1<p<\infty$, there exists a positive constant $C$ depending on $\Omega$ and $p$ such that
\bnn
\|\mathcal{B}(f)\|_{W^{1,p}(\om)}\leq C(p)\|f\|_{L^p(\om)}.
\enn

(2) If $f=\mathrm{div}h$, for some $h\in L^p$ with $h\cdot n=0$ on $\partial\Omega$, then $v=\mathcal{B}(f)$ is a weak solution to the problem \eqref{iod1} and satisfies
\bnn
\|\mathcal{B}(f)\|_{L^p(\om)}\leq C(p) \| h \|_{L^p(\om)}.
\enn
\end{lemma}

\section{Approximate system and basic a priori estimates}

The proof of the global existence of weak solutions relies on the construction of a family of global smooth approximate solutions.
A central issue is to derive upper and lower bounds for the approximate density.
To this end, for any $0<\ep<1$, we consider the approximate system:
\be\ba\la{bjns}
\begin{cases}
(\rho_\ep)_t + \div(\rho_\ep \mathbf{u}_\ep) = 0, \\
(\n_\ep \mathbf{u}_\ep)_t + \div(\n_\ep \mathbf{u}_\ep \otimes \mathbf{u}_\ep) - \div( \mu_\ep(\n_\ep) \mathbb{D}\mathbf{u}_\ep) - \na( \lambda_\ep(\n_\ep) \div \mathbf{u}_\ep ) + \na \n_\ep^\ga = 0,
\end{cases}
\ea\ee
in the exterior domain
\be\la{bjqy}\ba
\om_\ep \triangleq \{ x \in \mathbb{R}^N \mid \ep<|x|<\infty \}.
\ea\ee
The viscosity coefficients are given by
\be\la{bjnxxs}\ba
\mu_\ep(\n_\ep) = \n_\ep^\alpha + \ep \n_\ep^\theta, \quad \lam_\ep(\n_\ep) = (\alpha-1) \n_\ep^\alpha + \ep (\theta-1) \n_\ep^\theta,
\ea\ee
where $\theta$ is defined by
\be\la{rgylt1}\ba
\theta = 
\begin{cases}
\alpha & \text{ if } \alpha \le 1, \\
\frac{\alpha}{2\alpha-1} & \text{ if } N=2 \text{ and } \alpha > 1, \\
\frac{2\alpha-1}{3\alpha-2} & \text{ if } N=3 \text{ and } 1 < \alpha \le 2, \\
\frac{3}{4} & \text{ if } N=3 \text{ and } \alpha>2.
\end{cases}
\ea\ee
It follows from (\ref{cs1})--(\ref{cs4}) and the definition of $\theta$ that
\be\la{rgylt2}\ba
\varphi(\alpha) = \varphi(\theta) & \quad \text{ if } N=2, \\
\psi(\alpha) = \psi(\theta) & \quad \text{ if } N=3 \text{ and } \alpha \le 2, \\
\psi(\alpha) < \psi(\theta) & \quad \text{ if } N=3 \text{ and } \alpha>2.
\ea\ee
The approximate system is supplemented with the spherically symmetric initial data
\be\la{bjcstj}\ba
\n_\ep(x,0) = \n_{0,\ep}(x)= \n_{0,\ep}(r), \quad \n_\ep \mathbf{u}_\ep(x,0) = \mathbf{m}_{0,\ep}(x) = m_{0,\ep}(r) \frac{x}{r}, \quad x\in \OM_\ep,
\ea\ee
the Dirichlet boundary condition
\be\la{bjbjtj1}\ba
\mathbf{u}_\ep = 0 \ \text{ on } \ \p \om_\ep,
\ea\ee
and the far-field condition
\be\la{bjbjtj2}\ba
(\n_\ep,\mathbf{u}_\ep)(x,t) \to (1,0) \ \text{ as } |x| \to \infty, \  t>0.
\ea\ee
Let the initial data $(\n_0,\mathbf{m}_0)$ of the original system (\ref{ns}) satisfy (\ref{czre}) and
\be\la{bjcz1}\ba
\n_0 - 1 \in L^2(\mathbb{R}^N), \quad \na \n_0^{\alpha-1+\frac{1}{s_1}} \in L^{s_1}(\mathbb{R}^N), \quad \n_0^{-p_1+1}|\mathbf{m}_0|^{p_1} \in L^1(\mathbb{R}^N),
\ea\ee
for some $N < s_1 \le p_1 < \mathcal{P}_N(\alpha)$ with $\alpha-1+\frac{1}{s_1}>0$.
We choose spherically symmetric approximate initial data $(\n_{0,\ep},\mathbf{m}_{0,\ep}) \in C^\infty(\mathbb{R}^N)$ such that, as $\ep \to 0^+$,
\be\la{bjcz2}\ba
\begin{cases}
\n_{0,\ep} - 1 \to \n_0 - 1 \textnormal{ in } L^2(\mathbb{R}^N) \cap L^\infty(\mathbb{R}^N), \\
\na \n_{0,\ep}^{\alpha-\frac{1}{2}} \to \na \n_0^{\alpha-\frac{1}{2}} \textnormal{ in } L^2(\mathbb{R}^N), \quad
\na \n_{0,\ep}^{\alpha-1+\frac{1}{s_1}} \to \na \n_{0}^{\alpha-1+\frac{1}{s_1}} \textnormal{ in } L^{s_1}(\mathbb{R}^N), \\
\mathbf{m}_{0,\ep} \to \mathbf{m}_{0} \textnormal{ in } L^2(\mathbb{R}^N), \quad
\n_{0,\ep}^{-1}|\mathbf{m}_{0,\ep}|^2 \to \n_{0}^{-1}|\mathbf{m}_{0}|^2 \textnormal{ in } L^1(\mathbb{R}^N), \\
\n_{0,\ep}^{-p_1+1}|\mathbf{m}_{0,\ep}|^{p_1} \to \n_0^{-p_1+1}|\mathbf{m}_0|^{p_1} \textnormal{ in } L^1(\mathbb{R}^N).
\end{cases}
\ea\ee
Moreover, setting
\be\la{bjcz3}\ba
\mathbf{u}_{0,\ep} \triangleq \n_{0,\ep}^{-1} \mathbf{m}_{0,\ep},
\ea\ee
we require that
\be\la{bjcz4}\ba
\mathbf{u}_{0,\ep} = 0 \text{ on } \p \om_{\ep}, \quad
(\n_{0,\ep},\mathbf{u}_{0,\ep})(x) \to (1,0) \ \text{ as } |x| \to \infty,
\ea\ee
and
\be\la{bjcz5}\ba
\n_{0,\ep} \ge
\begin{cases}
\ep \quad & \text{ if } \alpha \le 1, \\
\ep^{\frac{1}{\alpha-\theta}} \quad & \text{ if } \alpha>1.
\end{cases}
\ea\ee

The detailed construction of the smooth spherically symmetric approximate initial data $(\n_{0,\ep},\mathbf{m}_{0,\ep})$ is given in Section 9.

Under the spherical symmetry assumption
\be\nonumber\ba
\n_\ep(x,t) = \n_\ep(r,t), \quad \mathbf{u}_\ep(x,t) = u_\ep(r,t) \frac{x}{r}, \quad r = |x|,
\ea\ee
the approximate system reduces to
\be\la{bjnsqdc}\ba
\begin{cases}
(\n_\ep)_t + (\n_\ep u_\ep)_r + \frac{(N-1)}{r} \n_\ep u_\ep = 0, \\
\n_\ep( (u_\ep)_t + u_\ep (u_\ep)_r) + (\n_\ep^\ga)_r \\
\qquad - \left( r^{-(N-1)} (\alpha \n_\ep^\alpha + \ep \theta \n_\ep^\theta) (r^{N-1}u_\ep)_r \right)_r
+ \frac{N-1}{r}(\n_\ep^\alpha+\ep \n_\ep^\theta)_r u_\ep = 0, \\
u_\ep(\ep,t) = 0, \\
(\n_\ep,u_\ep)(r,t) \to (1,0) \text{ as } r \to \infty.
\end{cases}
\ea\ee

We define the effective velocity associated with the approximate system by
\be\la{bjyxsd}\ba
w_\ep \triangleq u_\ep + \n_\ep^{-1} (\n_\ep^\alpha + \ep \n_\ep^\theta)_r.
\ea\ee
It follows from (\ref{bjnsqdc}) that $w_\ep$ satisfies
\be\la{bjyxsdfc1}\ba
\n_\ep (w_\ep)_t + \n_\ep u_\ep (w_\ep)_r + (\n_\ep^\ga)_r = 0.
\ea\ee

By adapting arguments similar to those in \cite{CZZ1,CZZ2,JR,LPZ,ZZ}, we can obtain that, for each fixed $\ep>0$, the problem (\ref{bjns}), (\ref{bjqy}), (\ref{bjnxxs}), (\ref{bjbjtj1}), (\ref{bjbjtj2}) with smooth spherically symmetric initial data $(\n_{0,\ep},\mathbf{m}_{0,\ep})$ satisfying (\ref{bjcz3})--(\ref{bjcz5}) has a unique local spherically symmetric classical solution $(\n_\ep,u_\ep)$ on $\om_\ep \times [0,T]$.

Moreover, there exists a positive constant $c_\ep$ such that
\be\nonumber\ba
\n_\ep(r,t) \ge c_\ep \quad \text{ for all } (r,t) \in [\ep,\infty) \times [0,T].
\ea\ee

For such a solution, we introduce the Lagrangian coordinates
\be\la{lzb1}
y(r,t) = \int_\ep^r \rho_\ep(s,t) s^{N-1} \, ds, \quad \tau(r,t) = t, \quad r \in (\ep,\infty).
\ee
Since $\n_\ep$ is strictly positive, the mapping $r \mapsto y(r,t)$ is strictly increasing and hence invertible.
Using the continuity equation $(\ref{bjnsqdc})_1$ and the boundary condition $u_\ep(\ep,t)=0$, we obtain
\be\la{lzb2}
\frac{\partial y}{\partial r} = \rho_\ep r^{N-1}, \quad \frac{\partial y}{\partial t} = -\rho_\ep u_\ep r^{N-1}, \quad \frac{\partial \tau}{\partial r} = 0, \quad \frac{\partial \tau}{\partial t} = 1, \quad \frac{\partial r}{\partial \tau} = u_\ep.
\ee

Next, we derive several basic a priori estimates for the local classical solution $(\n_\ep,u_\ep)$.

We begin with the standard energy estimate.
\begin{lemma}\la{bjl1}
There exists a positive constant $C$ independent of $\ep$ and $T$ such that
\be\la{bj01}\ba
& \sup_{0\le t \le T} \int_\ep^\infty \left( \n_\ep |u_\ep|^2 + K(\n_\ep) \right) r^{N-1} dr
+ \int_0^T \int_\ep^\infty \left( \n_\ep^{\alpha} \frac{|u_\ep|^2}{r^2} + \n_\ep^{\alpha} |\p_r u_\ep|^2 \right) r^{N-1} dr dt \\
& \quad + \ep \int_0^T \int_\ep^\infty \left( \n_\ep^{\theta} \frac{|u_\ep|^2}{r^2}
+ \n_\ep^{\theta} |\p_r u_\ep|^2 \right) r^{N-1} dr dt
\le C.
\ea\ee
\end{lemma}
\begin{proof}
First, multiplying $(\ref{bjnsqdc})_2$ by $r^{N-1} u_\ep$ and integrating by parts over $(\ep,\infty)$, we derive
\be\la{bj11}\ba
& \frac{1}{2} \frac{d}{dt} \int_\ep^\infty \n_\ep |u_\ep|^2 r^{N-1} dr
+ \left(\alpha(N-1)^2-(N-1)(N-2)\right) \int_\ep^\infty \n_\ep^{\alpha} |u_\ep|^2 r^{N-3} dr \\
& \quad + \alpha \int_\ep^\infty \n_\ep^{\alpha} |\p_r u_\ep|^2 r^{N-1} dr
+ \ep \theta \int_\ep^\infty \n_\ep^{\theta} |\p_r u_\ep|^2 r^{N-1} dr \\
& \quad + \ep \left(\theta(N-1)^2-(N-1)(N-2)\right) \int_\ep^\infty \n_\ep^{\theta} |u_\ep|^2 r^{N-3} dr \\
& = 2(N-1)(1-\alpha) \int_\ep^\infty \n_\ep^\alpha u_\ep (\p_r u_\ep) r^{N-2} dr
- \int_\ep^\infty \p_r(\n_\ep^\ga) u_\ep r^{N-1} dr \\
& \quad + 2\ep (N-1)(1-\theta) \int_\ep^\infty \n_\ep^\theta u_\ep (\p_r u_\ep) r^{N-2} dr.
\ea\ee
Since $\alpha,\theta>\frac{N-1}{N}$, we have
\be\la{bj12}\ba
\frac{(N-1)^2(1-\alpha)^2}{\alpha} < \alpha(N-1)^2-(N-1)(N-2),
\ea\ee
and
\be\la{bj12a}\ba
\frac{(N-1)^2(1-\theta)^2}{\theta} < \theta(N-1)^2-(N-1)(N-2).
\ea\ee
Thus, there exists a sufficiently small $0<\delta<\min\{\alpha,\theta\}$ such that
\be\la{bj13}\ba
\frac{(N-1)^2(1-\alpha)^2}{\alpha-\delta} \le \alpha(N-1)^2-(N-1)(N-2)-\delta,
\ea\ee
and
\be\la{bj13a}\ba
\frac{(N-1)^2(1-\theta)^2}{\theta-\delta} \le \theta(N-1)^2-(N-1)(N-2)-\delta.
\ea\ee
For such $\delta$, using Young's inequality, we obtain
\be\la{bj14}\ba
& 2(N-1)|1-\alpha| \n_\ep^{\alpha} |u_\ep| |\p_r u_\ep| r^{N-2} \\
& \le \frac{(N-1)^2(1-\alpha)^2}{\alpha-\delta} \n_\ep^{\alpha} |u_\ep|^2 r^{N-3} + (\alpha-\delta) \n_\ep^{\alpha} |\p_r u_\ep|^2 r^{N-1} \\
& \le \left( \alpha(N-1)^2-(N-1)(N-2)-\delta \right) \n_\ep^{\alpha} |u_\ep|^2 r^{N-3} + (\alpha-\delta) \n_\ep^{\alpha} |\p_r u_\ep|^2 r^{N-1}.
\ea\ee
Similarly,
\be\la{bj14a}\ba
& 2(N-1)|1-\theta| \n_\ep^{\theta} |u_\ep| |\p_r u_\ep| r^{N-2} \\
& \le \left( \theta(N-1)^2-(N-1)(N-2)-\delta \right) \n_\ep^{\theta} |u_\ep|^2 r^{N-3} + (\theta-\delta) \n_\ep^{\theta} |\p_r u_\ep|^2 r^{N-1}.
\ea\ee
Note that $K(\n_\ep)$ satisfies
\be\la{bj15a}\ba
K'(\n_\ep) \n_\ep - K(\n_\ep) = \n_\ep^\ga - 1.
\ea\ee
Multiplying $(\ref{bjnsqdc})_1$ by $K'(\n_\ep) r^{N-1}$, integrating by parts over $(\ep,\infty)$, and using (\ref{bj15a}), we arrive at
\be\la{bj15}\ba
\frac{d}{dt} \int_\ep^\infty K(\n_\ep) r^{N-1} dr - \int_\ep^\infty \p_r(\n_\ep^\ga) u_\ep r^{N-1} dr = 0.
\ea\ee
Combining (\ref{bj11}), (\ref{bj14}), (\ref{bj14a}), and (\ref{bj15}) and using the fact that $\alpha, \theta>\frac{N-1}{N}$, we obtain
\be\la{bj16}\ba
& \frac{d}{dt} \int_\ep^\infty \left( \frac{1}{2} \n_\ep |u_\ep|^2 + K(\n_\ep) \right) r^{N-1} dr
+ \delta \int_\ep^\infty \n_\ep^{\alpha} |u_\ep|^2 r^{N-3} dr
+ \delta \int_\ep^\infty \n_\ep^{\alpha} |\p_r u_\ep|^2 r^{N-1} dr \\
& + \ep \delta \int_\ep^\infty \n_\ep^{\theta} |u_\ep|^2 r^{N-3} dr
+ \ep \delta\int_\ep^\infty \n_\ep^{\theta} |\p_r u_\ep|^2 r^{N-1} dr
\le 0.
\ea\ee
Integrating (\ref{bj16}) over $(0,T)$ gives (\ref{bj01}) and completes the proof of Lemma \ref{bjl1}.
\end{proof}

Next, we derive the BD entropy estimate.
\begin{lemma}\la{bjl2}
There exists a positive constant $C$ independent of $\ep$ and $T$ such that
\be\la{bj02}\ba
& \sup_{0\le t \le T} \int_\ep^\infty \left( |\p_r \n_\ep^{\alpha-\frac{1}{2}}|^2 + \ep^2 |\p_r \n_\ep^{\theta-\frac{1}{2}}|^2 \right) r^{N-1} dr \\
& \quad + \int_0^T \int_\ep^\infty \left( |\p_r \n_\ep^{\frac{\ga+\alpha-1}{2}}|^{2}
+ \ep |\p_r \n_\ep^{\frac{\ga+\theta-1}{2}}|^{2} \right) r^{N-1} dr dt
\le C.
\ea\ee
\end{lemma}
\begin{proof}
Multiplying $(\ref{bjyxsdfc1})$ by $w_\ep r^{N-1}$, integrating by parts over $(\ep,\infty)$, and using (\ref{bjnsqdc}) and (\ref{bj15}), we obtain
\be\la{bj21}\ba
& \frac{1}{2} \frac{d}{dt} \int_\ep^\infty \n_\ep |w_\ep|^2 r^{N-1} dr \\
& = - \int_\ep^\infty \left( \p_r(\n_\ep^\ga) u_\ep + \p_r(\n_\ep^\ga) \n_\ep^{-1} \p_r(\n_\ep^\alpha + \ep \n_\ep^\theta) \right) r^{N-1} dr \\
& = - \frac{d}{dt} \int_\ep^\infty K(\n_\ep) r^{N-1} dr
- \int_\ep^\infty \left( \frac{\ga}{\alpha} \n_\ep^{\ga-\alpha-1} |\p_r \n_\ep^\alpha|^2
+ \frac{\ep \ga}{\theta} \n_\ep^{\ga-\theta-1} |\p_r \n_\ep^\theta|^2 \right) r^{N-1} dr.
\ea\ee
Integrating (\ref{bj21}) over $(0,T)$ gives
\be\la{bj22}\ba
& \sup_{0 \le t \le T} \int_\ep^\infty \left( \n_\ep |w_\ep|^2 + K(\n_\ep) \right) r^{N-1} dr \\
& \quad + \int_0^T \int_\ep^\infty \left( |\p_r \n_\ep^{\frac{\ga+\alpha-1}{2}}|^{2}
+ \ep |\p_r \n_\ep^{\frac{\ga+\theta-1}{2}}|^{2} \right) r^{N-1} dr dt \le C.
\ea\ee
By (\ref{bjyxsd}), we have
\be\la{bj23}\ba
\left(\alpha \n_\ep^{\alpha-\frac{3}{2}} + \ep \theta \n_\ep^{\theta-\frac{3}{2}} \right)^2 |\p_r \n_\ep|^2
= \n_\ep^{-1} |\p_r ( \n_\ep^{\alpha} + \ep \n_\ep^{\theta} )|^2
= \n_\ep|w_\ep - u_\ep|^2.
\ea\ee
It follows from (\ref{bj01}), (\ref{bj22}), (\ref{bj23}), and Cauchy's inequality that
\be\la{bj24}\ba
\int_\ep^\infty \left( |\p_r \n_\ep^{\alpha-\frac{1}{2}}|^2 + \ep^2 |\p_r \n_\ep^{\theta-\frac{1}{2}}|^2 \right) r^{N-1} dr
\le C \int_\ep^\infty \n_\ep \left( |u_\ep|^2 + |w_\ep|^2 \right) r^{N-1} dr \le C.
\ea\ee
Combining (\ref{bj22}) and (\ref{bj24}) yields (\ref{bj02}) and completes the proof of Lemma \ref{bjl2}.
\end{proof}

\begin{lemma}\la{cptl1}
If $N=2$, for any $p \in [2,\infty)$, there exists a positive constant $C$ independent of $\ep$ and $T$ such that
\be\la{cpt01}\ba
\sup_{0\le t \le T} \int_\ep^\infty |\n_\ep^{\alpha-\frac{1}{2}}-1|^p r dr \le C(p).
\ea\ee
If $N=3$, there exists a positive constant $C$ independent of $\ep$ and $T$ such that
\be\la{cpt001}\ba
\sup_{0\le t \le T} \int_\ep^\infty \left( |\n_\ep^{\alpha-\frac{1}{2}}-1|^2 + |\n_\ep^{\alpha-\frac{1}{2}}-1|^6 \right) r^2 dr \le C.
\ea\ee
\end{lemma}
\begin{proof}
First, from the definition \eqref{pe} of $K$, we conclude that there is a positive constant $C$ depending only on $\ga$ such that
\be\la{cpt11}\ba
(\n_\ep - 1)^2 \le C K(\n_\ep), \text{ if } \n_\ep < 2, \quad
(\n_\ep - 1)^\ga \le C K(\n_\ep), \text{ if } \n_\ep \ge 2,
\ea\ee
which together with (\ref{bj01}) gives
\be\la{cpt12}\ba
\int_\ep^\infty \chi_{(\n_\ep<2)} |\n_\ep - 1|^2 r^{N-1} dr
+ \int_\ep^\infty \chi_{(\n_\ep \ge 2)} |\n_\ep - 1|^\ga r^{N-1} dr \le C.
\ea\ee
By (\ref{cpt12}), we have
\be\la{cpt12a}\ba
\int_\ep^\infty \chi_{(\n_\ep \ge 2)} r^{N-1} dr
\le C \int_\ep^\infty \chi_{(\n_\ep \ge 2)} |\n_\ep - 1|^\ga r^{N-1} dr \le C.
\ea\ee
Combining (\ref{bj02}), (\ref{cpt12a}), and H\"older's inequality gives
\be\la{cpt12b}\ba
\int_\ep^\infty | \p_r ( \n_\ep^{\alpha-\frac{1}{2}} - 2^{\alpha-\frac{1}{2}} )_+ | r^{N-1} dr
& \le C \int_\ep^\infty \chi_{(\n_\ep \ge 2)} | \p_r \n_\ep^{\alpha-\frac{1}{2}} | r^{N-1} dr \\
& \le C \left( \int_\ep^\infty \chi_{(\n_\ep \ge 2)} r^{N-1} dr \right)^{\frac{1}{2}} \left( \int_\ep^\infty | \p_r \n_\ep^{\alpha-\frac{1}{2}} |^2 r^{N-1} dr \right)^{\frac{1}{2}} \\
& \le C.
\ea\ee
We next extend $\n_\ep(r)$ continuously to $(0,\infty)$ by
\be\la{cpt13}\ba
\tilde{\n}_\ep(r,t) \triangleq
\begin{cases}
\n_\ep(r,t) \quad & \text{ if } \ep < r < \infty, \\
\n_\ep(2\ep-r,t) \quad & \text{ if } 0 < r \le \ep.
\end{cases}
\ea\ee
If $N=2$, the Gagliardo-Nirenberg inequality together with (\ref{bj02}) and (\ref{cpt12b}) yields that for any $2 \le p < \infty$,
\be\la{cpt14}\ba
& \left( \int_\ep^\infty ( \n_\ep^{\alpha-\frac{1}{2}} - 2^{\alpha-\frac{1}{2}} )^{p}_+ r dr \right)^{\frac{2}{p}} \\
& \le \left( \int_0^\infty ( \tilde{\n}_\ep^{\alpha-\frac{1}{2}} - 2^{\alpha-\frac{1}{2}} )^{p}_+ r dr \right)^{\frac{2}{p}} \\
& \le C \int_0^\infty ( \tilde{\n}_\ep^{\alpha-\frac{1}{2}} - 2^{\alpha-\frac{1}{2}} )^2_+ r dr
+ C \int_0^\infty |\p_r ( \tilde{\n}_\ep^{\alpha-\frac{1}{2}} - 2^{\alpha-\frac{1}{2}} )_+|^2 r dr \\
& \le C \left( \int_0^\infty | \p_r ( \tilde{\n}_\ep^{\alpha-\frac{1}{2}} - 2^{\alpha-\frac{1}{2}} )_+ | r dr \right)^2
+ C \int_\ep^\infty |\p_r ( \n_\ep^{\alpha-\frac{1}{2}} - 2^{\alpha-\frac{1}{2}} )_+|^2 r dr \\
& \le C \left( \int_\ep^\infty | \p_r ( \n_\ep^{\alpha-\frac{1}{2}} - 2^{\alpha-\frac{1}{2}} )_+ | r dr \right)^2
+ C \int_\ep^\infty |\p_r \n_\ep^{\alpha-\frac{1}{2}}|^2 r dr \\
& \le C,
\ea\ee
where we have used
\be\nonumber\ba
\int_0^\ep |\p_r ( \tilde{\n}_\ep^{\alpha-\frac{1}{2}} - 2^{\alpha-\frac{1}{2}} )_+|^2 r dr
& = \int_0^\ep |\p_r ( \n_\ep^{\alpha-\frac{1}{2}}(2\ep-r) - 2^{\alpha-\frac{1}{2}} )_+|^2 r dr \\
& = \int_0^\ep \chi_{(\n_\ep(2\ep-r) \ge 2)} |\p_r \n_\ep^{\alpha-\frac{1}{2}}(2\ep-r)|^2 r dr \\
& = \int_\ep^{2\ep} \chi_{(\n_\ep(s) \ge 2)} |\p_s \n_\ep^{\alpha-\frac{1}{2}}(s)|^2 (2\ep-s) ds \\
& \le \int_\ep^{2\ep} \chi_{(\n_\ep(s) \ge 2)} |\p_s \n_\ep^{\alpha-\frac{1}{2}}(s)|^2 s ds \\
& = \int_\ep^{2\ep} |\p_r ( \n_\ep^{\alpha-\frac{1}{2}} - 2^{\alpha-\frac{1}{2}})_+|^2 r dr \\
& \le \int_\ep^{\infty} |\p_r ( \n_\ep^{\alpha-\frac{1}{2}} - 2^{\alpha-\frac{1}{2}})_+|^2 r dr,
\ea\ee
and
\be\nonumber\ba
\int_0^\ep |\p_r ( \tilde{\n}_\ep^{\alpha-\frac{1}{2}} - 2^{\alpha-\frac{1}{2}} )_+| r dr
\le \int_\ep^{\infty} |\p_r ( \n_\ep^{\alpha-\frac{1}{2}} - 2^{\alpha-\frac{1}{2}})_+| r dr.
\ea\ee
Using (\ref{cpt12a}) and (\ref{cpt14}), we arrive at
\be\la{cpt15}\ba
& \int_\ep^\infty \chi_{(\n_\ep \ge 2)} |\n_\ep^{\alpha-\frac{1}{2}}-1|^p r dr \\
& \le C \int_\ep^\infty \chi_{(\n_\ep \ge 2)} |\n_\ep^{\alpha-\frac{1}{2}} - 2^{\alpha-\frac{1}{2}}|^p r dr
+ C \int_\ep^\infty \chi_{(\n_\ep \ge 2)} |2^{\alpha-\frac{1}{2}} - 1|^p r dr \\
& \le C \int_\ep^\infty (\n_\ep^{\alpha-\frac{1}{2}} - 2^{\alpha-\frac{1}{2}})_+^p r dr
+ C \int_\ep^\infty \chi_{(\n_\ep \ge 2)} r dr \\
& \le C.
\ea\ee
Moreover, (\ref{cpt11}) implies that for any $2 \le p < \infty$,
\be\la{cpt16}\ba
\int_\ep^\infty \chi_{(\n_\ep < 2)} |\n_\ep^{\alpha-\frac{1}{2}}-1|^p r dr
& \le C \int_\ep^\infty \chi_{(\n_\ep < 2)} |\n_\ep^{\alpha-\frac{1}{2}}-1|^2 r dr \\
& \le C \int_\ep^\infty \chi_{(\n_\ep < 2)} |\n_\ep-1|^2 r dr
\le C,
\ea\ee
which together with (\ref{cpt15}) yields (\ref{cpt01}).

If $N=3$, by (\ref{bj02}), (\ref{cpt13}), and the Sobolev inequality, we have
\be\la{cpt17}\ba
\int_\ep^\infty |\n_\ep^{\alpha-\frac{1}{2}}-1|^6 r^2 dr
& \le \int_0^\infty |\tilde{\n}_\ep^{\alpha-\frac{1}{2}}-1|^6 r^2 dr
\le C \left( \int_0^\infty |\p_r \tilde{\n}_\ep^{\alpha-\frac{1}{2}}|^2 r^{2} dr \right)^3 \\
& \le C \left( \int_\ep^\infty |\p_r \n_\ep^{\alpha-\frac{1}{2}}|^2 r^{2} dr \right)^3
\le C,
\ea\ee
where we have used the following estimate:
\be\nonumber\ba
\int_0^\ep |\p_r \tilde{\n}_\ep^{\alpha-\frac{1}{2}}|^2 r^{2} dr
& = \int_0^\ep |\p_r \n_\ep^{\alpha-\frac{1}{2}}(2\ep-r)|^2 r^{2} dr \\
& = \int_\ep^{2 \ep} |\p_s \n_\ep^{\alpha-\frac{1}{2}}(s)|^2 (2\ep-s)^{2} ds \\
& \le \int_\ep^{2 \ep} |\p_s \n_\ep^{\alpha-\frac{1}{2}}(s)|^2 s^{2} ds \\
& \le \int_\ep^{\infty} |\p_r \n_\ep^{\alpha-\frac{1}{2}}|^2 r^{2} dr.
\ea\ee
Combining (\ref{cpt17}) with (\ref{cpt12}) leads to
\be\la{cpt18}\ba
\int_\ep^\infty |\n_\ep^{\alpha-\frac{1}{2}}-1|^2 r^2 dr
& = \int_\ep^\infty \chi_{(\n_\ep<2)} |\n_\ep^{\alpha-\frac{1}{2}}-1|^2 r^2 dr
+ \int_\ep^\infty \chi_{(\n_\ep \ge 2)} |\n_\ep^{\alpha-\frac{1}{2}}-1|^2 r^2 dr \\
& \le C \int_\ep^\infty \chi_{(\n_\ep<2)} |\n_\ep-1|^2 r^2 dr
+ C \int_\ep^\infty \chi_{(\n_\ep \ge 2)} |\n_\ep^{\alpha-\frac{1}{2}}-1|^6 r^2 dr \\
& \le C,
\ea\ee
which together with (\ref{cpt17}) yields (\ref{cpt001}) and completes the proof of Lemma \ref{cptl1}.
\end{proof}

\begin{lemma}\la{cptl2}
There exist positive constants $R_0 \ge 2$, $c_1$, and $c_2$ independent of $\ep$ and $T$ such that
\be\la{cpt02}\ba
0<c_1 \le \n_\ep(r,t) \le c_2 \quad \textnormal{ for all } (r,t) \in [R_0,\infty) \times (0,T).
\ea\ee
Furthermore, for any $2 \le p \le 6$, there exists a positive constant $C$ independent of $\ep$ and $T$ such that
\be\la{cpt02a}\ba
\left( \int_\ep^{R_0} \chi_{(\n_\ep \le \frac{c_1}{2})}(r,t) r^{N-1} dr \right)^{\frac{2}{p}}
\le C \int_\ep^{R_0} |\p_r \n_\ep^{\frac{\ga+\alpha-1}{2}}(r,t)|^{2} r^{N-1} dr.
\ea\ee
Moreover, if $N=2$, for any $\xi>0$, there exists a positive constant $C$ independent of $\ep$ and $T$ such that
\be\la{cpt002}\ba
\sup_{0 \le t \le T} \| \n_\ep r^{\xi} \|_{L^\infty(\ep,R_0)} \le C.
\ea\ee
If $N=3$, there exists a positive constant $C$ independent of $\ep$ and $T$ such that
\be\la{cpt0002}\ba
\sup_{0 \le t \le T} \| \n_\ep^{2\alpha-1} r \|_{L^\infty(\ep,R_0)} \le C,
\ea\ee
and for any $t \in [0,T]$,
\be\la{cpt012}\ba
\int_\ep^{R_0} \chi_{(\n_\ep > 2c_2)} \n_\ep^{\ga+\alpha-1}(r,t) dr
\le C \int_\ep^{R_0} |\p_r \n_\ep^{\frac{\ga+\alpha-1}{2}}(r,t)|^2 r^2 dr.
\ea\ee
\end{lemma}
\begin{proof}
First, it follows from (\ref{ywsi1}), (\ref{cpt01}), (\ref{cpt001}), and (\ref{bj02}) that for any $r \ge 1$,
\be\la{cpt21}\ba
(\n_\ep^{\alpha-\frac{1}{2}}-1)^2 r^{N-1}
& \le \int_1^\infty | \p_r (\n_\ep^{\alpha-\frac{1}{2}}-1)^2 | r^{N-1} dr
+ (N-1) \int_1^\infty | \n_\ep^{\alpha-\frac{1}{2}}-1 |^2 r^{N-2} dr \\
& \le 2 \int_1^\infty |\n_\ep^{\alpha-\frac{1}{2}}-1| |\p_r \n_\ep^{\alpha-\frac{1}{2}}| r^{N-1} dr
+ (N-1) \int_1^\infty | \n_\ep^{\alpha-\frac{1}{2}}-1 |^2 r^{N-1} dr \\
& \le C \int_1^\infty |\n_\ep^{\alpha-\frac{1}{2}}-1|^2 r^{N-1} dr
+ C \int_1^\infty |\p_r \n_\ep^{\alpha-\frac{1}{2}}|^2 r^{N-1} dr \\
& \le C_0,
\ea\ee
where $C_0$ is independent of $\ep$ and $T$.

Set
\be\la{cpt22}\ba
R_0 \triangleq 2 + \left( 4 C_0 \right)^{\frac{1}{N-1}}.
\ea\ee
Thus, from (\ref{cpt21}), we conclude that for any $r \ge R_0$,
\be\la{cpt23}\ba
(\n_\ep^{\alpha-\frac{1}{2}}-1)^2 \le C_0 r^{1-N} \le \frac{1}{4},
\ea\ee
which implies that there exist positive constants $c_1$ and $c_2$ independent of $\ep$ and $T$ such that
\be\la{cpt24}\ba
0<c_1 \le \n_\ep(r,t) \le c_2 \quad \textnormal{ for all } (r,t) \in [R_0,\infty) \times (0,T).
\ea\ee
We next prove (\ref{cpt02a}).
For $\tilde{\n}_\ep$ as in (\ref{cpt13}), we define
\be\la{xcpt21}\ba
H(r,t) \triangleq (c_1^{\frac{\ga+\alpha-1}{2}} - \tilde{\n}_\ep^{\frac{\ga+\alpha-1}{2}})_+.
\ea\ee
Since $\tilde{\n}_\ep(R_0,t) = \n_\ep(R_0,t) \ge c_1$, we have $H(R_0,t) = 0$.
Thus, the Sobolev-Poincar\'e inequality implies that for any $2 \le p \le 6$,
\be\la{xcpt22}\ba
\left( \int_0^{R_0} H^{p} r^{N-1} dr \right)^{\frac{2}{p}}
& = |\mathbb{S}^{N-1}|^{-\frac{2}{p}} \| H \|_{L^p(B_{R_0})}^2
\le C \| \na H \|_{L^2(B_{R_0})}^2 \\
& \le C \int_0^{R_0} |\p_r H|^{2} r^{N-1} dr
\le C \int_0^{R_0} |\p_r \tilde{\n}_\ep^{\frac{\ga+\alpha-1}{2}}|^{2} r^{N-1} dr \\
& \le C \int_\ep^{R_0} |\p_r \n_\ep^{\frac{\ga+\alpha-1}{2}}|^{2} r^{N-1} dr,
\ea\ee
where in the last inequality we have used the following estimate:
\be\nonumber\ba
\int_0^\ep |\p_r \tilde{\n}_\ep^{\frac{\ga+\alpha-1}{2}}|^{2} r^{N-1} dr
\le \int_\ep^{2 \ep} |\p_r \n_\ep^{\frac{\ga+\alpha-1}{2}}|^{2} r^{N-1} dr
\le \int_\ep^{R_0} |\p_r \n_\ep^{\frac{\ga+\alpha-1}{2}}|^{2} r^{N-1} dr.
\ea\ee
Noticing that
\be\la{xcpt23}\ba
\chi_{(\n_\ep \le \frac{c_1}{2})} \left( c_1^{\frac{\ga+\alpha-1}{2}} - \frac{c_1}{2}^{\frac{\ga+\alpha-1}{2}} \right)
\le \chi_{(\n_\ep \le \frac{c_1}{2})} (c_1^{\frac{\ga+\alpha-1}{2}} - \n_\ep^{\frac{\ga+\alpha-1}{2}})_+,
\ea\ee
we deduce from (\ref{cpt13}), (\ref{xcpt21}), and (\ref{xcpt22}) that for any $2 \le p \le 6$,
\be\la{xcpt24}\ba
\left( \int_\ep^{R_0} \chi_{(\n_\ep \le \frac{c_1}{2})} r^{N-1} dr \right)^{\frac{2}{p}}
& \le C \left( \int_\ep^{R_0} \chi_{(\n_\ep \le \frac{c_1}{2})} H^{p} r^{N-1} dr \right)^{\frac{2}{p}} \\
& \le C \left( \int_0^{R_0} H^{p} r^{N-1} dr \right)^{\frac{2}{p}} \\
& \le C \int_\ep^{R_0} |\p_r \n_\ep^{\frac{\ga+\alpha-1}{2}}|^{2} r^{N-1} dr,
\ea\ee
which gives (\ref{cpt02a}).

We next prove (\ref{cpt002}) and (\ref{cpt0002}).
If $N=2$, we deduce from (\ref{bj02}), (\ref{cpt01}), and H\"older's inequality that for any $0 < \zeta \le 2$,
\be\la{cpt25}\ba
& \| \n_\ep^{\alpha-\frac{1}{2}} r^{\zeta} \|_{L^\infty(\ep,R_0)} \\
& \le \frac{1}{R_0-\ep} \int_\ep^{R_0} \n_\ep^{\alpha-\frac{1}{2}} r^{\zeta} dr
+ C \int_\ep^{R_0} \left( |\p_r \n_\ep^{\alpha-\frac{1}{2}}| r^{\zeta} + \n_\ep^{\alpha-\frac{1}{2}} r^{\zeta-1} \right) dr \\
& \le C \int_\ep^{R_0} \n_\ep^{\alpha-\frac{1}{2}} r^{\zeta-1} dr
+ C \int_\ep^{R_0} |\p_r \n_\ep^{\alpha-\frac{1}{2}}| r^{\zeta} dr \\
& \le C \left( \int_\ep^{R_0} \n_\ep^{\frac{4\alpha-2}{\zeta}} r dr \right)^{\frac{\zeta}{4}}
\left( \int_\ep^{R_0} r^{-\frac{4-3\zeta}{4-\zeta}} dr \right)^{\frac{4-\zeta}{4}}
+ C \left( \int_\ep^{R_0} |\p_r \n_\ep^{\alpha-\frac{1}{2}}|^2 r dr \right)^{\frac{1}{2}}
\left( \int_\ep^{R_0} r^{2\zeta-1} dr \right)^{\frac{1}{2}} \\
& \le C + C \int_\ep^{R_0} \n_\ep^{\frac{4\alpha-2}{\zeta}} r dr \\
& \le C + C \int_\ep^{R_0} |\n_\ep^{\alpha-\frac{1}{2}}-1|^{\frac{4}{\zeta}} r dr \\
& \le C,
\ea\ee
due to $2 \le \frac{4}{\zeta}<\infty$.

In addition, for any $\zeta > 2$, by (\ref{cpt25}), we have
\be\la{cpt26}\ba
& \| \n_\ep^{\alpha-\frac{1}{2}} r^{\zeta} \|_{L^\infty(\ep,R_0)}
\le R_0^{\zeta-2} \| \n_\ep^{\alpha-\frac{1}{2}} r^{2} \|_{L^\infty(\ep,R_0)}
\le C.
\ea\ee
Hence, for any $\zeta>0$, it holds that
\be\la{cpt27}\ba
\sup_{0 \le t \le T} \| \n_\ep^{\alpha-\frac{1}{2}} r^{\zeta} \|_{L^\infty(\ep,R_0)} \le C.
\ea\ee
For any $\xi>0$, we choose $\zeta=(\alpha-\frac{1}{2}) \xi$ in (\ref{cpt27}) to obtain
\be\la{cpt28}\ba
\sup_{0 \le t \le T} \| \n_\ep r^{\xi} \|_{L^\infty(\ep,R_0)} \le C.
\ea\ee
If $N=3$, using (\ref{hdi1}), (\ref{bj02}), and (\ref{cpt001}), we arrive at
\be\la{cpt29}\ba
\left\| \n_\ep^{\alpha-\frac{1}{2}} r^{\frac{1}{2}} \right\|_{L^\infty(\ep,R_0)}^2
& \le C \int_\ep^{R_0} \n_\ep^{2\alpha-1} r^2 dr + C \int_\ep^{R_0} |\p_r \n_\ep^{\alpha-\frac{1}{2}}|^2 r^2 dr \\
& \le C \int_\ep^{R_0} |\n_\ep^{\alpha-\frac{1}{2}} - 1|^2 r^2 dr + C \\
& \le C,
\ea\ee
which yields (\ref{cpt0002}).

It remains to prove (\ref{cpt012}).
Define
\be\la{cpt210}\ba
\Pi(r,t) \triangleq \left( \tilde{\n}_\ep^{\frac{\ga+\alpha-1}{2}} - c_2^{\frac{\ga+\alpha-1}{2}} \right)_{+}.
\ea\ee
Since $\n_\ep(R_0,t) \le c_2$, we have $\Pi(R_0,t) = 0$.
Thus, integration by parts together with Young's inequality yields
\be\la{cpt2110}\ba
\int_\ep^{R_0} \Pi^2(r,t) dr & = (r \Pi^2(r,t))|_{\ep}^{R_0} - 2 \int_\ep^{R_0} r \Pi (\p_r \Pi) dr \\
& = - \ep \Pi^2(\ep,t) - 2 \int_\ep^{R_0} r \Pi (\p_r \Pi) dr \\
& \le \frac{1}{2} \int_\ep^{R_0} \Pi^2(r,t) dr
+ 2 \int_\ep^{R_0} |\p_r \n_\ep^{\frac{\ga+\alpha-1}{2}}(r,t)|^2 r^2 dr,
\ea\ee
which gives
\be\la{cpt211}\ba
\int_\ep^{R_0} \Pi^2(r,t) dr
\le 4 \int_\ep^{R_0} |\p_r \n_\ep^{\frac{\ga+\alpha-1}{2}}(r,t)|^2 r^2 dr.
\ea\ee
The definition of $\Pi(r,t)$ shows
\be\la{cpt212}\ba
\chi_{(\n_\ep > 2c_2)} (1-2^{-\frac{\ga+\alpha-1}{2}}) \n_\ep^{\frac{\ga+\alpha-1}{2}}(r,t)
\le \chi_{(\n_\ep > 2c_2)} \Pi(r,t),
\ea\ee
which along with (\ref{cpt211}) implies for any $t \in [0,T]$,
\be\la{cpt213}\ba
\int_\ep^{R_0} \chi_{(\n_\ep > 2c_2)} \n_\ep^{\ga+\alpha-1}(r,t) dr
\le C \int_\ep^{R_0} |\p_r \n_\ep^{\frac{\ga+\alpha-1}{2}}(r,t)|^2 r^2 dr.
\ea\ee
This gives (\ref{cpt012}) and completes the proof of Lemma \ref{cptl2}.
\end{proof}

\section{Global time-uniform upper bound for the density}

In this section, we derive a priori estimates that are uniform with respect to both time and the approximation parameter $\ep$.
These estimates play a crucial role in proving the finite-time vanishing of vacuum.

\begin{lemma}\la{tl4}
Let $s_1$ and $p_1$ satisfy
\be\nonumber\ba
N < s_1 \le p_1 < \mathcal{P}_N(\alpha), \quad \alpha-1+\frac{1}{s_1}>0,
\ea\ee
as in Theorem \ref{th1}.
Choose $s_2>N$ sufficiently close to $N$ such that
\be\la{t04a}\ba
N < s_2 \le \min\left\{s_1,4\right\}, \quad \frac{1}{s_2} \ge \frac{\alpha+1}{2} - \ga.
\ea\ee
There exists a positive constant $\eta_0 \in (0,1]$ depending only on $s_1$, $\alpha$, $\mu$, and the initial data such that if $\alpha$ and $\ga$ satisfy \eqref{th110} and \eqref{th11}, then there exists a positive constant $C$ independent of $\ep$ and $T$ such that
\be\la{t04}\ba
& \sup_{0\le t \le T} \left( \int_\ep^\infty \n_\ep |u_\ep|^{s_2} r^{N-1} dr
+ \int_\ep^\infty |\p_r \n_\ep^{\alpha-1+\frac{1}{s_2}}|^{s_2} r^{N-1} dr \right) \\
& \quad + \int_0^T \int_\ep^\infty \left( \n_\ep^\alpha \frac{|u_\ep|^{s_2}}{r^2} + \n_\ep^\alpha |u_\ep|^{s_2-2} |\p_r u_\ep|^2 \right) r^{N-1} dr dt \\
& \quad + \ep \int_0^T \int_\ep^\infty \left( \n_\ep^{\theta} \frac{|u_\ep|^{s_2}}{r^2} + \n_\ep^{\theta} |u_\ep|^{s_2-2} |\p_r u_\ep|^2 \right) r^{N-1} dr dt
\le C.
\ea\ee
\end{lemma}
\begin{proof}
First, multiplying $(\ref{bjnsqdc})_2$ by $u_\ep (1+|u_\ep|^2)^{\frac{s_2-2}{2}} r^{N-1}$ and integrating by parts over $(\ep,\infty)$, we derive
\be\la{t41}\ba
& \frac{1}{s_2} \frac{d}{dt} \int_\ep^\infty \n_\ep \left( (1+|u_\ep|^2)^{\frac{s_2}{2}}-1 \right) r^{N-1} dr
+ \int_\ep^\infty \alpha (s_{u_\ep}-1) \n_\ep^\alpha (1+|u_\ep|^2)^{\frac{s_2-2}{2}} |\p_r u_\ep|^2 r^{N-1} dr \\
& \quad + \left(\alpha(N-1)^2-(N-1)(N-2)\right) \int_\ep^\infty \n_\ep^\alpha |u_\ep|^2 (1+|u_\ep|^2)^{\frac{s_2-2}{2}} r^{N-3} dr \\
& \quad + \ep \int_\ep^\infty \theta (s_{u_\ep}-1) \n_\ep^\theta (1+|u_\ep|^2)^{\frac{s_2-2}{2}} |\p_r u_\ep|^2 r^{N-1} dr \\
& \quad + \ep \left(\theta(N-1)^2-(N-1)(N-2)\right) \int_\ep^\infty \n_\ep^\theta |u_\ep|^2 (1+|u_\ep|^2)^{\frac{s_2-2}{2}} r^{N-3} dr \\
& = (1-\alpha)(N-1) \int_\ep^\infty s_{u_\ep} \n_\ep^\alpha u_\ep (1+|u_\ep|^2)^{\frac{s_2-2}{2}} (\p_r u_\ep) r^{N-2} dr \\
& \quad + \ep (1-\theta)(N-1) \int_\ep^\infty s_{u_\ep} \n_\ep^\theta u_\ep (1+|u_\ep|^2)^{\frac{s_2-2}{2}} (\p_r u_\ep) r^{N-2} dr \\
& \quad - \int_\ep^\infty \p_r (\n_\ep^\ga) u_\ep (1+|u_\ep|^2)^{\frac{s_2-2}{2}} r^{N-1} dr,
\ea\ee
where
\be\la{t42}\ba
s_{u_\ep} \triangleq 2 + (s_2-2) \frac{|u_\ep|^2}{1+|u_\ep|^2}.
\ea\ee
Note that for $2 < s_2 < \mathcal{P}_{N}(\alpha)$, we have
\be\la{t43}\ba
M_{N,s_2}(\alpha) \triangleq \frac{ s_2^2 (N-1)^2 (1-\alpha)^2 }{4\alpha (s_2-1) \left(\alpha(N-1)^2-(N-1)(N-2)\right) } < 1.
\ea\ee
Define
\be\la{t44}\ba
\delta_1 \triangleq \frac{ 1-\sqrt{ M_{N,s_2}(\alpha) } }{2} \in \left(0,\frac{1}{2}\right].
\ea\ee
Since the function $f(z) = \frac{z^2}{z-1}$ is strictly increasing on $[2,\infty)$ and $2 \le s_{u_\ep} < s_2$, we obtain
\be\la{t45}\ba
\frac{ s_{u_\ep}^2 (N-1)^2 (1-\alpha)^2 }{4\alpha (s_{u_\ep}-1) \left(\alpha(N-1)^2-(N-1)(N-2)\right) } \le M_{N,s_2}(\alpha) = ( 1 - 2\delta_1 )^2 < (1-\delta_1)^2.
\ea\ee
It follows from (\ref{t45}) and Young's inequality that
\be\la{t46}\ba
& |1-\alpha| (N-1) s_{u_\ep} \n_\ep^\alpha |u_\ep| (1+|u_\ep|^2)^{\frac{s_2-2}{2}} |\p_r u_\ep| r^{N-2} \\
& \le (1-\delta_1) \alpha (s_{u_\ep}-1) \n_\ep^\alpha (1+|u_\ep|^2)^{\frac{s_2-2}{2}} |\p_r u_\ep|^2 r^{N-1} \\
& \quad + \frac{ s_{u_\ep}^2 (N-1)^2 (1-\alpha)^2 \n_\ep^\alpha |u_\ep|^2 (1+|u_\ep|^2)^{\frac{s_2-2}{2}} r^{N-3} } {4 \alpha (s_{u_\ep}-1) (1-\delta_1) } \\
& \le (1-\delta_1) \alpha (s_{u_\ep}-1) \n_\ep^\alpha (1+|u_\ep|^2)^{\frac{s_2-2}{2}} |\p_r u_\ep|^2 r^{N-1} \\
& \quad + (1-\delta_1) \left(\alpha(N-1)^2-(N-1)(N-2)\right) \n_\ep^\alpha |u_\ep|^2 (1+|u_\ep|^2)^{\frac{s_2-2}{2}} r^{N-3}.
\ea\ee
On the other hand, the definition of $\theta$ implies that $\mathcal{P}_{N}(\alpha) \le \mathcal{P}_{N}(\theta)$.
Hence, for $2 < s_2 < \mathcal{P}_{N}(\alpha)$, it holds that
\be\la{t47}\ba
M_{N,s_2}(\theta) \triangleq \frac{ s_2^2 (N-1)^2 (1-\theta)^2 }{4\theta (s_2-1) \left(\theta(N-1)^2-(N-1)(N-2)\right) } < 1.
\ea\ee
Set
\be\la{t48}\ba
\delta_2 \triangleq \frac{ 1 - \sqrt{ M_{N,s_2}(\theta) } }{2} \in \left(0,\frac{1}{2}\right],
\ea\ee
which together with $2 \le s_{u_\ep}<s_2$ gives
\be\la{t48a}\ba
\frac{ s_{u_\ep}^2 (N-1)^2 (1-\theta)^2 }{4\theta (s_{u_\ep}-1) \left(\theta(N-1)^2-(N-1)(N-2)\right) } \le M_{N,s_2}(\theta) = ( 1 - 2\delta_2 )^2 < ( 1 - \delta_2 )^2.
\ea\ee
Similarly, Young's inequality implies
\be\la{t49}\ba
& |1-\theta| (N-1) s_{u_\ep} \n_\ep^\theta |u_\ep| (1+|u_\ep|^2)^{\frac{s_2-2}{2}} |\p_r u_\ep| r^{N-2} \\
& \le (1-\delta_2) \theta (s_{u_\ep}-1) \n_\ep^\theta (1+|u_\ep|^2)^{\frac{s_2-2}{2}} |\p_r u_\ep|^2 r^{N-1} \\
& \quad + (1-\delta_2) \left(\theta(N-1)^2-(N-1)(N-2)\right) \n_\ep^\theta |u_\ep|^2 (1+|u_\ep|^2)^{\frac{s_2-2}{2}} r^{N-3}.
\ea\ee
Recalling (\ref{bj15}), we have that $K(\n_\ep)$ satisfies
\be\la{t410}\ba
\frac{d}{dt} \int_\ep^\infty K(\n_\ep) r^{N-1} dr - \int_\ep^\infty \p_r(\n_\ep^\ga) u_\ep r^{N-1} dr = 0.
\ea\ee
Putting (\ref{t46}) and (\ref{t49}) into (\ref{t41}) and adding the resulting inequality to (\ref{t410}), we arrive at
\be\la{t411}\ba
& \frac{d}{dt} \int_\ep^\infty \left[ \frac{1}{s_2} \n_\ep \left( (1+|u_\ep|^2)^{\frac{s_2}{2}}-1 \right) + K(\n_\ep) \right] r^{N-1} dr \\
& \quad + \delta_1 \int_\ep^\infty \alpha (s_{u_\ep}-1) \n_\ep^\alpha (1+|u_\ep|^2)^{\frac{s_2-2}{2}} |\p_r u_\ep|^2 r^{N-1} dr \\
& \quad + \delta_1 \left(\alpha(N-1)^2-(N-1)(N-2)\right) \int_\ep^\infty \n_\ep^\alpha |u_\ep|^2 (1+|u_\ep|^2)^{\frac{s_2-2}{2}} r^{N-3} dr \\
& \quad + \delta_2 \ep \int_\ep^\infty \theta (s_{u_\ep}-1) \n_\ep^\theta (1+|u_\ep|^2)^{\frac{s_2-2}{2}} |\p_r u_\ep|^2 r^{N-1} dr \\
& \quad + \delta_2 \ep \left(\theta(N-1)^2-(N-1)(N-2)\right) \int_\ep^\infty \n_\ep^\theta |u_\ep|^2 (1+|u_\ep|^2)^{\frac{s_2-2}{2}} r^{N-3} dr \\
& \le - \int_\ep^\infty \p_r (\n_\ep^\ga) \left[ u_\ep (1+|u_\ep|^2)^{\frac{s_2-2}{2}} - u_\ep \right] r^{N-1} dr.
\ea\ee
We next estimate the pressure term.
Let $R_0$ be as in Lemma \ref{cptl2}, and choose a cutoff function $\eta(r) \in C^\infty([0,\infty))$ satisfying
\be\la{bj415}\ba
0 \le \eta(r) \le 1, \quad
\eta(r) =
\begin{cases}
0 & \text{ if } 0 \le r \le R_0, \\
1 & \text{ if } r \ge 2 R_0.
\end{cases}
\ea\ee
Integration by parts gives
\be\la{t411a}\ba
& - \int_\ep^\infty \p_r (\n_\ep^\ga) \left[ u_\ep (1+|u_\ep|^2)^{\frac{s_2-2}{2}} - u_\ep \right] r^{N-1} dr \\
& = - \int_\ep^\infty \p_r \left(\n_\ep^\ga - \eta(r)\right) \left[ u_\ep (1+|u_\ep|^2)^{\frac{s_2-2}{2}} - u_\ep \right] r^{N-1} dr \\
& \quad - \int_\ep^\infty \p_r \left( \eta(r) \right) \left[ u_\ep (1+|u_\ep|^2)^{\frac{s_2-2}{2}} - u_\ep \right] r^{N-1} dr \\
& = \int_\ep^\infty \left(\n_\ep^\ga - \eta(r)\right) \left[ (s_{u_\ep}-1) (1+|u_\ep|^2)^{\frac{s_2-2}{2}} - 1 \right] (\p_r u_\ep) r^{N-1} dr \\
& \quad + (N-1) \int_\ep^\infty \left(\n_\ep^\ga - \eta(r)\right) \left[ u_\ep (1+|u_\ep|^2)^{\frac{s_2-2}{2}} - u_\ep \right] r^{N-2} dr \\
& \quad - \int_{R_0}^{2 R_0} \eta'(r) \left[ u_\ep (1+|u_\ep|^2)^{\frac{s_2-2}{2}} - u_\ep \right] r^{N-1} dr
\triangleq J_1 + J_2 + J_3.
\ea\ee
Note that for any $z\in (0,\infty)$ and $s \in (0,1]$, it holds that
\be\la{t412}\ba
(1+z)^s - 1 \le \min\{z,z^s\},
\ea\ee
which together with the fact that $2 < s_2 \le 4$ implies
\be\la{t415}\ba
(1+|u_\ep|^2)^{\frac{s_2-2}{2}} - 1 \le \min\{ |u_\ep|^2,|u_\ep|^{s_2-2} \}.
\ea\ee
Using (\ref{t415}) and the definition of $s_{u_\ep}$, we arrive at
\be\la{t416}\ba
& (s_{u_\ep}-1) (1+|u_\ep|^2)^{\frac{s_2-2}{2}} - 1 \\
& = (1+|u_\ep|^2)^{\frac{s_2-2}{2}} - 1 + (s_2-2) \frac{|u_\ep|^2}{1+|u_\ep|^2} (1+|u_\ep|^2)^{\frac{s_2-2}{2}} \\
& \le \min\{ |u_\ep|^2,|u_\ep|^{s_2-2} \} + (s_2-2) |u_\ep|^2 (1+|u_\ep|^2)^{\frac{s_2-4}{2}} \\
& \le C \min\{ |u_\ep|^2,|u_\ep|^{s_2-2} \}.
\ea\ee
Define
\be\la{aue}\ba
A(u_\ep) \triangleq \min\{ |u_\ep|^2,|u_\ep|^{s_2-2} \}.
\ea\ee
By (\ref{t416}), Young's inequality, and the fact that $s_{u_\ep} \ge 2$, we have
\be\la{t416a1}\ba
|J_1| & \le \frac{\delta_1}{4} \int_\ep^\infty \alpha \left[ (s_{u_\ep}-1) (1+|u_\ep|^2)^{\frac{s_2-2}{2}} - 1 \right] \n_\ep^\alpha |\p_r u_\ep|^2 r^{N-1} dr \\
& \quad + C \int_\ep^\infty \n_\ep^{-\alpha} |\n_\ep^\ga - \eta(r)|^2 \left[ (s_{u_\ep}-1) (1+|u_\ep|^2)^{\frac{s_2-2}{2}} - 1 \right] r^{N-1} dr \\
& \le \frac{\delta_1}{4} \int_\ep^\infty \alpha (s_{u_\ep}-1) \n_\ep^\alpha (1+|u_\ep|^2)^{\frac{s_2-2}{2}} |\p_r u_\ep|^2 r^{N-1} dr \\
& \quad + C \int_\ep^{R_0} \n_\ep^{2\ga-\alpha} A(u_\ep) r^{N-1} dr
+ C \int_{R_0}^\infty \n_\ep^{-\alpha} |\n_\ep^\ga - \eta(r)|^2 |u_\ep|^2 r^{N-1} dr.
\ea\ee
Similarly, Young's inequality and (\ref{t415}) yield
\be\la{t416a2}\ba
|J_2| & \le \frac{\delta_1}{8} \left(\alpha(N-1)^2-(N-1)(N-2)\right) \int_\ep^\infty \n_\ep^\alpha |u_\ep|^2 \left[ (1+|u_\ep|^2)^{\frac{s_2-2}{2}} - 1 \right] r^{N-3} dr \\
& \quad + C \int_\ep^\infty \n_\ep^{-\alpha} |\n_\ep^\ga - \eta(r)|^2 \left[ (1+|u_\ep|^2)^{\frac{s_2-2}{2}} - 1 \right] r^{N-1} dr \\
& \le \frac{\delta_1}{8} \left(\alpha(N-1)^2-(N-1)(N-2)\right) \int_\ep^\infty \n_\ep^\alpha |u_\ep|^2 (1+|u_\ep|^2)^{\frac{s_2-2}{2}} r^{N-3} dr \\
& \quad + C \int_\ep^{R_0} \n_\ep^{2\ga-\alpha} A(u_\ep) r^{N-1} dr
+ C \int_{R_0}^\infty \n_\ep^{-\alpha} |\n_\ep^\ga - \eta(r)|^2 |u_\ep|^2 r^{N-1} dr,
\ea\ee
and
\be\la{t416a3}\ba
|J_3| & \le \frac{\delta_1}{8} \left(\alpha(N-1)^2-(N-1)(N-2)\right) \int_{R_0}^{2 R_0} \n_\ep^\alpha |u_\ep|^2 \left[ (1+|u_\ep|^2)^{\frac{s_2-2}{2}} - 1 \right] r^{N-3} dr \\
& \quad + C \int_{R_0}^{2 R_0} \left[ (1+|u_\ep|^2)^{\frac{s_2-2}{2}} - 1 \right] r^{N-1} dr \\
& \le \frac{\delta_1}{8} \left(\alpha(N-1)^2-(N-1)(N-2)\right) \int_\ep^\infty \n_\ep^\alpha |u_\ep|^2 (1+|u_\ep|^2)^{\frac{s_2-2}{2}} r^{N-3} dr \\
& \quad + C \int_\ep^\infty \n_\ep^\alpha |u_\ep|^2 r^{N-3} dr.
\ea\ee
It follows from (\ref{cpt02}) that
\be\la{cpt47}\ba
\int_{R_0}^{2 R_0} \n_\ep^{-\alpha} |\n_\ep^\ga - \eta(r)|^2 |u_\ep|^2 r^{N-1} dr
\le C \int_{R_0}^{2 R_0} \n_\ep^{\alpha} |u_\ep|^2 r^{N-3} dr
\le C \int_{\ep}^{\infty} \n_\ep^{\alpha} |u_\ep|^2 r^{N-3} dr.
\ea\ee
Moreover, in view of (\ref{cpt01}), (\ref{cpt001}), (\ref{cpt02}), and the fact that $\eta(r) = 1$ on $r \in (2 R_0,\infty)$, we derive
\be\la{cpt48}\ba
\int_{2 R_0}^\infty \n_\ep^{-\alpha} |\n_\ep^\ga - \eta(r)|^2 |u_\ep|^2 r^{N-1} dr
& = \int_{2 R_0}^\infty \n_\ep^{-\alpha} |\n_\ep^\ga - 1|^2 |u_\ep|^2 r^{N-1} dr \\
& \le C \| u_\ep \|_{L^\infty(2R_0,\infty)}^2 \int_{2 R_0}^\infty |\n_\ep^{\alpha-\frac{1}{2}} - 1|^2 r^{N-1} dr \\
& \le C \int_\ep^\infty \n_\ep^\alpha |u_\ep|^2 r^{N-3} dr
+ C \int_\ep^\infty \n_\ep^\alpha |\p_r u_\ep|^2 r^{N-1} dr,
\ea\ee
where in the last inequality we have used the following estimate: for any $r > 2 R_0$,
\be\la{cpt49}\ba
|u_\ep(r)|^2
& = -2 \int_r^\infty u_\ep(s) (\p_s u_\ep(s) ) ds \\
& \le C \left( \int_{r}^\infty |u_\ep(s)|^2 s^{1-N} ds \right)^{\frac{1}{2}}
\left( \int_{r}^\infty |\p_s u_\ep(s)|^2 s^{N-1} ds \right)^{\frac{1}{2}} \\
& \le C \left( \int_r^\infty \n_\ep^\alpha |u_\ep(s)|^2 s^{N-3} ds \right)^{\frac{1}{2}}
\left( \int_r^\infty \n_\ep^\alpha |\p_s u_\ep(s)|^2 s^{N-1} ds \right)^{\frac{1}{2}} \\
& \le C \int_\ep^\infty \n_\ep^\alpha |u_\ep|^2 r^{N-3} dr
+ C \int_\ep^\infty \n_\ep^\alpha |\p_r u_\ep|^2 r^{N-1} dr.
\ea\ee
From (\ref{t411a}), (\ref{t416a1}), (\ref{t416a2}), (\ref{t416a3}), (\ref{cpt47}), and (\ref{cpt48}), we conclude that
\be\la{cpt410}\ba
& - \int_\ep^\infty \p_r (\n_\ep^\ga) \left[ u_\ep (1+|u_\ep|^2)^{\frac{s_2-2}{2}} - u_\ep \right] r^{N-1} dr \\
& \le \frac{\delta_1}{2} \int_\ep^\infty \alpha (s_{u_\ep}-1) \n_\ep^\alpha (1+|u_\ep|^2)^{\frac{s_2-2}{2}} |\p_r u_\ep|^2 r^{N-1} dr \\
& \quad + \frac{\delta_1}{2} \left(\alpha(N-1)^2-(N-1)(N-2)\right) \int_\ep^\infty \n_\ep^\alpha |u_\ep|^2 (1+|u_\ep|^2)^{\frac{s_2-2}{2}} r^{N-3} dr \\
& \quad + C \int_\ep^{R_0} \n_\ep^{2\ga-\alpha} A(u_\ep) r^{N-1} dr
+ C \int_\ep^\infty \left( \n_\ep^{\alpha} \frac{|u_\ep|^2}{r^2} + \n_\ep^{\alpha} |\p_r u_\ep|^2 \right) r^{N-1} dr.
\ea\ee
Substituting (\ref{cpt410}) into (\ref{t411}) leads to
\be\la{cpt414}\ba
& \frac{d}{dt} \int_\ep^\infty \left[ \frac{1}{s_2} \n_\ep \left( (1+|u_\ep|^2)^{\frac{s_2}{2}}-1 \right) + K(\n_\ep) \right] r^{N-1} dr \\
& \quad + \frac{\delta_1}{2} \int_\ep^\infty \alpha (s_{u_\ep}-1) \n_\ep^\alpha (1+|u_\ep|^2)^{\frac{s_2-2}{2}} |\p_r u_\ep|^2 r^{N-1} dr \\
& \quad + \frac{\delta_1}{2} \left(\alpha(N-1)^2-(N-1)(N-2)\right) \int_\ep^\infty \n_\ep^\alpha |u_\ep|^2 (1+|u_\ep|^2)^{\frac{s_2-2}{2}} r^{N-3} dr \\
& \quad + \delta_2 \ep \int_\ep^\infty \theta (s_{u_\ep}-1) \n_\ep^\theta (1+|u_\ep|^2)^{\frac{s_2-2}{2}} |\p_r u_\ep|^2 r^{N-1} dr \\
& \quad + \delta_2 \ep \left(\theta(N-1)^2-(N-1)(N-2)\right) \int_\ep^\infty \n_\ep^\theta |u_\ep|^2 (1+|u_\ep|^2)^{\frac{s_2-2}{2}} r^{N-3} dr \\
& \le C \int_\ep^{R_0} \n_\ep^{2\ga-\alpha} A(u_\ep) r^{N-1} dr
+ C \int_\ep^\infty \left( \n_\ep^{\alpha} \frac{|u_\ep|^2}{r^2} + \n_\ep^{\alpha} |\p_r u_\ep|^2 \right) r^{N-1} dr.
\ea\ee
Integrating (\ref{cpt414}) over $(0,T)$ and using (\ref{bj01}) and $s_{u_\ep} \ge 2$, we obtain
\be\la{cpt415}\ba
& \sup_{0 \le t \le T} \int_\ep^\infty \left[ \frac{1}{s_2} \n_\ep \left( (1+|u_\ep|^2)^{\frac{s_2}{2}}-1 \right) + K(\n_\ep) \right] r^{N-1} dr \\
& \quad + \int_0^T \int_\ep^\infty \left( \n_\ep^{\alpha} \frac{|u_\ep|^{s_2}}{r^2}
+ \n_\ep^{\alpha} |u_\ep|^{s_2-2} |\p_r u_\ep|^2 \right) r^{N-1} dr dt \\
& \quad + \ep \int_0^T \int_\ep^\infty \left( \n_\ep^{\theta} \frac{|u_\ep|^{s_2}}{r^2} + \n_\ep^{\theta} |u_\ep|^{s_2-2} |\p_r u_\ep|^2 \right) r^{N-1} dr dt \\
& \le C \int_\ep^\infty \n_{0,\ep} \left( (1+|u_{0,\ep}|^2)^{\frac{s_2}{2}}-1 \right) r^{N-1} dr + C
+ C \int_0^T \int_\ep^{R_0} \n_\ep^{2\ga-\alpha} A(u_\ep) r^{N-1} dr dt.
\ea\ee
Noting that for any $2 < s_2 \le 4$,
\be\nonumber\ba
|u_\ep|^{s_2} \le (1+|u_\ep|^2)^{\frac{s_2}{2}}-1 \le C \left( |u_\ep|^2 + |u_\ep|^{s_2} \right) \le C \left( |u_\ep|^2 + |u_\ep|^{p_1} \right),
\ea\ee
we deduce from (\ref{bj01}) and (\ref{cpt415}) that
\be\la{cpt416}\ba
& \sup_{0 \le t \le T} \int_\ep^\infty \n_\ep |u_\ep|^{s_2} r^{N-1} dr
+ \int_0^T \int_\ep^\infty \left( \n_\ep^{\alpha} \frac{|u_\ep|^{s_2}}{r^2}
+ \n_\ep^{\alpha} |u_\ep|^{s_2-2} |\p_r u_\ep|^2 \right) r^{N-1} dr dt \\
& \quad + \ep \int_0^T \int_\ep^\infty \left( \n_\ep^{\theta} \frac{|u_\ep|^{s_2}}{r^2} + \n_\ep^{\theta} |u_\ep|^{s_2-2} |\p_r u_\ep|^2 \right) r^{N-1} dr dt \\
& \le C + C \int_\ep^\infty \left( \n_{0,\ep} |u_{0,\ep}|^2 + \n_{0,\ep} |u_{0,\ep}|^{p_1} \right) r^{N-1} dr
+ C \int_0^T \int_\ep^{R_0} \n_\ep^{2\ga-\alpha} A(u_\ep) r^{N-1} dr dt \\
& \le C + C \int_\ep^\infty \left( \n^{-1}_{0,\ep} |m_{0,\ep}|^2 + \n^{-p_1+1}_{0,\ep} |m_{0,\ep}|^{p_1} \right) r^{N-1} dr
+ C \int_0^T \int_\ep^{R_0} \n_\ep^{2\ga-\alpha} A(u_\ep) r^{N-1} dr dt \\
& \le C + C \int_0^T \int_\ep^{R_0} \n_\ep^{2\ga-\alpha} A(u_\ep) r^{N-1} dr dt.
\ea\ee
Next, we estimate the last integral in (\ref{cpt416}).
We distinguish three cases.

\textit{Case 1: $\ga<\alpha$.}

For $c_1$ as in Lemma \ref{cptl2}, on the one hand, by (\ref{aue}), we have
\be\la{xcpt41}\ba
& \int_0^T \int_\ep^{R_0} \chi_{(\n_\ep>\frac{c_1}{2})} \n_\ep^{2\ga-\alpha} A(u_\ep) r^{N-1} dr dt \\
& \le \int_0^T \int_\ep^{R_0} \chi_{(\n_\ep>\frac{c_1}{2})} \n_\ep^{2\ga-\alpha} |u_\ep|^2 r^{N-1} dr dt \\
& \le C \int_0^T \int_\ep^{R_0} \chi_{(\n_\ep>\frac{c_1}{2})} \n_\ep^\alpha |u_\ep|^2 r^{N-3} dr dt \le C.
\ea\ee

On the other hand, it follows from (\ref{t04a}) and (\ref{cpt02a}) that
\be\la{xcpt42}\ba
& \int_0^T \int_\ep^{R_0} \chi_{(\n_\ep \le \frac{c_1}{2})} \n_\ep^{2\ga-\alpha} A(u_\ep) r^{N-1} dr dt \\
& \le \int_0^T \int_\ep^{R_0} \chi_{(\n_\ep \le \frac{c_1}{2})} \n_\ep^{2\ga-\alpha} |u_\ep|^{s_2 - 2} r^{N-1} dr dt \\
& \le \int_0^T \left( \int_\ep^{R_0} \n_\ep |u_\ep|^{s_2} r^{N-1} dr \right)^{\frac{s_2-2}{s_2}}
\left( \int_\ep^{R_0} \chi_{(\n_\ep \le \frac{c_1}{2})} \n_\ep^{\frac{s_2}{2}(2\ga-\alpha-1)+1} r^{N-1} dr \right)^{\frac{2}{s_2}} dt \\
& \le \int_0^T \left( \int_\ep^{R_0} \n_\ep |u_\ep|^{s_2} r^{N-1} dr \right)^{\frac{s_2-2}{s_2}}
\left( \int_\ep^{R_0} \chi_{(\n_\ep \le \frac{c_1}{2})} r^{N-1} dr \right)^{\frac{2}{s_2}} dt \\
& \le C \int_0^T \left( \int_\ep^{R_0} \n_\ep |u_\ep|^{s_2} r^{N-1} dr \right)^{\frac{s_2-2}{s_2}}
\int_\ep^{R_0} |\p_r \n_\ep^{\frac{\ga+\alpha-1}{2}}|^{2} r^{N-1} dr dt \\
& \le C + C \int_0^T \left( \int_\ep^{R_0} \n_\ep |u_\ep|^{s_2} r^{N-1} dr
\int_\ep^{R_0} |\p_r \n_\ep^{\frac{\ga+\alpha-1}{2}}|^{2} r^{N-1} dr \right) dt,
\ea\ee
which along with (\ref{xcpt41}) yields
\be\la{xcpt43}\ba
& \int_0^T \int_\ep^{R_0} \n_\ep^{2\ga-\alpha} A(u_\ep) r^{N-1} dr dt \\
& \le C + C \int_0^T \left( \int_\ep^{R_0} \n_\ep |u_\ep|^{s_2} r^{N-1} dr
\int_\ep^{R_0} |\p_r \n_\ep^{\frac{\ga+\alpha-1}{2}}|^{2} r^{N-1} dr \right) dt.
\ea\ee

\textit{Case 2: $N=2$ and $\ga \ge \alpha$.}

Using (\ref{aue}), (\ref{bj01}), (\ref{cpt002}), and the fact that $\ga \ge \alpha$, we derive
\be\la{t419a}\ba
\int_0^T \int_\ep^{R_0} \n_\ep^{2\ga-\alpha} A(u_\ep) r dr dt
& \le \int_0^T \int_\ep^{R_0} \n_\ep^{2\ga-\alpha} |u_\ep|^2 r dr dt \\
& \le \sup_{0 \le t \le T} \| \n_\ep^{2(\ga-\alpha)} r^2 \|_{L^\infty(\ep,R_0)} \int_0^T \int_\ep^{R_0} \n_\ep^{\alpha} |u_\ep|^2 r^{-1} dr dt \\
& \le C + C \sup_{0 \le t \le T} \| \n_\ep^{2(\ga-\alpha+1)} r^2 \|_{L^\infty(\ep,R_0)} \\
& \le C + C \sup_{0 \le t \le T} \| \n_\ep r^{\frac{1}{\ga-\alpha+1}} \|_{L^\infty(\ep,R_0)}^{2(\ga-\alpha+1)}
\le C.
\ea\ee

\textit{Case 3: $N=3$ and $\ga \ge \alpha$.}

For $c_2$ as in Lemma \ref{cptl2}, on the one hand, from (\ref{bj01}), (\ref{aue}), and $\ga \ge \alpha$, we deduce that
\be\la{3t2}\ba
\int_0^T \int_\ep^{R_0} \chi_{(\n_\ep \le 2c_2)} \n_\ep^{2\ga-\alpha} A(u_\ep) r^{2} dr dt
& \le \int_0^T \int_\ep^{R_0} \chi_{(\n_\ep \le 2c_2)} \n_\ep^{2\ga-\alpha} |u_\ep|^2 r^{2} dr dt \\
& \le C \int_0^T \int_\ep^{R_0} \chi_{(\n_\ep \le 2c_2)} \n_\ep^{\alpha} |u_\ep|^2 dr dt \le C.
\ea\ee
On the other hand, by (\ref{aue}) and Young's inequality, we have for any $\de>0$,
\be\la{3t3}\ba
& \int_0^T \int_\ep^{R_0} \chi_{(\n_\ep > 2c_2)} \n_\ep^{2\ga-\alpha} A(u_\ep) r^{2} dr dt \\
& \le \int_0^T \int_\ep^{R_0} \chi_{(\n_\ep > 2c_2)} \n_\ep^{2\ga-\alpha} |u_\ep|^{s_2-2} r^{2} dr dt \\
& \le \de \int_0^T \int_\ep^{R_0} \n_\ep^\alpha |u_\ep|^{s_2} dr dt
+ C(\de) \int_0^T \int_\ep^{R_0} \chi_{(\n_\ep > 2c_2)} \n_\ep^{s_2(\ga-\alpha)+\alpha} r^{s_2} dr dt.
\ea\ee
We next estimate the last integral in (\ref{3t3}).
Set
\be\la{gama1}\ba
(\ga-6\alpha+3)_+ \triangleq \max\{\ga-6\alpha+3,0\},
\ea\ee
and
\be\la{gama2}\ba
\si \triangleq (3\alpha-2)(s_2-2)+(s_2-1)(\ga-6\alpha+3)_+.
\ea\ee
For any fixed $t \in [0,T]$, from (\ref{cpt0002}), (\ref{cpt012}), and (\ref{gama2}), we deduce that
\be\la{4x1}\ba
& \int_\ep^{R_0} \chi_{(\n_\ep > 2c_2)} \n_\ep^{s_2(\ga-\alpha)+\alpha} r^{s_2} dr \\
& \le \left( 1 + \| \n_\ep \|_{L^\infty(\ep,R_0)}^{\si} \right) \| \n_\ep^{2\alpha-1} r \|_{L^\infty(\ep,R_0)}^{s_2} \int_\ep^{R_0} \chi_{(\n_\ep > 2c_2)} \n_\ep^{\ga+\alpha-1} dr \\
& \le C \left( 1 + \| \n_\ep \|_{L^\infty(\ep,R_0)}^{\si} \right) \int_\ep^{R_0} |\p_r \n_\ep^{\frac{\ga+\alpha-1}{2}}|^{2} r^{2} dr.
\ea\ee
The definition of $w_\ep$ shows that
\be\la{4x3}\ba
\frac{1}{(\alpha-1+\frac{1}{s_2})^{s_2}} |\p_r \n_\ep^{\alpha-1+\frac{1}{s_2}}|^{s_2}
& = \n_\ep^{s_2(\alpha-2)+1} |\p_r \n_\ep|^{s_2} \\
& \le \alpha^{-s_2} \n_\ep^{1-s_2} (\alpha \n_\ep^{\alpha-1} + \ep \theta \n_\ep^{\theta-1})^{s_2} |\p_r \n_\ep|^{s_2} \\
& \le C \n_\ep |u_\ep|^{s_2} + C \n_\ep |w_\ep|^{s_2},
\ea\ee
which implies
\be\la{4x4}\ba
\int_\ep^\infty |\p_r \n_\ep^{\alpha-1+\frac{1}{s_2}}|^{s_2} r^{2} dr
\le C \int_\ep^\infty \left( \n_\ep |u_\ep|^{s_2} + \n_\ep |w_\ep|^{s_2} \right) r^{2} dr.
\ea\ee
For $\tilde{\n}_\ep(r,t)$ as in (\ref{cpt13}), we define
\be\la{4x5}\ba
\ka \triangleq \alpha-1+\frac{1}{s_2}, \quad \Pi_2(x,t) \triangleq ( \tilde{\n}_\ep^\ka(|x|,t) - (c_2+2)^\ka )_+.
\ea\ee
Using (\ref{cpt001}), (\ref{4x4}), (\ref{4x5}), and a change of variables, we obtain
\be\la{4x6}\ba
\int_0^{R_0} |\Pi_2|^{\frac{6\alpha-3}{\ka}} r^2 dr
& \le 2 \int_\ep^{R_0} |\Pi_2|^{\frac{6\alpha-3}{\ka}} r^2 dr \\
& \le C \int_\ep^{R_0} \chi_{(\n_\ep \ge c_2+2)} \n_\ep^{6\alpha-3} r^2 dr \\
& \le C \int_\ep^{R_0} \chi_{(\n_\ep \ge 2)} |\n_\ep^{\alpha-\frac{1}{2}}-1|^{6} r^2 dr \le C,
\ea\ee
and
\be\la{4x7}\ba
\int_0^{R_0} |\p_r \Pi_2|^{s_2} r^{2} dr
& \le 2 \int_\ep^{R_0} |\p_r \Pi_2|^{s_2} r^{2} dr \\
& \le C \int_\ep^{R_0} |\p_r \n_\ep^{\alpha-1+\frac{1}{s_2}}|^{s_2} r^{2} dr \\
& \le C \int_\ep^{R_0} \left( \n_\ep |u_\ep|^{s_2} + \n_\ep |w_\ep|^{s_2} \right) r^{2} dr.
\ea\ee
Since $\tilde{\n}_\ep(R_0,t) = \n_\ep(R_0,t) \le c_2$, we have $\Pi_2(R_0,t) = 0$.
Hence, the Gagliardo-Nirenberg inequality implies
\be\la{4x8}\ba
\| \Pi_2 \|_{L^\infty(B_{R_0})} \le C \| \Pi_2 \|_{L^{\frac{6\alpha-3}{\ka}}(B_{R_0})}^{a}
\| \na \Pi_2 \|_{L^{s_2}(B_{R_0})}^{ 1-a },
\ea\ee
where
\be\la{4x9}\ba
a = \frac{(6\alpha-3)(s_2-3)}{3 \ka s_2 + (6\alpha-3)(s_2-3)}, \quad 1-a = \frac{3 \ka s_2}{3 \ka s_2 + (6\alpha-3)(s_2-3)}.
\ea\ee
Set
\be\la{zb11}\ba
b \triangleq \frac{(s_2-1)(\ga-6\alpha+3)_+}{(3\alpha-2)(s_2-2)} \ge 0.
\ea\ee
By (\ref{4x5}), (\ref{zb11}), and (\ref{gama2}), we have
\be\la{zb12}\ba
3 \ka s_2 + (6\alpha-3)(s_2-3) = 3 (3\alpha-2)(s_2-2),
\ea\ee
and
\be\la{zb13}\ba
(1-a)\frac{\si}{\ka} = \frac{3 {\si} s_2}{3 \ka s_2 + (6\alpha-3)(s_2-3)}
= \frac{{\si} s_2}{(3\alpha-2)(s_2-2)}
= s_2(1+b).
\ea\ee
It follows from (\ref{4x8}), (\ref{4x9}), (\ref{zb13}), and (\ref{4x7}) that for any $t \in [0,T]$,
\be\la{4x10}\ba
\| \tilde{\n}_\ep \|_{L^\infty(B_{R_0})}^{\si} & \le C + C \| \Pi_2 \|_{L^\infty(B_{R_0})}^{\frac{\si}{\ka}} \\
& \le C + C \| \na \Pi_2 \|_{L^{s_2}(B_{R_0})}^{ (1-a)\frac{\si}{\ka} } \\
& \le C + C \| \na \Pi_2 \|_{L^{s_2}(B_{R_0})}^{s_2(1+b)} \\
& \le C + C \left( \int_\ep^\infty \left( \n_\ep |u_\ep|^{s_2} + \n_\ep |w_\ep|^{s_2} \right) r^{2} dr \right)^{1+b}.
\ea\ee
Set
\be\la{4x11}\ba
\tilde{b} \triangleq
\begin{cases}
0 \quad & \text{ if } N=2, \\
b \quad & \text{ if } N=3.
\end{cases}
\ea\ee

Hence, when $\ga \ge \alpha$, the combination of (\ref{t419a}), (\ref{3t2}), (\ref{3t3}), (\ref{4x1}), and (\ref{4x10}) yields for any $\de>0$,
\be\la{4x13}\ba
& \int_0^T \int_\ep^{R_0} \chi_{(\n_\ep > 2c_2)} \n_\ep^{2\ga-\alpha} A(u_\ep) r^{N-1} dr dt \\
& \le \de \int_0^T \int_\ep^{R_0} \n_\ep^\alpha |u_\ep|^{s_2} r^{N-3} dr dt
+ C(\de) + C(\de) \int_0^T \left( Y^{1+\tilde{b}}(t) \int_\ep^{R_0} |\p_r \n_\ep^{\frac{\ga+\alpha-1}{2}}|^{2} r^{N-1} dr \right) dt,
\ea\ee
where
\be\la{4x14}\ba
Y(t) \triangleq \int_\ep^\infty \left( \n_\ep |u_\ep|^{s_2} + \n_\ep |w_\ep|^{s_2} \right) r^{N-1} dr.
\ea\ee
Combining (\ref{cpt416}), (\ref{xcpt43}), (\ref{t419a}), (\ref{3t2}), and (\ref{4x13}), and choosing $\de>0$ sufficiently small, we obtain
\be\la{x3t1}\ba
& \sup_{0 \le t \le T} \int_\ep^\infty \n_\ep |u_\ep|^{s_2} r^{N-1} dr
+ \int_0^T \int_\ep^\infty \left( \n_\ep^{\alpha} \frac{|u_\ep|^{s_2}}{r^2}
+ \n_\ep^{\alpha} |u_\ep|^{s_2-2} |\p_r u_\ep|^2 \right) r^{N-1} dr dt \\
& \quad + \ep \int_0^T \int_\ep^\infty \left( \n_\ep^{\theta} \frac{|u_\ep|^{s_2}}{r^2} + \n_\ep^{\theta} |u_\ep|^{s_2-2} |\p_r u_\ep|^2 \right) r^{N-1} dr dt \\
& \le C + C \int_0^T \left( Y^{1+\tilde{b}}(t) \int_\ep^{R_0} |\p_r \n_\ep^{\frac{\ga+\alpha-1}{2}}|^{2} r^{N-1} dr \right) dt.
\ea\ee
On the other hand, multiplying (\ref{bjyxsdfc1}) by $w_\ep |w_\ep|^{s_2-2} r^{N-1}$ and integrating by parts, we arrive at
\be\la{cpt51}\ba
\frac{1}{s_2} \frac{d}{dt} \int_\ep^\infty \n_\ep |w_\ep|^{s_2} r^{N-1} dr
+ \int_\ep^\infty (\n_\ep^\ga)_r w_\ep |w_\ep|^{s_2-2} r^{N-1} dr
= 0.
\ea\ee
Moreover, multiplying $(\ref{bjnsqdc})_2$ by $u_\ep |u_\ep|^{s_2-2} r^{N-1}$ and integrating by parts, we derive
\be\la{cpt52}\ba
& \frac{1}{s_2} \frac{d}{dt} \int_\ep^\infty \n_\ep |u_\ep|^{s_2} r^{N-1} dr
+ \int_\ep^\infty \p_r (\n_\ep^\ga) u_\ep |u_\ep|^{s_2-2} r^{N-1} dr \\
& = - \left(\alpha(N-1)^2-(N-1)(N-2)\right) \int_\ep^\infty \n_\ep^{\alpha} |u_\ep|^{s_2} r^{N-3} dr \\
& \quad - \alpha (s_2-1) \int_\ep^\infty \n_\ep^{\alpha} |u_\ep|^{s_2-2} |\p_r u_\ep|^2 r^{N-1} dr
- \ep \theta (s_2-1) \int_\ep^\infty \n_\ep^{\theta} |u_\ep|^{s_2-2} |\p_r u_\ep|^2 r^{N-1} dr \\
& \quad - \ep \left(\theta(N-1)^2-(N-1)(N-2)\right) \int_\ep^\infty \n_\ep^{\theta} |u_\ep|^{s_2} r^{N-3} dr \\
& \quad + s_2(N-1)(1-\alpha) \int_\ep^\infty \n_\ep^\alpha |u_\ep|^{s_2-2} u_\ep (\p_r u_\ep) r^{N-2} dr \\
& \quad + \ep s_2 (N-1)(1-\theta) \int_\ep^\infty \n_\ep^\theta |u_\ep|^{s_2-2} u_\ep (\p_r u_\ep) r^{N-2} dr
\triangleq \sum_{i=1}^6 J_i.
\ea\ee
Subtracting (\ref{cpt52}) from (\ref{cpt51}) yields
\be\la{cpt53}\ba
& \frac{1}{s_2} \frac{d}{dt} \int_\ep^\infty \left( \n_\ep |w_\ep|^{s_2} - \n_\ep |u_\ep|^{s_2} \right) r^{N-1} dr \\
& = - \int_\ep^\infty (\n_\ep^\ga)_r \left( w_\ep |w_\ep|^{s_2-2} - u_\ep |u_\ep|^{s_2-2} \right) r^{N-1} dr - \sum_{i=1}^6 J_i \\
& = - \int_\ep^\infty \frac{\ga \n_\ep^{\ga}}{ \alpha \n_\ep^{\alpha-1} + \ep \theta \n_\ep^{\theta-1} } (w_\ep - u_\ep) \left( w_\ep |w_\ep|^{s_2-2} - u_\ep |u_\ep|^{s_2-2} \right) r^{N-1} dr - \sum_{i=1}^6 J_i,
\ea\ee
where in the second equality we have used
\be\la{cpt54}\ba
\p_r (\n_\ep^\ga) = \frac{\ga \n_\ep^{\ga}}{ \alpha \n_\ep^{\alpha-1} + \ep \theta \n_\ep^{\theta-1} } (w_\ep - u_\ep),
\ea\ee
due to (\ref{bjyxsd}).

Noticing that the function $f(z) = z |z|^{s_2-2}$ is strictly increasing and $\n_\ep>0$, we have
\be\la{cpt55}\ba
- \int_\ep^\infty \frac{\ga \n_\ep^{\ga}}{ \alpha \n_\ep^{\alpha-1} + \ep \theta \n_\ep^{\theta-1} } (w_\ep - u_\ep) \left( w_\ep |w_\ep|^{s_2-2} - u_\ep |u_\ep|^{s_2-2} \right) r^{N-1} dr \le 0.
\ea\ee
Furthermore, Young's inequality implies that
\be\la{cpt56}\ba
- \sum_{i=1}^6 J_i
& \le C \int_\ep^\infty \left( \n_\ep^\alpha \frac{|u_\ep|^{s_2}}{r^2} + \n_\ep^\alpha |u_\ep|^{s_2-2} |\p_r u_\ep|^2 \right) r^{N-1} dr \\
& \quad + C \ep \int_\ep^\infty \left( \n_\ep^{\theta} \frac{|u_\ep|^{s_2}}{r^2} + \n_\ep^{\theta} |u_\ep|^{s_2-2} |\p_r u_\ep|^2 \right) r^{N-1} dr.
\ea\ee
It follows from (\ref{cpt53}), (\ref{cpt55}), and (\ref{cpt56}) that
\be\la{cpt57}\ba
& \frac{1}{s_2} \frac{d}{dt} \int_\ep^\infty \left( \n_\ep |w_\ep|^{s_2} - \n_\ep |u_\ep|^{s_2} \right) r^{N-1} dr \\
& \le C \int_\ep^\infty \left( \n_\ep^\alpha \frac{|u_\ep|^{s_2}}{r^2} + \n_\ep^\alpha |u_\ep|^{s_2-2} |\p_r u_\ep|^2 \right) r^{N-1} dr \\
& \quad + C \ep \int_\ep^\infty \left( \n_\ep^{\theta} \frac{|u_\ep|^{s_2}}{r^2} + \n_\ep^{\theta} |u_\ep|^{s_2-2} |\p_r u_\ep|^2 \right) r^{N-1} dr.
\ea\ee
Integrating (\ref{cpt57}) over $(0,T)$ leads to
\be\la{cpt58}\ba
& \sup_{0 \le t \le T} \int_\ep^\infty \n_\ep |w_\ep|^{s_2} r^{N-1} dr \\
& \le \int_\ep^\infty \n_{0,\ep} |w_{0,\ep}|^{s_2} r^{N-1} dr
+ C \sup_{0 \le t \le T} \int_\ep^\infty \n_\ep |u_\ep|^{s_2} r^{N-1} dr \\
& \quad + C \int_0^T \int_\ep^\infty \left( \n_\ep^\alpha \frac{|u_\ep|^{s_2}}{r^2} + \n_\ep^\alpha |u_\ep|^{s_2-2} |\p_r u_\ep|^2 \right) r^{N-1} dr dt \\
& \quad + C \ep \int_0^T \int_\ep^\infty \left( \n_\ep^{\theta} \frac{|u_\ep|^{s_2}}{r^2} + \n_\ep^{\theta} |u_\ep|^{s_2-2} |\p_r u_\ep|^2 \right) r^{N-1} dr dt.
\ea\ee
Using the fact that $\alpha=\theta$ when $\alpha \le 1$, and $\n_{0,\ep} \ge c_0 \ep^{\frac{1}{\alpha-\theta}}$ when $\alpha>1$, we derive
\be\la{cpt59}\ba
& \int_\ep^\infty \n_{0,\ep} |w_{0,\ep}|^{s_2} r^{N-1} dr \\
& = \int_\ep^\infty \n_{0,\ep} | u_{0,\ep} + \n_{0,\ep}^{-1} \p_r (\n_{0,\ep}^\alpha + \ep \n_{0,\ep}^\theta)|^{s_2} r^{N-1} dr \\
& \le C \int_\ep^\infty \left( \n_{0,\ep} |u_{0,\ep}|^{s_2} + | \p_r \n_{0,\ep}^{\alpha-1+\frac{1}{s_2}} |^{s_2} + \ep^{s_2} | \p_r \n_{0,\ep}^{\theta-1+\frac{1}{s_2}} |^{s_2} \right) r^{N-1} dr \\
& \le C \int_\ep^\infty \left( \n_{0,\ep} (|u_{0,\ep}|^{2}+|u_{0,\ep}|^{p_1}) + | \p_r \n_{0,\ep}^{\alpha-1+\frac{1}{s_2}} |^{s_2} + \ep^{s_2} | \p_r \n_{0,\ep}^{\alpha-1+\frac{1}{s_2}} |^{s_2} \n_{0,\ep}^{-s_2(\alpha-\theta)} \right) r^{N-1} dr \\
& \le C \int_\ep^\infty \left( \n_{0,\ep}^{-1}|m_{0,\ep}|^{2} + \n_{0,\ep}^{-p_1+1}|m_{0,\ep}|^{p_1} + | \p_r \n_{0,\ep}^{\alpha-1+\frac{1}{s_2}} |^{s_2} \right) r^{N-1} dr \\
& \le C + C \int_\ep^\infty \left( | \p_r \n_{0,\ep}^{\alpha-\frac{1}{2}} |^{2} + | \p_r \n_{0,\ep}^{\alpha-1+\frac{1}{s_1}} |^{s_1} \right) r^{N-1} dr \\
& \le C.
\ea\ee
If $N=2$, it follows from (\ref{4x11}), (\ref{x3t1}), (\ref{cpt58}), and (\ref{cpt59}) that
\be\la{cpt590}\ba
& \sup_{0 \le t \le T} Y(t)
+ \int_0^T \int_\ep^\infty \left( \n_\ep^{\alpha} \frac{|u_\ep|^{s_2}}{r^2}
+ \n_\ep^{\alpha} |u_\ep|^{s_2-2} |\p_r u_\ep|^2 \right) r dr dt \\
& \quad + \ep \int_0^T \int_\ep^\infty \left( \n_\ep^{\theta} \frac{|u_\ep|^{s_2}}{r^2} + \n_\ep^{\theta} |u_\ep|^{s_2-2} |\p_r u_\ep|^2 \right) r dr dt \\
& \le C + C \int_0^T \left( Y(t) \int_\ep^{R_0} |\p_r \n_\ep^{\frac{\ga+\alpha-1}{2}}|^{2} r dr \right) dt,
\ea\ee
which together with Gr\"onwall's inequality and (\ref{bj02}) yields
\be\la{cpt591}\ba
& \sup_{0 \le t \le T} \int_\ep^\infty \left( \n_\ep |u_\ep|^{s_2} + \n_\ep |w_\ep|^{s_2} \right) r dr \\
& \quad + \int_0^T \int_\ep^\infty \left( \n_\ep^{\alpha} \frac{|u_\ep|^{s_2}}{r^2}
+ \n_\ep^{\alpha} |u_\ep|^{s_2-2} |\p_r u_\ep|^2 \right) r dr dt \\
& \quad + \ep \int_0^T \int_\ep^\infty \left( \n_\ep^{\theta} \frac{|u_\ep|^{s_2}}{r^2} + \n_\ep^{\theta} |u_\ep|^{s_2-2} |\p_r u_\ep|^2 \right) r dr dt
\le C.
\ea\ee

Next, we suppose that $N=3$.
Since $1<\ga \le 6\alpha-2$, from (\ref{x3t1}), (\ref{cpt58}), (\ref{cpt59}), and (\ref{bj02}), we conclude that there exist positive constants $\mathbf{C}_1$, $\mathbf{C}_2$, and $\mathbf{C}_3$ depending only on $s_1$, $\alpha$, $\mu$, and the initial data such that
\be\la{cpt510}\ba
& \sup_{0 \le t \le T} Y(t)
+ \int_0^T \int_\ep^\infty \left( \n_\ep^{\alpha} \frac{|u_\ep|^{s_2}}{r^2}
+ \n_\ep^{\alpha} |u_\ep|^{s_2-2} |\p_r u_\ep|^2 \right) r^{2} dr dt \\
& \quad + \ep \int_0^T \int_\ep^\infty \left( \n_\ep^{\theta} \frac{|u_\ep|^{s_2}}{r^2} + \n_\ep^{\theta} |u_\ep|^{s_2-2} |\p_r u_\ep|^2 \right) r^{2} dr dt \\
& \le \mathbf{C}_1 + \mathbf{C}_2 \int_0^T \left( Y^{1+\tilde{b}}(t) \int_\ep^{R_0} |\p_r \n_\ep^{\frac{\ga+\alpha-1}{2}}|^{2} r^{2} dr \right) dt,
\ea\ee
and
\be\la{cpt511}\ba
\int_0^T \int_\ep^{R_0} |\p_r \n_\ep^{\frac{\ga+\alpha-1}{2}}|^{2} r^{2} dr dt \le \mathbf{C}_3.
\ea\ee
Since $1 < \ga \le 6\alpha-2$, by (\ref{zb11}), we have
\be\la{cpt5111}\ba
b = \frac{(s_2-1)(\ga-6\alpha+3)_+}{(3\alpha-2)(s_2-2)} \le \frac{s_2-1}{(3\alpha-2)(s_2-2)} \triangleq b_0.
\ea\ee
Define
\be\la{cpt5112}\ba
\eta_0 \triangleq \min\left\{ 1,\frac{ (3\alpha-2)(s_2-2) }{ (s_2-1) (1+\mathbf{C}_1)^{b_0} \mathbf{C}_2 \mathbf{C}_3 } \right\}.
\ea\ee
Thus, when $1 < \ga \le 6\alpha-3+\eta_0$, it holds that
\be\la{cpt5113}\ba
b \mathbf{C}_1^{b} \mathbf{C}_2 \mathbf{C}_3 < \frac{(s_2-1) \eta_0}{(3\alpha-2)(s_2-2)} (1+\mathbf{C}_1)^{b_0} \mathbf{C}_2 \mathbf{C}_3 \le 1.
\ea\ee
The combination of (\ref{cpt510}), (\ref{cpt511}), (\ref{cpt5113}), and Lemma \ref{gnti} yields
\be\la{cpt512}\ba
& \sup_{0 \le t \le T} \int_\ep^\infty \left( \n_\ep |u_\ep|^{s_2} + \n_\ep |w_\ep|^{s_2} \right) r^{2} dr \\
& \quad + \int_0^T \int_\ep^\infty \left( \n_\ep^{\alpha} \frac{|u_\ep|^{s_2}}{r^2}
+ \n_\ep^{\alpha} |u_\ep|^{s_2-2} |\p_r u_\ep|^2 \right) r^{2} dr dt \\
& \quad + \ep \int_0^T \int_\ep^\infty \left( \n_\ep^{\theta} \frac{|u_\ep|^{s_2}}{r^2} + \n_\ep^{\theta} |u_\ep|^{s_2-2} |\p_r u_\ep|^2 \right) r^{2} dr dt
\le C.
\ea\ee
By (\ref{cpt591}), (\ref{cpt512}), and (\ref{4x3}), we have
\be\la{cpt513}\ba
\sup_{0 \le t \le T} \int_\ep^\infty |\p_r \n_\ep^{\alpha-1+\frac{1}{s_2}}|^{s_2} r^{N-1} dr
\le C \sup_{0 \le t \le T} \int_\ep^\infty \left( \n_\ep |u_\ep|^{s_2} + \n_\ep |w_\ep|^{s_2} \right) r^{N-1} dr.
\ea\ee

In view of (\ref{cpt591}), (\ref{cpt512}), and (\ref{cpt513}), we obtain (\ref{t04}), which completes the proof of Lemma \ref{tl4}.
\end{proof}

\begin{lemma}\la{tl5}
Under the hypotheses of Lemma \ref{tl4}, there exists a positive constant $C$ independent of $\ep$ and $T$ such that
\be\la{t05}\ba
\sup_{0\le t \le T} \| \n_\ep \|_{L^\infty(\ep,\infty)} \le C.
\ea\ee
\end{lemma}
\begin{proof}
From (\ref{t04}), we conclude that there exists $s_2>N$ satisfying
\be\la{t51}\ba
\alpha-1+\frac{1}{s_2}>0,
\ea\ee
such that
\be\la{t52}\ba
\int_\ep^{R_0} |\p_r \n_\ep^{\alpha-1+\frac{1}{s_2}}(r)|^{s_2} r^{N-1} dr \le C.
\ea\ee
It follows from (\ref{cpt02}), (\ref{t51}), (\ref{t52}), H\"older's inequality, and the fundamental theorem of calculus that for any $r \in (\ep,R_0)$,
\be\la{t53}\ba
\n_\ep^{\alpha-1+\frac{1}{s_2}}(r)
& \le \n_\ep^{\alpha-1+\frac{1}{s_2}}(R_0) - \int_r^{R_0} \p_s \n_\ep^{\alpha-1+\frac{1}{s_2}}(s) ds \\
& \le C + \int_\ep^{R_0} |\p_r \n_\ep^{\alpha-1+\frac{1}{s_2}}(r)| dr \\
& \le C + C \left( \int_\ep^{R_0} |\p_r \n_\ep^{\alpha-1+\frac{1}{s_2}}(r)|^{s_2} r^{N-1} dr \right)^{\frac{1}{s_2}}
\left( \int_0^{R_0} r^{ -\frac{N-1}{s_2-1} } dr \right)^{\frac{s_2-1}{s_2}} \\
& \le C,
\ea\ee
which yields
\be\la{t54}\ba
\| \n_\ep \|_{L^\infty(\ep,R_0)} \le C.
\ea\ee
This, together with (\ref{cpt02}), gives (\ref{t05}), thereby completing the proof of Lemma \ref{tl5}.
\end{proof}

\section{Positive lower bound for the density of the approximate system}

In this section, we derive an $\ep$-dependent positive lower bound for the density of the approximate solution.
Let $(\n_\ep,\mathbf{u}_\ep)$ be the spherically symmetric classical solution to the approximate problem (\ref{bjns}), (\ref{bjqy}), (\ref{bjnxxs}), (\ref{bjbjtj1}), (\ref{bjbjtj2}) introduced in Section 3 on $\om_\ep \times (0,T]$.

Throughout this section, we assume that $\alpha$ and $\gamma$ satisfy (\ref{th110}) and (\ref{th11}).

\begin{lemma}\la{bjl4}
For any $2 < p < \mathcal{P}_N(\alpha)$, there exists a positive constant $C$ depending on $\ep$, $T$, and $p$ such that
\be\la{bj042}\ba
& \sup_{0 \le t \le T} \int_\ep^\infty \n_\ep |u_\ep|^p r^{N-1} dr
+ \int_0^T \int_\ep^\infty \left( \n_\ep^{\theta} \frac{|u_\ep|^p}{r^2} + \n_\ep^{\theta} |u_\ep|^{p-2} |\p_r u_\ep|^2 \right) r^{N-1} dr dt \le C,
\ea\ee
and
\be\la{bj05}\ba
& \sup_{0 \le t \le T} \int_\ep^\infty |\p_r \n_\ep^{\theta-1+\frac{1}{p}}|^{p} r^{N-1} dr \le C.
\ea\ee
\end{lemma}
\begin{proof}
For any $2 < p < \mathcal{P}_N(\alpha)$, we multiply $(\ref{bjnsqdc})_2$ by $|u_\ep|^{p-2} u_\ep r^{N-1}$ and integrate by parts over $(\ep,\infty)$ to obtain
\be\la{bj41}\ba
& \frac{1}{p} \frac{d}{dt} \int_\ep^\infty \n_\ep |u_\ep|^p r^{N-1} dr
+ \left(\alpha(N-1)^2-(N-1)(N-2)\right) \int_\ep^\infty \n_\ep^{\alpha} |u_\ep|^p r^{N-3} dr \\
& \quad + \alpha (p-1) \int_\ep^\infty \n_\ep^{\alpha} |u_\ep|^{p-2} |\p_r u_\ep|^2 r^{N-1} dr
+ \ep \theta (p-1) \int_\ep^\infty \n_\ep^{\theta} |u_\ep|^{p-2} |\p_r u_\ep|^2 r^{N-1} dr \\
& \quad + \ep \left(\theta(N-1)^2-(N-1)(N-2)\right) \int_\ep^\infty \n_\ep^{\theta} |u_\ep|^p r^{N-3} dr \\
& = p(N-1)(1-\alpha) \int_\ep^\infty \n_\ep^\alpha |u_\ep|^{p-2} u_\ep (\p_r u_\ep) r^{N-2} dr \\
& \quad + \ep p (N-1)(1-\theta) \int_\ep^\infty \n_\ep^\theta |u_\ep|^{p-2} u_\ep (\p_r u_\ep) r^{N-2} dr
- \int_\ep^\infty \p_r (\n_\ep^\ga) u_\ep |u_\ep|^{p-2} r^{N-1} dr.
\ea\ee

We now estimate the first two terms on the right-hand side of (\ref{bj41}).
We distinguish the cases $N=2$ and $N=3$.

\textit{Case 1: $N=2$.}
Using Young's inequality, we obtain for any sufficiently small $\delta>0$,
\be\la{bj43}\ba
p |1-\alpha| \n_\ep^\alpha |u_\ep|^{p-2} |u_\ep| |\p_r u_\ep|
\le (1-2\delta) \alpha \n_\ep^\alpha \frac{|u_\ep|^p}{r}
+ \frac{p^2(1-\alpha)^2}{4(1-2\delta)\alpha} \n_\ep^{\alpha} |u_\ep|^{p-2} |\p_r u_\ep|^2 r.
\ea\ee
Since $2 < p < \mathcal{P}_N(\alpha)$, the definition of $\mathcal{P}_2(\alpha)$ gives
\be\la{bj44}\ba
\frac{ p^2 - 2p \sqrt{p-1} }{(p-2)^2} < \alpha < \frac{ p^2 + 2p \sqrt{p-1} }{(p-2)^2},
\ea\ee
which yields
\be\la{bj45}
p^2 (1-\alpha)^2 < 4(p-1)\alpha^2.
\ee
Hence, we can choose sufficiently small $\delta>0$ such that
\be\la{bj46}\ba
\frac{p^2(1-\alpha)^2}{4(1-2\delta)\alpha} \le (1-2\delta) (p-1) \alpha.
\ea\ee
It follows from (\ref{bj43}) and (\ref{bj46}) that
\be\la{bj47}\ba
& p |1-\alpha| \n_\ep^\alpha |u_\ep|^{p-2} |u_\ep| |\p_r u_\ep| \\
& \le (1-2\delta) \alpha \n_\ep^\alpha \frac{|u_\ep|^p}{r}
+ (1-2\delta) (p-1) \alpha \n_\ep^{\alpha} |u_\ep|^{p-2} |\p_r u_\ep|^2 r.
\ea\ee
By (\ref{rgylt2}), we have $\varphi(\theta)=\varphi(\alpha)$ for $N=2$.
Thus, repeating the same argument with $\alpha$ replaced by $\theta$, we obtain
\be\la{bj472}\ba
& p |1-\theta| \n_\ep^\theta |u_\ep|^{p-2} |u_\ep| |\p_r u_\ep| \\
& \le (1-2\delta) \theta \n_\ep^\theta \frac{|u_\ep|^p}{r}
+ (1-2\delta) (p-1) \theta \n_\ep^{\theta} |u_\ep|^{p-2} |\p_r u_\ep|^2 r.
\ea\ee

\textit{Case 2: $N=3$.}
Young's inequality shows that for any sufficiently small $\delta>0$,
\be\la{bj48}\ba
& 2p |1-\alpha| \n_\ep^\alpha |u_\ep|^{p-2} |u_\ep| |\p_r u_\ep| r \\
& \le (1-2\delta) (4\alpha-2) \n_\ep^\alpha |u_\ep|^p
+ \frac{p^2(1-\alpha)^2}{(1-2\delta)(4\alpha-2)} \n_\ep^{\alpha} |u_\ep|^{p-2} |\p_r u_\ep|^2 r^2.
\ea\ee
Using $2 < p < \mathcal{P}_N(\alpha)$, we get
\be\la{bj49}\ba
\frac{p^2-p+1-\sqrt{2p^3-p^2-2p+1} }{(p-2)^2} < \alpha < \frac{p^2-p+1+\sqrt{2p^3-p^2-2p+1} }{(p-2)^2},
\ea\ee
which implies that
\be\la{bj410}
p^2 (1-\alpha)^2 < (4\alpha-2)(p-1)\alpha.
\ee
Consequently, we may choose sufficiently small $\delta>0$ such that
\be\la{bj411}\ba
\frac{p^2(1-\alpha)^2}{(1-2\delta)(4\alpha-2)} \le (1-2\delta) (p-1) \alpha.
\ea\ee
Combining (\ref{bj48}) and (\ref{bj411}), we arrive at
\be\la{bj412}\ba
& 2p |1-\alpha| \n_\ep^\alpha |u_\ep|^{p-2} |u_\ep| |\p_r u_\ep| r \\
& \le (1-2\delta) (4\alpha-2) \n_\ep^\alpha |u_\ep|^p
+ (1-2\delta) (p-1) \alpha \n_\ep^{\alpha} |u_\ep|^{p-2} |\p_r u_\ep|^2 r^2.
\ea\ee
A similar argument also yields
\be\la{bj413}\ba
& 2p |1-\theta| \n_\ep^\theta |u_\ep|^{p-2} |u_\ep| |\p_r u_\ep| r \\
& \le (1-2\delta) (4\theta-2) \n_\ep^\theta |u_\ep|^p
+ (1-2\delta) (p-1) \theta \n_\ep^{\theta} |u_\ep|^{p-2} |\p_r u_\ep|^2 r^2.
\ea\ee
In view of (\ref{bj41}), (\ref{bj47}), (\ref{bj472}), (\ref{bj412}), and (\ref{bj413}), we obtain
\be\la{bj414}\ba
& \frac{1}{p} \frac{d}{dt} \int_\ep^\infty \n_\ep |u_\ep|^p r^{N-1} dr
+ 2 \delta \ep \theta (p-1) \int_\ep^\infty \n_\ep^{\theta} |u_\ep|^{p-2} |\p_r u_\ep|^2 r^{N-1} dr \\
& \quad + 2 \delta \ep \left(\theta(N-1)^2-(N-1)(N-2)\right) \int_\ep^\infty \n_\ep^{\theta} |u_\ep|^p r^{N-3} dr \\
& \le - \int_\ep^\infty \p_r (\n_\ep^\ga) u_\ep |u_\ep|^{p-2} r^{N-1} dr.
\ea\ee
Let $\eta(r)$ be the smooth cutoff function satisfying (\ref{bj415}).
Integration by parts together with (\ref{cpt02}), (\ref{bj415}), and Young's inequality yields
\be\la{bj416}\ba
& - \int_\ep^\infty \p_r (\n_\ep^\ga) u_\ep |u_\ep|^{p-2} r^{N-1} dr \\
& = - \int_\ep^\infty \p_r (\n_\ep^\ga - \eta(r)) u_\ep |u_\ep|^{p-2} r^{N-1} dr
- \int_\ep^\infty \p_r (\eta(r)) u_\ep |u_\ep|^{p-2} r^{N-1} dr \\
& = (p-1) \int_\ep^\infty (\n_\ep^\ga - \eta(r)) |u_\ep|^{p-2} (\p_r u_\ep) r^{N-1} dr \\
& \quad + (N-1) \int_\ep^\infty (\n_\ep^\ga - \eta(r)) u_\ep |u_\ep|^{p-2} r^{N-2} dr
- \int_{R_0}^{2 R_0} \eta'(r) u_\ep |u_\ep|^{p-2} r^{N-1} dr \\
& \le \frac{\delta \ep}{8} \left(\theta(N-1)^2-(N-1)(N-2)\right) \int_\ep^\infty \n_\ep^{\theta} |u_\ep|^p r^{N-3} dr \\
& \quad + \frac{\delta \ep}{8} \theta (p-1) \int_\ep^\infty \n_\ep^{\theta} |u_\ep|^{p-2} |\p_r u_\ep|^2 r^{N-1} dr
+ C \int_\ep^\infty \n_\ep^{-\theta} |\n_\ep^\ga - \eta(r)|^2 |u_\ep|^{p-2} r^{N-1} dr \\
& \quad + C + C \int_{R_0}^{2 R_0} \n_\ep^\theta |u_\ep|^{p} r^{N-1} dr.
\ea\ee
Moreover, by (\ref{cpt01}), (\ref{cpt001}), (\ref{cpt02}), (\ref{bj415}), and Young's inequality, we have
\be\la{bj417}\ba
& C \int_\ep^\infty \n_\ep^{-\theta} |\n_\ep^\ga - \eta(r)|^2 |u_\ep|^{p-2} r^{N-1} dr \\
& = C \int_\ep^{R_0} \n_\ep^{-\theta} |\n_\ep^\ga - \eta(r)|^2 |u_\ep|^{p-2} r^{N-1} dr
+ C \int_{R_0}^{2 R_0} \n_\ep^{-\theta} |\n_\ep^\ga - \eta(r)|^2 |u_\ep|^{p-2} r^{N-1} dr \\
& \quad + C \int_{2 R_0}^\infty \n_\ep^{-\theta} |\n_\ep^\ga - \eta(r)|^2 |u_\ep|^{p-2} r^{N-1} dr \\
& \le C \int_\ep^{R_0} \n_\ep^{2\ga-\theta} |u_\ep|^{p-2} r^{N-1} dr
+ C \int_{R_0}^{2 R_0} \n_\ep ( 1 + |u_\ep|^{p} ) r^{N-1} dr \\
& \quad + C \int_{2 R_0}^\infty |\n_\ep - 1|^2 |u_\ep|^{p-2} r^{N-1} dr \\
& \le C \int_\ep^{R_0} \n_\ep^{2\ga-\theta} |u_\ep|^{p-2} r^{N-1} dr
+ C + C \int_{R_0}^{2 R_0} \n_\ep |u_\ep|^{p} r^{N-1} dr \\
& \quad + C \int_{2 R_0}^\infty |\n_\ep^{\theta-\frac{1}{2}} - 1|^2 ( 1 + |u_\ep|^p ) r^{N-1} dr \\
& \le \frac{\delta \ep}{8} \left(\theta(N-1)^2-(N-1)(N-2)\right) \int_\ep^\infty \n_\ep^{\theta} |u_\ep|^p r^{N-3} dr \\
& \quad + C + C \int_\ep^{R_0} \n_\ep^{p(\ga-\theta)+\theta} r^{p+N-3} dr
+ C \int_\ep^\infty \n_\ep |u_\ep|^{p} r^{N-1} dr \\
& \le \frac{\delta \ep}{8} \left(\theta(N-1)^2-(N-1)(N-2)\right) \int_\ep^\infty \n_\ep^{\theta} |u_\ep|^p r^{N-3} dr
+ C + C \int_\ep^\infty \n_\ep |u_\ep|^{p} r^{N-1} dr.
\ea\ee
Combining (\ref{bj416}) and (\ref{bj417}) leads to
\be\la{bj417a}\ba
& - \int_\ep^\infty \p_r (\n_\ep^\ga) u_\ep |u_\ep|^{p-2} r^{N-1} dr \\
& \le \frac{\delta \ep}{4} \left(\theta(N-1)^2-(N-1)(N-2)\right) \int_\ep^\infty \n_\ep^{\theta} |u_\ep|^p r^{N-3} dr \\
& \quad + \frac{\delta \ep}{8} \theta (p-1) \int_\ep^\infty \n_\ep^{\theta} |u_\ep|^{p-2} |\p_r u_\ep|^2 r^{N-1} dr
+ C + C \int_\ep^\infty \n_\ep |u_\ep|^{p} r^{N-1} dr.
\ea\ee
Putting (\ref{bj417a}) into (\ref{bj414}) gives
\be\la{bj418}\ba
& \frac{1}{p} \frac{d}{dt} \int_\ep^\infty \n_\ep |u_\ep|^p r^{N-1} dr
+ \delta \ep \theta (p-1) \int_\ep^\infty \n_\ep^{\theta} |u_\ep|^{p-2} |\p_r u_\ep|^2 r^{N-1} dr \\
& + \delta \ep \left(\theta(N-1)^2-(N-1)(N-2)\right) \int_\ep^\infty \n_\ep^{\theta} |u_\ep|^p r^{N-3} dr \\
& \le C + C \int_\ep^\infty \n_\ep |u_\ep|^{p} r^{N-1} dr.
\ea\ee
Applying Gr\"onwall's inequality to (\ref{bj418}) yields (\ref{bj042}).

Moreover, from (\ref{bjyxsd}) and (\ref{bjyxsdfc1}), we deduce that the effective velocity $w_\ep$ satisfies
\be\la{bj51}\ba
\n_\ep (w_\ep)_t + \n_\ep u_\ep (w_\ep)_r
+ \frac{\ga \n_\ep^{\ga}}{ \alpha \n_\ep^{\alpha-1} + \ep \theta \n_\ep^{\theta-1} } (w_\ep - u_\ep) = 0.
\ea\ee
For any $2 < p < \mathcal{P}_N(\alpha)$, we multiply (\ref{bj51}) by $|w_\ep|^{p-2} w_\ep r^{N-1}$, integrate over $(\ep,\infty)$, and use $(\ref{bjnsqdc})_1$ and Young's inequality to derive
\be\la{bj52}\ba
& \frac{1}{p} \frac{d}{dt} \int_\ep^\infty \n_\ep |w_\ep|^{p} r^{N-1} dr
+ \int_\ep^\infty \frac{\ga \n_\ep^{\ga}}{ \alpha \n_\ep^{\alpha-1} + \ep \theta \n_\ep^{\theta-1} } |w_\ep|^{p} r^{N-1} dr \\
& = \int_\ep^\infty \frac{\ga \n_\ep^{\ga}}{ \alpha \n_\ep^{\alpha-1} + \ep \theta \n_\ep^{\theta-1} } |w_\ep|^{p-2} w_\ep u_\ep r^{N-1} dr \\
& \le \frac{1}{2} \int_\ep^\infty \frac{\ga \n_\ep^{\ga}}{ \alpha \n_\ep^{\alpha-1} + \ep \theta \n_\ep^{\theta-1} } |w_\ep|^{p} r^{N-1} dr
+ C \int_\ep^\infty \n_\ep^{\ga+1-\theta} |u_\ep|^{p} r^{N-1} dr,
\ea\ee
which together with (\ref{t05}) gives
\be\la{bj52a}\ba
\frac{1}{p} \frac{d}{dt} \int_\ep^\infty \n_\ep |w_\ep|^{p} r^{N-1} dr
& \le C \int_\ep^{\infty} \n_\ep^{\ga+1-\theta} |u_\ep|^{p} r^{N-1} dr \\
& \le C \int_\ep^{\infty} \n_\ep |u_\ep|^{p} r^{N-1} dr \le C,
\ea\ee
where we have used $\ga > 1 \ge \theta$.

Integrating (\ref{bj52a}) over $(0,T)$ yields
\be\la{bj53}\ba
\sup_{0 \le t \le T} \int_\ep^\infty \n_\ep |w_\ep|^{p} r^{N-1} dr \le C.
\ea\ee
In addition, the definition of $w_\ep$ and Young's inequality give
\be\la{bj54}\ba
\n_\ep^{1-p} (\alpha \n_\ep^{\alpha-1} + \ep \theta \n_\ep^{\theta-1})^{p} |\p_r \n_\ep|^{p}
\le C \n_\ep |u_\ep|^p + C \n_\ep |w_\ep|^p.
\ea\ee

Combining (\ref{bj53}), (\ref{bj54}), and (\ref{bj042}), we obtain (\ref{bj05}), thereby finishing the proof of Lemma \ref{bjl4}.
\end{proof}

Using Lemmas \ref{tl5} and \ref{bjl4}, and adapting the arguments in \cite{CZZ1,CZZ2}, we can obtain the following positive lower bound for the density of the approximate solution.

\begin{lemma}\la{1ee3}
Let $N=2$ or $N=3$, and let $\alpha \le 1$.
There exists a positive constant $C$ depending on $\ep$ and $T$ such that
\be\la{1e03}\ba
\sup_{0 \le t \le T} \| \n_\ep^{-1} \|_{L^\infty(\ep,\infty)} \le C.
\ea\ee
\end{lemma}

\begin{lemma}\la{2ee2}
Let $N=2$ and $\alpha >1$.
There exists a positive constant $C$ depending on $\ep$ and $T$ such that
\be\la{2e02}\ba
\sup_{0 \le t \le T} \| \n_\ep^{-1} \|_{L^\infty(\ep,\infty)} \le C.
\ea\ee
\end{lemma}
\begin{proof}
First, a direct calculation shows that for any $1<\alpha<\infty$,
\be\la{2e21}\ba
\varphi(\alpha) - \frac{2\alpha-1}{\alpha-1}
= \frac{ 2\alpha^2 + 2\alpha\sqrt{2\alpha-1} }{(\alpha-1)^2} - \frac{2\alpha-1}{\alpha-1}
= \frac{ 3\alpha - 1 + 2\alpha\sqrt{2\alpha-1} }{(\alpha-1)^2} > 0.
\ea\ee
Thus, we may choose $s_0$ such that
\be\la{2e22}\ba
2 < \frac{2\alpha-1}{\alpha-1} < s_0 < \varphi(\alpha),
\ea\ee
which implies that
\be\la{2e23}\ba
\frac{\alpha}{2\alpha-1}+\frac{1}{s_0} < 1.
\ea\ee
Choosing $p=s_0$ in (\ref{bj05}), we deduce that
\be\la{2e29}\ba
\int_\ep^\infty | \p_r \n_\ep^{\theta-1+\frac{1}{s_0}} |^{s_0} r dr \le C(\ep).
\ea\ee
Let $R_0$ be as in Lemma \ref{cptl2}, and set
\be\la{2e29a1}\ba
y_0(t) \triangleq \int_\ep^{2 R_0} \rho_\ep(s,t) s ds.
\ea\ee
On the one hand, (\ref{cpt02}) implies that
\be\la{2e29a2}\ba
y_0(t) = \int_\ep^{2 R_0} \rho_\ep(s,t) s ds > \int_{R_0}^{2 R_0} \rho_\ep(s,t) s ds \ge c_1 R_0^2.
\ea\ee
On the other hand, (\ref{t05}) ensures that
\be\la{2e29a3}\ba
y_0(t) = \int_\ep^{2 R_0} \rho_\ep(s,t) s ds \le C.
\ea\ee
In view of (\ref{2e29a1}), (\ref{2e29a2}), and (\ref{2e29a3}), we have
\be\la{2e29a4}\ba
c_1 R_0^2 \le y_0 \le C.
\ea\ee

For any $r\in(\ep,2R_0)$, with $y$ defined by (\ref{lzb1}), we use (\ref{ywsi2}), (\ref{lzb1}), (\ref{lzb2}), (\ref{2e21}), (\ref{2e29}), (\ref{2e29a4}), and H\"older's inequality to obtain
\be\la{2e210}\ba
\n_\ep^{-1}(r,t) = \n_\ep^{-1}(y,t)
& \le \frac{1}{y_0} \int_0^{y_0} \n_\ep^{-1} dy
+ \int_0^{y_0} |\p_y(\n_\ep^{-1})| dy \\
& \le \frac{1}{c_1 R_0^2} \int_\ep^{2 R_0} \n_\ep^{-1} \n_\ep r dr + C \int_\ep^{2 R_0} \n_\ep^{-2} |\p_r \n_\ep| dr \\
& \le C + C \| \n_\ep^{-1} \|^{\theta+\frac{1}{s_0}}_{L^\infty(\ep,2 R_0)}
\left( \int_\ep^{2 R_0} | \p_r \n_\ep^{\theta-1+\frac{1}{s_0}} |^{s_0} r dr \right)^{\frac{1}{s_0}}
\left( \int_\ep^{2 R_0} r^{- \frac{1}{s_0-1}} dr \right)^{\frac{s_0-1}{s_0}} \\
& \le C + C(\ep) \| \n_\ep^{-1} \|^{\frac{\alpha}{2\alpha-1}+\frac{1}{s_0}}_{L^\infty(\ep,2 R_0)},
\ea\ee
where in the last inequality we have used
\be\la{2e24}\ba
\theta = \frac{\alpha}{2\alpha-1}.
\ea\ee
By (\ref{2e23}) and Young's inequality, we arrive at
\be\la{2e211}\ba
\sup_{0 \le t \le T} \| \n_\ep^{-1}(r,t) \|_{L^\infty(\ep,2 R_0)} \le C(\ep),
\ea\ee
which together with (\ref{cpt02}) yields (\ref{2e02}) and completes the proof of Lemma \ref{2ee2}.
\end{proof}

\begin{lemma}\la{2ee4}
Let $N=3$ and $1 < \alpha \le 2$.
There exists a positive constant $C$ depending on $\ep$ and $T$ such that
\be\la{2e04}\ba
\sup_{0 \le t \le T} \| \n_\ep^{-1} \|_{L^\infty(\ep,\infty)} \le C.
\ea\ee
\end{lemma}
\begin{proof}
First, for any $1<\alpha \le 2$, set
\be\nonumber\ba
\tilde{q}_1 \triangleq 3+\frac{2}{\alpha-1}.
\ea\ee
A direct computation shows that $\tilde{q}_1 \in (3,\psi(\alpha))$ and
\be\la{2e41}\ba
\alpha > \frac{\tilde{q}_1-2}{\tilde{q}_1-3}.
\ea\ee
Recall from (\ref{rgylt1}) that for $N=3$ and $1 < \alpha \le 2$,
\be\nonumber\ba
\theta = \frac{2\alpha-1}{3\alpha-2},
\ea\ee
which together with (\ref{2e41}) yields
\be\la{2e41a}\ba
\theta + \frac{1}{\tilde{q}_1} < 1.
\ea\ee
Choosing $p=\tilde{q}_1$ in (\ref{bj05}), we obtain
\be\la{2e41b}\ba
\int_\ep^\infty | \p_r \n_\ep^{\theta-1+\frac{1}{\tilde{q}_1}} |^{\tilde{q}_1} r^2 dr \le C(\ep).
\ea\ee

Let $R_0$ be as in Lemma \ref{cptl2}, and define
\be\la{2e42a1}\ba
y_1(t) \triangleq \int_\ep^{2 R_0} \rho_\ep(s,t) s^2 ds.
\ea\ee
By (\ref{cpt02}), we have
\be\la{2e42a2}\ba
\int_\ep^{2 R_0} \rho_\ep(s,t) s^2 ds > \int_{R_0}^{2 R_0} \rho_\ep(s,t) s^2 ds \ge c_1 R_0^3.
\ea\ee
It follows from (\ref{t05}) that
\be\la{2e42a3}\ba
\int_\ep^{2 R_0} \rho_\ep(s,t) s^2 ds \le C.
\ea\ee
Using (\ref{2e42a1}), (\ref{2e42a2}), and (\ref{2e42a3}), we arrive at
\be\la{2e42a4}\ba
c_1 R_0^3 \le y_1 \le C.
\ea\ee
For any $r\in(\ep,2 R_0)$ with $y$ defined by (\ref{lzb1}), we deduce from (\ref{ywsi2}), (\ref{lzb1}), (\ref{lzb2}), and (\ref{2e41b}) that
\be\la{2e42}\ba
\n_\ep^{-1}(r,t) = \n_\ep^{-1}(y,t)
& \le \frac{1}{y_1} \int_0^{y_1} \n_\ep^{-1} dy
+ \int_0^{y_1} |\p_y(\n_\ep^{-1})| dy \\
& \le \frac{1}{c_1 R_0^3} \int_\ep^{2 R_0} \n_\ep^{-1} \n_\ep r^2 dr + \int_\ep^{2 R_0} |\p_r(\n_\ep^{-1})| dr \\
& \le C + C \int_\ep^{2 R_0} \n_\ep^{-2} |\p_r \n_\ep| dr \\
& \le C + C \| \n_\ep^{-1} \|^{\theta+\frac{1}{\tilde{q}_1}}_{L^\infty(\ep,2 R_0)}
\left( \int_\ep^{2 R_0} | \p_r \n_\ep^{\theta-1+\frac{1}{\tilde{q}_1}} |^{\tilde{q}_1} r^2 dr \right)^{\frac{1}{\tilde{q}_1}}
\left( \int_\ep^{2 R_0} r^{- \frac{2}{\tilde{q}_1-1}} dr \right)^{\frac{\tilde{q}_1-1}{\tilde{q}_1}} \\
& \le C + C \| \n_\ep^{-1} \|^{\theta + \frac{1}{\tilde{q}_1}}_{L^\infty(\ep,2 R_0)},
\ea\ee
which together with (\ref{cpt02}), (\ref{2e41a}), and Young's inequality implies that
\be\la{2e43}\ba
\sup_{0 \le t \le T} \| \n_\ep^{-1}(r,t) \|_{L^\infty(\ep,\infty)} \le C.
\ea\ee
This gives (\ref{2e04}) and completes the proof of Lemma \ref{2ee4}.
\end{proof}

\begin{lemma}\la{2ee5}
Let $N=3$ and $2 < \alpha < 7 + 2 \sqrt{10}$.
There exists a positive constant $C$ depending on $\ep$ and $T$ such that
\be\la{2e05}\ba
\sup_{0 \le t \le T} \| \n_\ep^{-1} \|_{L^\infty(\ep,\infty)} \le C.
\ea\ee
\end{lemma}
\begin{proof}
First, we conclude from (\ref{bj41}) that for any $p>2$,
\be\la{2e51}\ba
& \frac{1}{p} \frac{d}{dt} \int_\ep^\infty \n_\ep |u_\ep|^p r^{2} dr
+ (4\alpha - 2) \int_\ep^\infty \n_\ep^{\alpha} |u_\ep|^p dr
+ \alpha (p-1) \int_\ep^\infty \n_\ep^{\alpha} |u_\ep|^{p-2} |\p_r u_\ep|^2 r^{2} dr \\
& \quad + \ep (4\theta - 2) \int_\ep^\infty \n_\ep^{\theta} |u_\ep|^p dr + \ep \theta (p-1) \int_\ep^\infty \n_\ep^{\theta} |u_\ep|^{p-2} |\p_r u_\ep|^2 r^{2} dr \\
& = 2p(1-\alpha) \int_\ep^\infty \n_\ep^\alpha |u_\ep|^{p-2} u_\ep (\p_r u_\ep) r dr
+ 2\ep p(1-\theta) \int_\ep^\infty \n_\ep^\theta |u_\ep|^{p-2} u_\ep (\p_r u_\ep) r dr \\
& \quad - \int_\ep^\infty \p_r (\n_\ep^\ga) u_\ep |u_\ep|^{p-2} r^{2} dr.
\ea\ee
By Young's inequality, we have
\be\la{2e52}\ba
& 2p |1-\alpha| \int_\ep^\infty \n_\ep^\alpha |u_\ep|^{p-2} |u_\ep| |\p_r u_\ep| r dr \\
& \le C \int_\ep^\infty \n_\ep^\alpha |u_\ep|^p dr + \frac{\alpha (p-1)}{4} \int_\ep^\infty \n_\ep^{\alpha} |u_\ep|^{p-2} |\p_r u_\ep|^2 r^2 dr \\
& \le C(\ep) \int_\ep^\infty \n_\ep |u_\ep|^p r^2 dr + \frac{\alpha (p-1)}{4} \int_\ep^\infty \n_\ep^{\alpha} |u_\ep|^{p-2} |\p_r u_\ep|^2 r^2 dr.
\ea\ee

For any $p \in (2,\psi(\theta))$, the argument leading to (\ref{bj48})--(\ref{bj412}), with $\alpha$ replaced by $\theta$, shows that there exists a sufficiently small $\delta>0$ such that
\be\la{2e53}\ba
& 2p |1-\theta| \n_\ep^\theta |u_\ep|^{p-2} |u_\ep| |\p_r u_\ep| r \\
& \le (1-\delta) (4\theta-2) \n_\ep^\theta |u_\ep|^p
+ (1-\delta) (p-1) \theta \n_\ep^{\theta} |u_\ep|^{p-2} |\p_r u_\ep|^2 r^2.
\ea\ee
Moreover, from (\ref{bj417a}), we conclude that
\be\la{2e54}\ba
- \int_\ep^\infty \p_r (\n_\ep^\ga) u_\ep |u_\ep|^{p-2} r^{2} dr
& \le \frac{\de \theta \ep (p-1)}{8} \int_\ep^\infty \n_\ep^{\theta} |u_\ep|^{p-2} |\p_r u_\ep|^2 r^{2} dr \\
& \quad + \frac{\delta \ep}{4} (4\theta-2) \int_\ep^\infty \n_\ep^{\theta} |u_\ep|^p dr
+ C + C \int_\ep^\infty \n_\ep |u_\ep|^{p} r^{2} dr.
\ea\ee
Substituting (\ref{2e52}), (\ref{2e53}), and (\ref{2e54}) into (\ref{2e51}) and using (\ref{t05}), we obtain for any $p \in (2,\psi(\theta))$,
\be\la{2e55}\ba
& \frac{1}{p} \frac{d}{dt} \int_\ep^\infty \n_\ep |u_\ep|^p r^{2} dr
+ \frac{\delta \ep}{4} (4\theta - 2) \int_\ep^\infty \n_\ep^{\theta} |u_\ep|^p dr
+ \frac{\de \theta \ep (p-1)}{4} \int_\ep^\infty \n_\ep^{\theta} |u_\ep|^{p-2} |\p_r u_\ep|^2 r^{2} dr \\
& \le C + C(\ep) \int_\ep^\infty \n_\ep |u_\ep|^p r^{2} dr.
\ea\ee
Applying Gr\"onwall's inequality to (\ref{2e55}) yields for any $p \in (2,\psi(\theta))$,
\be\la{2e56}\ba
\sup_{0 \le t \le T} \int_\ep^\infty \n_\ep |u_\ep|^p r^{2} dr
+ \int_0^T \int_\ep^\infty \left( \n_\ep^{\theta} |u_\ep|^p + \n_\ep^{\theta} |u_\ep|^{p-2} |\p_r u_\ep|^2 r^{2} \right) dr dt
\le C(\ep).
\ea\ee

On the other hand, for any $p \in (2,\psi(\theta))$, multiplying (\ref{bj51}) by $|w_\ep|^{p-2} w_\ep r^{N-1}$, integrating over $(\ep,\infty)$, and using $(\ref{bjnsqdc})_1$ and Young's inequality, we derive
\be\la{2e560}\ba
& \frac{1}{p} \frac{d}{dt} \int_\ep^\infty \n_\ep |w_\ep|^{p} r^{N-1} dr
+ \int_\ep^\infty \frac{\ga \n_\ep^{\ga}}{ \alpha \n_\ep^{\alpha-1} + \ep \theta \n_\ep^{\theta-1} } |w_\ep|^{p} r^{N-1} dr \\
& = \int_\ep^\infty \frac{\ga \n_\ep^{\ga}}{ \alpha \n_\ep^{\alpha-1} + \ep \theta \n_\ep^{\theta-1} } |w_\ep|^{p-2} w_\ep u_\ep r^{N-1} dr \\
& \le \frac{1}{2} \int_\ep^\infty \frac{\ga \n_\ep^{\ga}}{ \alpha \n_\ep^{\alpha-1} + \ep \theta \n_\ep^{\theta-1} } |w_\ep|^{p} r^{N-1} dr
+ C \int_\ep^\infty \n_\ep^{\ga+1-\theta} |u_\ep|^{p} r^{N-1} dr.
\ea\ee
Integrating (\ref{2e560}) over $(0,T)$ and using (\ref{t05}) and $\ga > 1 \ge \theta$, we arrive at
\be\la{2e57}\ba
\sup_{0 \le t \le T} \int_\ep^\infty \n_\ep |w_\ep|^{p} r^{N-1} dr
& \le C + C \int_0^T \int_\ep^{\infty} \n_\ep^{\ga+1-\theta} |u_\ep|^{p} r^{N-1} dr dt \\
& \le C + C \int_0^T \int_\ep^{\infty} \n_\ep |u_\ep|^{p} r^{N-1} dr dt \le C.
\ea\ee
Combining (\ref{bj54}), (\ref{2e56}), and (\ref{2e57}) leads to
\be\la{2e58}\ba
\sup_{0 \le t \le T} \int_\ep^\infty |\p_r \n_\ep^{\theta-1+\frac{1}{p}}|^{p} r^{2} dr \le C(\ep).
\ea\ee
Since for $N=3$ and $\alpha>2$, we have $\theta=\frac{3}{4}$ and $\psi(\frac{3}{4}) = 6+2\sqrt{6}>4$, we may choose $\tilde{q}_2 \in (4,\psi(\frac{3}{4}))$.
Taking $p=\tilde{q}_2$ in (\ref{2e58}) gives
\be\la{2e59}\ba
\sup_{0 \le t \le T} \int_\ep^\infty |\p_r \n_\ep^{\theta-1+\frac{1}{\tilde{q}_2}}|^{\tilde{q}_2} r^{2} dr \le C.
\ea\ee
Let $y_1$ be as in (\ref{2e42a1}) and $y$ be defined by (\ref{lzb1}).
Then for any $r\in(\ep,2 R_0)$, we deduce from (\ref{ywsi2}), (\ref{lzb1}), (\ref{lzb2}), (\ref{2e42a4}), and (\ref{2e59}) that
\be\la{2e510}\ba
\n_\ep^{-1}(r,t) = \n_\ep^{-1}(y,t)
& \le \frac{1}{y_1} \int_0^{y_1} \n_\ep^{-1} dy
+ \int_0^{y_1} |\p_y(\n_\ep^{-1})| dy \\
& \le \frac{1}{c_1 R_0^3} \int_\ep^{2 R_0} \n_\ep^{-1} \n_\ep r^2 dr + \int_\ep^{2 R_0} |\p_r(\n_\ep^{-1})| dr \\
& \le C + C \int_\ep^{2 R_0} \n_\ep^{-2} |\p_r \n_\ep| dr \\
& \le C + C \| \n_\ep^{-1} \|^{\theta+\frac{1}{\tilde{q}_2}}_{L^\infty(\ep,2 R_0)}
\left( \int_\ep^{2 R_0} |\p_r \n_\ep^{\theta-1+\frac{1}{\tilde{q}_2}}|^{\tilde{q}_2} r^2 dr \right)^{\frac{1}{\tilde{q}_2}}
\left( \int_\ep^{2 R_0} r^{- \frac{2}{\tilde{q}_2-1}} dr \right)^{\frac{\tilde{q}_2-1}{\tilde{q}_2}} \\
& \le C + C \| \n_\ep^{-1} \|^{\frac{3}{4} + \frac{1}{\tilde{q}_2}}_{L^\infty(\ep,2 R_0)}.
\ea\ee
This, along with Young's inequality and the fact that $\tilde{q}_2>4$, implies that
\be\la{2e511}\ba
\sup_{0 \le t \le T} \| \n_\ep^{-1}(r,t) \|_{L^\infty(\ep,2 R_0)} \le C,
\ea\ee
which together with (\ref{cpt02}) gives (\ref{2e05}) and completes the proof of Lemma \ref{2ee5}.
\end{proof}

\section{Global smooth solutions to the approximate system and compactness results}

Throughout this section, it will always be assumed that $\alpha$ and $\ga$ satisfy the conditions stated in Theorem \ref{th1}.
The aim of this section is to prove the global existence of smooth solutions to the approximate system and to establish the compactness results required for passing to the limit.

\subsection{Global smooth solutions to the approximate system}

For any fixed $\ep \in (0,1)$, combining the upper and strictly positive lower bounds for the density established in Sections 4 and 5 with the standard arguments in \cite{CZZ1,CZZ2,Lei}, we can derive the higher-order estimates.
Consequently, the local classical solution $(\n_\ep,\mathbf{u}_\ep)$ can be extended globally in time.

\subsection{Compactness results}

In this subsection, we adapt the compactness arguments in \cite{GJX,LX1,MV} to prove that, up to a subsequence, the limit $(\n,\sqrt{\n} \mathbf{u})$ of $(\n_\ep,\sqrt{\n_\ep} \mathbf{u}_\ep)$ is a weak solution to (\ref{ns}).

Since the approximate solutions $(\n_\ep(r),u_\ep(r))$ are defined only on $(\ep,\infty)$, we extend them continuously to $[0,\infty)$ by
\be\la{cr1}\ba
\tilde{\n}_{\ep}(r,t) \triangleq
\begin{cases}
\n_{\ep}(r,t) \quad & \text{ if } r \in (\ep,\infty), \\
\n_{\ep}(2\ep-r,t) \quad & \text{ if } r \in [0,\ep], \\
\end{cases}
\ea\ee
and
\be\la{cr2}\ba
\tilde{u}_{\ep}(r,t) \triangleq
\begin{cases}
u_{\ep}(r,t) \quad & \text{ if } r \in (\ep,\infty), \\
0 \quad & \text{ if } r \in [0,\ep].
\end{cases}
\ea\ee
The corresponding velocity field on $\mathbb{R}^N$ is defined by
\be\la{cr2a}\ba
\tilde{ \mathbf{u} }_{\ep}(x,t) \triangleq
\begin{cases}
\tilde{u}_{\ep}(r,t) \frac{x}{r} \quad & \text{ if } x \neq 0, \\
0 \quad & \text{ if } x = 0.
\end{cases}
\ea\ee
For simplicity, we denote the approximate solutions $( \tilde{\n}_{\ep}, \tilde{u}_{\ep},\tilde{ \mathbf{u} }_{\ep} )$ by $( \n_\ep,u_\ep,\mathbf{u}_{\ep} )$.

The boundary condition $u_\ep(\ep,t) = 0$, together with the estimates established in Section 4, implies the following bounds, in which $C=C(T)>0$ is independent of $\ep$:
\be\la{cr3}\ba
& \sup_{0 \le t \le T} \left( \| \n_{\ep} - 1 \|_{L^2(\mathbb{R}^N) \cap L^\infty(\mathbb{R}^N)} + \| \na \n_{\ep}^{\alpha-\frac{1}{2}} \|_{L^2(\mathbb{R}^N)} + \| \na \n_{\ep}^{\alpha-1+\frac{1}{q}} \|_{L^q(\mathbb{R}^N)} \right) \\
& \quad + \int_0^T \left( \| \na \n_{\ep}^{\frac{\ga+\alpha-1}{2}} \|^2_{L^2(\mathbb{R}^N)} + \| \na \n_{\ep}^{\alpha-1+\frac{\ga-\alpha+1}{q}} \|^q_{L^q(\mathbb{R}^N)} \right) dt \le C,
\ea\ee
and
\be\la{cr4}\ba
\sup_{0 \le t \le T} \int_{\mathbb{R}^N} \left( \n_\ep |\mathbf{u}_\ep|^2 + \n_\ep |\mathbf{u}_\ep|^p \right) dx
+ \int_0^T \int_{\mathbb{R}^N} \left( \n_\ep^\alpha |\na \mathbf{u}_\ep|^2 + \ep \n_\ep^\theta |\na \mathbf{u}_\ep|^2 \right) dx dt \le C,
\ea\ee
for some $p$ and $q$ satisfying
\be\nonumber\ba
N < q \le p < \mathcal{P}_N(\alpha), \quad \alpha-1+\frac{1}{q}>0.
\ea\ee

We begin with the convergence of $\n_\ep$.
\begin{lemma}\la{jxl1}
There exists a function $\n \in L^\infty(\mathbb{R}^N \times (0,T)) \cap C(\mathbb{R}^N \times (0,T))$ such that up to a subsequence for any $M>0$, 
\be\la{jx01}\ba
\n_\ep \to \n \textnormal{ in } L^\infty(B_M \times (0,T)).
\ea\ee
In particular,
\be\la{jx01a}\ba
\n_\ep \to \n \textnormal{ almost everywhere in } \mathbb{R}^N \times (0,T).
\ea\ee
Moreover,
\be\la{jx001}\ba
\na \n_\ep^{\alpha-\frac{1}{2}} \rightharpoonup \na \n^{\alpha-\frac{1}{2}} \textnormal{ in } L^\infty(0,T;L^2(\mathbb{R}^N)),
\ea\ee
and
\be\la{jx0001}\ba
\na \n_\ep^{\alpha-1+\frac{1}{q}} \rightharpoonup \na \n^{\alpha-1+\frac{1}{q}} \textnormal{ in } L^\infty(0,T;L^q (\mathbb{R}^N)).
\ea\ee
\end{lemma}
\begin{proof}
First, from (\ref{bjns}), we conclude that
\be\la{jx11}\ba
\p_t(\n_\ep^\alpha) & = - \mathbf{u}_\ep \cdot \na \n_\ep^\alpha - \alpha \n_\ep^\alpha \div \mathbf{u}_\ep \\
& = - \frac{2 \alpha}{2 \alpha - 1} \sqrt{\n_\ep} \mathbf{u}_\ep \cdot \na \n_\ep^{\alpha-\frac{1}{2}}
- \alpha \n_\ep^{\frac{\alpha}{2}} \n_\ep^{\frac{\alpha}{2}} \div \mathbf{u}_\ep \quad \text{ in } \om_\ep,
\ea\ee
which together with (\ref{cr3}) and (\ref{cr4}) yields for any $M > 1$,
\be\la{jx12}\ba
\p_t(\n_\ep^\alpha) \text{ is bounded in } L^2( 0,T;L^1(B_M \setminus B_\ep) ).
\ea\ee
Moreover, by (\ref{cr1}) and a change of variables, we arrive at
\be\la{jx13}\ba
\int_{B_\ep} |\p_t(\n_\ep^\alpha)| dx
& = |\mathbb{S}^{N-1}| \int_0^\ep \alpha \n_\ep^{\alpha-1}(2\ep-r) |\p_t \n_\ep(2\ep-r)| r^{N-1} dr \\
& = |\mathbb{S}^{N-1}| \int_\ep^{2 \ep} \alpha \n_\ep^{\alpha-1}(s) |\p_t \n_\ep(s)| (2\ep-s)^{N-1} ds \\
& \le |\mathbb{S}^{N-1}| \int_\ep^{2 \ep} \alpha \n_\ep^{\alpha-1}(s) |\p_t \n_\ep(s)| s^{N-1} ds \\
& = \int_{B_{2\ep} \setminus B_\ep } |\p_t(\n_\ep^\alpha)| dx.
\ea\ee
The combination of (\ref{jx12}) and (\ref{jx13}) shows that for any $M>0$,
\be\la{jx14}\ba
\p_t(\n_\ep^\alpha) \text{ is bounded in } L^2( 0,T;L^1(B_M) ).
\ea\ee
On the other hand, noticing that
\be\la{jx15}\ba
\na \n_\ep^\alpha = \frac{\alpha}{\alpha-1+\frac{1}{q}} \n_\ep^{1-\frac{1}{q}} \na \n_\ep^{\alpha-1+\frac{1}{q}},
\ea\ee
we deduce from (\ref{cr3}) that for any $M>0$,
\be\la{jx16}\ba
\n_\ep^\alpha \text{ is bounded in } L^\infty(0,T;W^{1,q}(B_M)).
\ea\ee
Using (\ref{cr3}), (\ref{jx14}), (\ref{jx16}), and the Aubin-Lions lemma together with a standard diagonal argument, we obtain that there exists a non-negative function $Z \in L^\infty(\mathbb{R}^N \times (0,T))$ such that for any $M>0$,
\be\la{jx17}\ba
\n_\ep^\alpha \to Z \text{ in } C(\ol{B_M} \times (0,T)).
\ea\ee
Set
\be\la{jx18}\ba
\n \triangleq Z^{\frac{1}{\alpha}}.
\ea\ee
It follows from (\ref{jx17}) that for any $M>0$,
\be\la{jx19}\ba
\n_\ep \to \n \text{ in } C(\ol{B_M} \times (0,T)),
\ea\ee
which gives (\ref{jx01}) and (\ref{jx01a}).

Finally, (\ref{jx19}) and (\ref{cr3}) ensure (\ref{jx001}) and (\ref{jx0001}).
This completes the proof of Lemma \ref{jxl1}.
\end{proof}

\begin{lemma}\la{jxl2}
There exists a function $\mathbf{m}_\alpha \in L^\infty(0,T;L^2(\mathbb{R}^N))$ such that up to a subsequence for any $M>0$ and $s \in [1,\frac{3}{2})$,
\be\la{jx02}\ba
\n_\ep^{2\alpha+1} \mathbf{u}_\ep \to \mathbf{m}_\alpha \textnormal{ in } L^2(0,T;L^s(B_M)).
\ea\ee
Moreover, for any $M>0$,
\be\la{jx002}\ba
\sqrt{\n_\ep} \mathbf{u}_\ep \to \sqrt{\n} \mathbf{u} \textnormal{ in } L^2(B_M \times (0,T)),
\ea\ee
where
\be\la{jx0002}\ba
\mathbf{u} \triangleq
\begin{cases}
\n^{-(2\alpha+1)} \mathbf{m}_\alpha \quad & \textnormal{ if } \n>0, \\
0 \quad & \textnormal{ if } \n=0,
\end{cases}
\ea\ee
and
\be\la{jx0002a}\ba
\sqrt{\n} \mathbf{u} \in L^\infty(0,T;L^2(\mathbb{R}^N)).
\ea\ee
\end{lemma}
\begin{proof}
First, by (\ref{cr3}) and (\ref{cr4}), we have
\be\la{jx021}\ba
\n_\ep^{2\alpha+1} \mathbf{u}_\ep \text{ is bounded in } L^\infty(0,T;L^2(\mathbb{R}^N)),
\ea\ee
which implies that there exists a function $\mathbf{m}_\alpha \in L^\infty(0,T;L^2(\mathbb{R}^N))$ such that for any $M>0$,
\be\la{jx022}\ba
\n_\ep^{2\alpha+1} \mathbf{u}_\ep \stackrel{*}{\rightharpoonup} \mathbf{m}_\alpha \text{ in } L^\infty(0,T;L^2(\mathbb{R}^N)).
\ea\ee
In addition, a direct calculation gives
\be\la{jx21}\ba
\na (\n_\ep^{2\alpha+1} \mathbf{u}_\ep)
& = \mathbf{u}_\ep \otimes \na \n_\ep^{2\alpha+1} + \n_\ep^{2\alpha+1} \na \mathbf{u}_\ep \\
& = \frac{4\alpha+2}{2\alpha-1} \n_\ep^{\alpha+1} \sqrt{\n_\ep} \mathbf{u}_\ep \otimes \na \n_\ep^{\alpha-\frac{1}{2}} + \n_\ep^{\frac{3\alpha}{2}+1} \n_\ep^{\frac{\alpha}{2}} \na \mathbf{u}_\ep,
\ea\ee
and
\be\la{jx22}\ba
\n_\ep^{2\alpha+1} \mathbf{u}_\ep = \n_\ep^{2\alpha+\frac{1}{2}} \sqrt{\n_\ep} \mathbf{u}_\ep.
\ea\ee
From (\ref{cr3}), (\ref{cr4}), (\ref{jx21}), and (\ref{jx22}), we conclude that for any $M>0$,
\be\la{jx23}\ba
\n_\ep^{2\alpha+1} \mathbf{u}_\ep \text{ is bounded in } L^2(0,T;W^{1,1}(B_M)).
\ea\ee
We note that
\be\la{jx24}\ba
\p_t (\n_\ep^{2\alpha+1} \mathbf{u}_\ep) = (2\alpha+1) \n_\ep^{2\alpha} (\n_\ep)_t \mathbf{u}_\ep + \n_\ep^{2\alpha+1} (\mathbf{u}_\ep)_t.
\ea\ee
On the one hand, the continuity equation $(\ref{bjns})_1$ yields
\be\la{jx25}\ba
& (2\alpha+1) \n_\ep^{2\alpha} (\n_\ep)_t \mathbf{u}_\ep \\
& = - (2\alpha+1) \n_\ep^{2\alpha} \div(\n_\ep \mathbf{u}_\ep) \mathbf{u}_\ep \\
& = - \div(\n_\ep^{2\alpha+1} \mathbf{u}_\ep \otimes \mathbf{u}_\ep) - 2\alpha \n_\ep^{2\alpha+1} (\div \mathbf{u}_\ep) \mathbf{u}_\ep + \n_\ep^{2\alpha+1} \mathbf{u}_\ep \cdot \na \mathbf{u}_\ep  \quad \text{ in } \om_\ep.
\ea\ee
On the other hand, we use (\ref{bjns}) and (\ref{bjnxxs}) to derive
\be\la{jx26}\ba
\n_\ep^{2\alpha+1} (\mathbf{u}_\ep)_t
& = - \n_\ep^{2\alpha+1} \mathbf{u}_\ep \cdot \na \mathbf{u}_\ep
+ \n_\ep^{2\alpha} \div( (\n_\ep^\alpha + \ep \n_\ep^\theta) \mathbb{D}\mathbf{u}_\ep) \\
& \quad + \n_\ep^{2\alpha} \na( ((\alpha-1) \n_\ep^\alpha + \ep (\theta-1) \n_\ep^\theta) \div \mathbf{u}_\ep ) - \n_\ep^{2\alpha} \na \n_\ep^\ga \\
& = - \n_\ep^{2\alpha+1} \mathbf{u}_\ep \cdot \na \mathbf{u}_\ep
+ \div( \n_\ep^{2\alpha} (\n_\ep^\alpha + \ep \n_\ep^\theta) \mathbb{D}\mathbf{u}_\ep)
- (\n_\ep^\alpha + \ep \n_\ep^\theta) \mathbb{D}\mathbf{u}_\ep \cdot \na \n_\ep^{2\alpha} \\
& \quad + \na( \n_\ep^{2\alpha} ((\alpha-1) \n_\ep^\alpha + \ep (\theta-1) \n_\ep^\theta) \div \mathbf{u}_\ep ) \\
& \quad - ((\alpha-1) \n_\ep^\alpha + \ep (\theta-1) \n_\ep^\theta) \div \mathbf{u}_\ep \na \n_\ep^{2\alpha}
- \frac{\ga}{\ga+2\alpha} \na \n_\ep^{\ga+2\alpha}  \quad \text{ in } \om_\ep.
\ea\ee
For any $M>0$, fix $0<\delta<M$ and choose $\ep<\de$.
It follows from (\ref{cr3}) and (\ref{cr4}) that
\be\la{jx27}\ba
& \int_0^T \int_{B_M \setminus \ol{B_\de}} \left( \n_\ep^{2\alpha+1} |\mathbf{u}_\ep|^2
+ \n_\ep^{2\alpha+1} |\na \mathbf{u}_\ep| |\mathbf{u}_\ep| \right) dx dt \\
& \le C \int_0^T \int_{B_M \setminus \ol{B_\de}} \left( \n_\ep |\mathbf{u}_\ep|^2
+ \n_\ep^{\frac{\alpha}{2}} |\na \mathbf{u}_\ep| |\sqrt{\n_\ep} \mathbf{u}_\ep| \right) dx dt \\
& \le C \int_0^T \int_{B_M \setminus \ol{B_\de}} \left( \n_\ep |\mathbf{u}_\ep|^2
+ \n_\ep^{\alpha} |\na \mathbf{u}_\ep|^2 \right) dx dt
\le C,
\ea\ee
and
\be\la{jx28}\ba
& \int_0^T \int_{B_M \setminus \ol{B_\de}} \left( \n_\ep^{2\alpha} (\n_\ep^\alpha + \ep \n_\ep^\theta) |\na \mathbf{u}_\ep|
+ (\n_\ep^\alpha + \ep \n_\ep^\theta) |\na \mathbf{u}_\ep| |\na \n_\ep^{2\alpha}|
+ |\na \n_\ep^{\ga+2\alpha}| \right) dx dt \\
& \le C \int_0^T \int_{B_M \setminus \ol{B_\de}} \left( \n_\ep^{\frac{\alpha}{2}} |\na \mathbf{u}_\ep|
+ \n_\ep^{\frac{\alpha}{2}} |\na \mathbf{u}_\ep| |\na \n_\ep^{\alpha-\frac{1}{2}}|
+ |\na \n_\ep^{\alpha-\frac{1}{2}}| \right) dx dt \\
& \le C + C \int_0^T \int_{B_M \setminus \ol{B_\de}} \left( \n_\ep^{\alpha} |\na \mathbf{u}_\ep|^2
+ |\na \n_\ep^{\alpha-\frac{1}{2}}|^2 \right) dx dt
\le C.
\ea\ee
Combining (\ref{jx24}), (\ref{jx25}), (\ref{jx26}), (\ref{jx27}), and (\ref{jx28}), we obtain for any $0<\de<M$,
\be\la{jx29}\ba
\p_t (\n_\ep^{2\alpha+1} \mathbf{u}_\ep) \text{ is bounded in } L^1(0,T;W^{-1,1}(B_M \setminus \ol{B_\de})).
\ea\ee
By (\ref{jx022}), (\ref{jx23}), (\ref{jx29}), the Aubin-Lions lemma, and a standard diagonal argument, we deduce that for any $0<\de<M$ and $s \in [1,\frac{3}{2})$,
\be\la{jx210}\ba
\n_\ep^{2\alpha+1} \mathbf{u}_\ep \to \mathbf{m}_\alpha \text{ in } L^2(0,T;L^s(B_M \setminus \ol{B_\de})).
\ea\ee
Moreover, using (\ref{cr3}), (\ref{cr4}), and H\"older's inequality, we get for any $s \in [1,\frac{3}{2})$,
\be\la{jx0210}\ba
\| \n_\ep^{2\alpha+1} \mathbf{u}_\ep - \mathbf{m}_\alpha \|_{L^2(0,T;L^s(B_\de))}
& \le T^{\frac{1}{2}} \| \n_\ep^{2\alpha+1} \mathbf{u}_\ep - \mathbf{m}_\alpha \|_{L^\infty(0,T;L^2(B_\de))} |B_\de|^{\frac{2-s}{2s}} \\
& \le C \de^{\frac{(2-s)N}{2s}} \to 0 \text{ as } \de \to 0.
\ea\ee
The combination of (\ref{jx210}) and (\ref{jx0210}) yields for any $M>0$ and $s \in [1,\frac{3}{2})$,
\be\la{jx0211}\ba
\n_\ep^{2\alpha+1} \mathbf{u}_\ep \to \mathbf{m}_\alpha \text{ in } L^2(0,T;L^s(B_M)).
\ea\ee
In particular,
\be\la{jx211}\ba
\n_\ep^{2\alpha+1} \mathbf{u}_\ep \to \mathbf{m}_\alpha \text{ almost everywhere in } \mathbb{R}^N \times (0,T).
\ea\ee
Since $\sqrt{\n_\ep} \mathbf{u}_\ep$ is bounded in $L^\infty(0,T;L^2(\mathbb{R}^N))$, Fatou's lemma implies
\be\la{jx212}\ba
\int_0^T \int_{\mathbb{R}^N} \liminf_{\ep \to 0} \frac{ |\n^{2\alpha+1}_\ep u_\ep|^2 }{\n^{4\alpha+1}_\ep} dx dt
\le \liminf_{\ep \to 0} \int_0^T \int_{\mathbb{R}^N} \n_\ep |\mathbf{u}_\ep|^2 dx dt \le C,
\ea\ee
which shows $\mathbf{m}_\alpha(x,t) = 0$ almost everywhere in $\{ (x,t) \in \mathbb{R}^N \times (0,T) \mid \n(x,t)=0 \}$.
Thus, for $\mathbf{u}(x,t)$ as in (\ref{jx0002}), we have
\be\la{jx214}\ba
\mathbf{m}_\alpha(x,t) = \n^{2\alpha+1}(x,t) \mathbf{u}(x,t).
\ea\ee
It remains to prove (\ref{jx002}).
Using (\ref{jx01a}) and (\ref{jx211}), we arrive at
\be\la{jx214a}\ba
\mathbf{u}_\ep \to \mathbf{u} \text{ almost everywhere in } \{ (x,t) \in \mathbb{R}^N \times (0,T) \mid \n(x,t)>0 \},
\ea\ee
which together with (\ref{cr4}) and Fatou's lemma gives
\be\la{jx215}\ba
\int_0^T \int_{\mathbb{R}^N} \n |\mathbf{u}|^p dx dt
& = \int_0^T \int_{\mathbb{R}^N} \chi_{(\n>0)} \n |\mathbf{u}|^p dx dt \\
& = \int_0^T \int_{\mathbb{R}^N} \chi_{(\n>0)} \liminf_{\ep \to 0} \n_\ep |\mathbf{u}_\ep|^p dx dt \\
& \le \liminf_{\ep \to 0} \int_0^T \int_{\mathbb{R}^N} \n_\ep |\mathbf{u}_\ep|^p dx dt \le C.
\ea\ee
For $L>0$ and $\mathbf{v} \in \mathbb{R}^N$, we define the continuous truncation operator
\be\la{jx2150}\ba
T_L(\mathbf{v}) \triangleq \frac{L \mathbf{v}}{ \max\{L,|\mathbf{v}| \} },
\ea\ee
which satisfies
\be\la{jx2151}\ba
|T_L(\mathbf{v})| \le L, \quad T_L(\mathbf{v}) = \mathbf{v} \ \text{ for } |\mathbf{v}| \le L, \quad |\mathbf{v}-T_L(\mathbf{v})| \le |\mathbf{v}| \chi_{(|\mathbf{v}| \ge L)}.
\ea\ee
A direct calculation yields for any $M,L>0$,
\be\la{jx216}\ba
& \int_0^T \int_{B_M} \left| \sqrt{\n_\ep} \mathbf{u}_\ep - \sqrt{\n} \mathbf{u} \right|^2 dx dt \\
& \le C \int_0^T \int_{B_M} \left| \sqrt{\n_\ep} \mathbf{u}_\ep - \sqrt{\n_\ep} T_L(\mathbf{u}_\ep) \right|^2 dx dt
+ C \int_0^T \int_{B_M} \left| \sqrt{\n} \mathbf{u} - \sqrt{\n} T_L(\mathbf{u}) \right|^2 dx dt \\
& \quad + C \int_0^T \int_{B_M} \left| \sqrt{\n_\ep} T_L(\mathbf{u}_\ep) - \sqrt{\n} T_L(\mathbf{u}) \right|^2 dx dt.
\ea\ee
From (\ref{jx01a}) and (\ref{jx211}), we conclude that $\sqrt{\n_\ep} \mathbf{u}_\ep$ converges almost everywhere to $\sqrt{\n} \mathbf{u}$ in the set $\{ (x,t) \in \mathbb{R}^N \times (0,T) \mid \n(x,t)>0 \}$.

Since
\be\la{jx217}\ba
\sqrt{\n_\ep} |T_L(\mathbf{u}_\ep)| \le L \sqrt{\n_\ep},
\ea\ee
and $\n_\ep \to 0$ almost everywhere in the set $\{ (x,t) \in \mathbb{R}^N \times (0,T) \mid \n(x,t)=0 \}$, we obtain by (\ref{jx214a}),
\be\la{jx218}\ba
\sqrt{\n_\ep} T_L(\mathbf{u}_\ep) \to \sqrt{\n} T_L(\mathbf{u}) \text{ almost everywhere in } \mathbb{R}^N \times (0,T).
\ea\ee
It follows from (\ref{cr3}), (\ref{jx217}), (\ref{jx218}), and the dominated convergence theorem that
\be\la{jx219}\ba
\lim_{\ep \to 0} \int_0^T \int_{B_M} \left| \sqrt{\n_\ep} T_L(\mathbf{u}_\ep) - \sqrt{\n} T_L(\mathbf{u}) \right|^2 dx dt = 0.
\ea\ee
On the other hand, (\ref{cr4}), (\ref{jx215}), and (\ref{jx2151}) give
\be\la{jx220}\ba
& \int_0^T \int_{B_M} \left| \sqrt{\n_\ep} \mathbf{u}_\ep - \sqrt{\n_\ep} T_L(\mathbf{u}_\ep) \right|^2 dx dt
+ \int_0^T \int_{B_M} \left| \sqrt{\n} \mathbf{u} - \sqrt{\n} T_L(\mathbf{u}) \right|^2 dx dt \\
& \le \int_0^T \int_{B_M} \left| \sqrt{\n_\ep} \mathbf{u}_\ep \right|^2 \chi_{(|\mathbf{u}_\ep| \ge L)} dx dt
+ \int_0^T \int_{B_M} \left| \sqrt{\n} \mathbf{u} \right|^2 \chi_{(|\mathbf{u}| \ge L)} dx dt \\
& \le \frac{C}{L^{p-2}} \int_0^T \int_{B_M} \left( \n_\ep |\mathbf{u}_\ep|^p + \n |\mathbf{u}|^p \right) dx dt
\le \frac{C}{L^{p-2}}.
\ea\ee
Combining (\ref{jx216}), (\ref{jx219}), and (\ref{jx220}) leads to
\be\la{jx221}\ba
\limsup_{\ep \to 0} \int_0^T \int_{B_M} \left| \sqrt{\n_\ep} \mathbf{u}_\ep - \sqrt{\n} \mathbf{u} \right|^2 dx dt
& \le \frac{C}{L^{p-2}}.
\ea\ee
Letting $L \to \infty$ in (\ref{jx221}), we arrive at (\ref{jx002}).

Finally, (\ref{jx002}) and (\ref{cr4}) imply (\ref{jx0002a}), which completes the proof of Lemma \ref{jxl2}.
\end{proof}

\section{Proofs of Theorems \ref{th1} and \ref{th3}}

In this section, we prove Theorems \ref{th1} and \ref{th3}.

\subsection{Proof of Theorem \ref{th1}}

\noindent\textbf{Proof of Theorem \ref{th1}.}
First, we consider the whole space case, while the bounded domain case can be proved with minor modifications.

Using Lemmas \ref{jxl1} and \ref{jxl2}, we can pass to the limit in the continuity equation to conclude that $(\sqrt{\n},\sqrt{\n}\mathbf{u})$ satisfies (\ref{rjdf1}) and (\ref{rjdf2}).

Next, we prove that $(\sqrt{\n},\sqrt{\n}\mathbf{u})$ satisfies (\ref{rjdf3}).
Let
\be\nonumber\ba
\phi(x,t) = (\phi(x,t)^1,\dots,\phi(x,t)^N) \in C_c^\infty(\mathbb{R}^N \times [0,T])
\ea\ee
be a smooth test function satisfying $\phi(x,T) = 0$.

Multiplying $(\ref{bjns})_2$ by $\phi$, integrating by parts over $\om_\ep$, and using the boundary condition (\ref{bjbjtj1}), we obtain
\be\la{pf11}\ba
& \int_{\om_\ep} \mathbf{m}_{0,\ep} \cdot \phi(x,0) dx
+ \int_0^T \int_{\om_\ep} \left( \sqrt{\n_\ep} \sqrt{\n_\ep} \mathbf{u}_\ep \cdot \phi_t
+ (\sqrt{\n_\ep} \mathbf{u}_\ep \otimes \sqrt{\n_\ep} \mathbf{u}_\ep) : \na \phi \right) dx dt \\
& + \int_0^T \int_{\om_\ep} \n_\ep^\ga \div \phi dx dt
- \int_0^T \int_{\om_\ep} \n_\ep^\alpha \na \mathbf{u}_\ep : \na \phi dx dt
- (\alpha-1) \int_0^T \int_{\om_\ep} \n_\ep^\alpha \div \mathbf{u}_\ep \div \phi dx dt \\
& - \ep \int_0^T \int_{\om_\ep} \n_\ep^\theta \na \mathbf{u}_\ep : \na \phi dx dt
- (\theta-1) \ep \int_0^T \int_{\om_\ep} \n_\ep^\theta \div \mathbf{u}_\ep \div \phi dx dt
- \int_0^T \int_{\p \om_\ep} \n_\ep^\ga (\phi \cdot \mathbf{n}) dS dt \\
& + \int_0^T \int_{\p \om_\ep} \n_\ep^\alpha \mathbf{n} \cdot \mathbb{D}\mathbf{u}_\ep \cdot \phi dS dt
+ (\alpha-1) \int_0^T \int_{\p \om_\ep} \n_\ep^\alpha \div \mathbf{u}_\ep (\phi \cdot \mathbf{n}) dS dt \\
& + \ep \int_0^T \int_{\p \om_\ep} \n_\ep^\theta \mathbf{n} \cdot \mathbb{D}\mathbf{u}_\ep \cdot \phi dS dt
+ (\theta-1) \ep \int_0^T \int_{\p \om_\ep} \n_\ep^\theta \div \mathbf{u}_\ep (\phi \cdot \mathbf{n}) dS dt = 0,
\ea\ee
where $\mathbf{n} = - \frac{x}{\ep}$ denotes the unit outer normal vector on $\p \om_\ep$.

Based on the compactness results established in Section 6, we now pass to the limit in each term of (\ref{pf11}).

Since $\mathbf{m}_{0,\ep} = 0$ in $B_\ep$, we have
\be\la{pf12}\ba
& \left| \int_{\om_\ep} \mathbf{m}_{0,\ep} \cdot \phi(x,0) dx - \int_{\mathbb{R}^N} \mathbf{m}_{0} \cdot \phi(x,0) dx \right| \\
& = \left| \int_{\mathbb{R}^N} \left( \mathbf{m}_{0,\ep} - \mathbf{m}_{0} \right) \cdot \phi(x,0) dx \right| \\
& \le \| \mathbf{m}_{0,\ep} - \mathbf{m}_{0} \|_{L^2(\mathbb{R}^N)} \| \phi(x,0) \|_{L^2(\mathbb{R}^N)},
\ea\ee
which together with (\ref{bjcz2}) yields
\be\la{pf13}\ba
\int_{\om_\ep} \mathbf{m}_{0,\ep} \cdot \phi(x,0) dx \to \int_{\mathbb{R}^N} \mathbf{m}_{0} \cdot \phi(x,0) dx \quad \text{ as } \ep \to 0.
\ea\ee
Using (\ref{jx01}), (\ref{jx002}), and the fact that $\mathbf{u}_\ep = 0$ in $B_\ep$, we derive
\be\la{pf14}\ba
& \int_0^T \int_{\om_\ep} \left( \sqrt{\n_\ep} \sqrt{\n_\ep} \mathbf{u}_\ep \cdot \phi_t
+ (\sqrt{\n_\ep} \mathbf{u}_\ep \otimes \sqrt{\n_\ep} \mathbf{u}_\ep) : \na \phi \right) dx dt \\
& = \int_0^T \int_{\mathbb{R}^N} \left( \sqrt{\n_\ep} \sqrt{\n_\ep} \mathbf{u}_\ep \cdot \phi_t
+ (\sqrt{\n_\ep} \mathbf{u}_\ep \otimes \sqrt{\n_\ep} \mathbf{u}_\ep) : \na \phi \right) dx dt \\
& \to \int_0^T \int_{\mathbb{R}^N} \left( \sqrt{\n} \sqrt{\n} \mathbf{u} \cdot \phi_t
+ (\sqrt{\n} \mathbf{u} \otimes \sqrt{\n} \mathbf{u}) : \na \phi \right) dx dt \quad \text{ as } \ep \to 0.
\ea\ee
It follows from (\ref{cr3}) and H\"older's inequality that
\be\la{pf15}\ba
& \left| \int_0^T \int_{\om_\ep} \n_\ep^\ga \div \phi dx dt - \int_0^T \int_{\mathbb{R}^N} \n^\ga \div \phi dx dt \right| \\
& \le \int_0^T \int_{B_\ep} \n_\ep^\ga |\div \phi| dx dt
+ \left| \int_0^T \int_{\mathbb{R}^N} (\n_\ep^\ga - \n^\ga) \div \phi dx dt \right| \\
& \le C \ep^N + C \sup_{0 \le t \le T} \| \n_\ep - \n \|_{L^2(\text{supp}(\phi))} \| \div \phi \|_{L^2(\mathbb{R}^N \times (0,T))},
\ea\ee
which together with (\ref{jx01}) implies
\be\la{pf16}\ba
\int_0^T \int_{\om_\ep} \n_\ep^\ga \div \phi dx dt \to
\int_0^T \int_{\mathbb{R}^N} \n^\ga \div \phi dx dt \quad \text{ as } \ep \to 0.
\ea\ee

Integrating by parts and using (\ref{jx01}), (\ref{jx001}), (\ref{jx002}), and the fact that $\mathbf{u}_\ep = 0$ in $\ol{B_\ep}$, we arrive at
\be\la{pf17}\ba
& \int_0^T \int_{\om_\ep} \n_\ep^\alpha \na \mathbf{u}_\ep : \na \phi dx dt \\
& = - \int_0^T \int_{\om_\ep} \n_\ep^\alpha \mathbf{u}_\ep \cdot \Delta \phi dx dt
- \int_0^T \int_{\om_\ep} \na \n_\ep^\alpha \cdot \na \phi \cdot \mathbf{u}_\ep dx dt \\
& = - \int_0^T \int_{\mathbb{R}^N} \left( \n_\ep^{\alpha-\frac{1}{2}} \sqrt{\n_\ep} \mathbf{u}_\ep \cdot \Delta \phi
+ \frac{2\alpha}{2\alpha-1} \na \n_\ep^{\alpha-\frac{1}{2}} \cdot \na \phi \cdot \sqrt{\n_\ep} \mathbf{u}_\ep \right) dx dt \\
& \to - \int_0^T \int_{\mathbb{R}^N} \left( \n^{\alpha-\frac{1}{2}} \sqrt{\n} \mathbf{u} \cdot \Delta \phi + \frac{2\alpha}{2\alpha-1} \na \n^{\alpha-\frac{1}{2}} \cdot \na \phi \cdot \sqrt{\n} \mathbf{u} \right) dx dt \quad \text{ as } \ep \to 0,
\ea\ee
which along with (\ref{rjdf4}) gives
\be\la{pf18}\ba
\int_0^T \int_{\om_\ep} \n_\ep^\alpha \na \mathbf{u}_\ep : \na \phi dx dt
\to \left< \n^\alpha \na \mathbf{u}, \na \phi \right> \quad \text{ as } \ep \to 0.
\ea\ee
Similarly,
\be\la{pf19}\ba
\int_0^T \int_{\om_\ep} \n_\ep^\alpha \div \mathbf{u}_\ep \div \phi dx dt
\to \left< \n^\alpha \div \mathbf{u}, \div \phi \right> \quad \text{ as } \ep \to 0.
\ea\ee
By (\ref{cr3}), (\ref{cr4}), and H\"older's inequality, we have
\be\la{pf110}\ba
& \left| \ep \int_0^T \int_{\om_\ep} \n_\ep^\theta \na \mathbf{u}_\ep : \na \phi dx dt
+ (\theta-1) \ep \int_0^T \int_{\om_\ep} \n_\ep^\theta \div \mathbf{u}_\ep \div \phi dx dt \right| \\
& = \left| \ep \int_0^T \int_{\mathbb{R}^N} \n_\ep^\theta \na \mathbf{u}_\ep : \na \phi dx dt
+ (\theta-1) \ep \int_0^T \int_{\mathbb{R}^N} \n_\ep^\theta \div \mathbf{u}_\ep \div \phi dx dt \right| \\
& \le C \sqrt{\ep} \| \sqrt{\ep} \n_\ep^{\frac{\theta}{2}} \na \mathbf{u}_\ep \|_{L^2(\mathbb{R}^N \times (0,T))} \| \na \phi \|_{L^2(\mathbb{R}^N \times (0,T))} \| \n_\ep^{\frac{\theta}{2}} \|_{L^\infty(\mathbb{R}^N \times (0,T))} \\
& \le C \sqrt{\ep}.
\ea\ee
Thus,
\be\la{pf111}\ba
\ep \int_0^T \int_{\om_\ep} \n_\ep^\theta \na \mathbf{u}_\ep : \na \phi dx dt
+ (\theta-1) \ep \int_0^T \int_{\om_\ep} \n_\ep^\theta \div \mathbf{u}_\ep \div \phi dx dt
\to 0 \text{ as } \ep \to 0.
\ea\ee
We next treat the boundary terms.
From (\ref{cr3}), we deduce that
\be\la{pf112}\ba
\left| \int_0^T \int_{\p \om_\ep} \n_\ep^\ga \phi \cdot \mathbf{n} dS dt \right|
\le C \ep^{N-1} \to 0 \quad \text{ as } \ep \to 0.
\ea\ee
Noticing that in the radially symmetric setting $\na \mathbf{u}_\ep = (\na \mathbf{u}_\ep)^{\top}$, we have
\be\la{pf113}\ba
\mathbf{n} \cdot \na \mathbf{u}_\ep \cdot \phi
& = \phi \cdot \na \mathbf{u}_\ep \cdot \mathbf{n}
= \phi \cdot \na ( \mathbf{u}_\ep \cdot \mathbf{n} ) - \phi \cdot \na \mathbf{n} \cdot \mathbf{u}_\ep
= \phi \cdot \na ( \mathbf{u}_\ep \cdot \mathbf{n} ) \\
& = - \phi^i \frac{\p r}{\p x_i} \p_r ( u_\ep )
= (\phi \cdot \mathbf{n}) \div \mathbf{u}_\ep
\quad \text{ on } \p \om_\ep,
\ea\ee
where we have used the following facts
\be\la{pf114}\ba
\mathbf{u}_\ep \cdot \mathbf{n} = - u_\ep, \quad \frac{\p r}{\p x} = \frac{x}{r} = -\mathbf{n}, \quad \div \mathbf{u}_\ep = \p_r u_\ep \quad \text{ on } \p \om_\ep.
\ea\ee
In view of (\ref{pf113}), we obtain
\be\la{pf115}\ba
& \int_0^T \int_{\p \om_\ep} \n_\ep^\alpha \mathbf{n} \cdot \mathbb{D}\mathbf{u}_\ep \cdot \phi dS dt
+ (\alpha-1) \int_0^T \int_{\p \om_\ep} \n_\ep^\alpha \div \mathbf{u}_\ep (\phi \cdot \mathbf{n}) dS dt \\
& = \alpha \int_0^T \int_{\p \om_\ep} \n_\ep^\alpha \div \mathbf{u}_\ep (\phi \cdot \mathbf{n}) dS dt.
\ea\ee
Multiplying $(\ref{bjns})_1$ by $\alpha \n_\ep^{\alpha-1}$ gives
\be\la{pf116}\ba
\p_t \n_\ep^\alpha + \mathbf{u}_\ep \cdot \na \n_\ep^\alpha + \alpha \n_\ep^\alpha \div \mathbf{u}_\ep = 0,
\ea\ee
which together with the fact that $\mathbf{u}_\ep = 0$ on $\p \om_\ep$ yields
\be\la{pf117}\ba
& \left| \alpha \int_0^T \int_{\p \om_\ep} \n_\ep^\alpha \div \mathbf{u}_\ep (\phi \cdot \mathbf{n}) dS dt \right| \\
& = \left| \int_0^T \int_{\p \om_\ep} \left( \p_t \n_\ep^\alpha + \mathbf{u}_\ep \cdot \na \n_\ep^\alpha  \right) (\phi \cdot \mathbf{n}) dS dt \right| \\
& = \left| \int_0^T \int_{\p \om_\ep} \left( \p_t \left( \n_\ep^\alpha (\phi \cdot \mathbf{n}) \right) - \n_\ep^\alpha (\p_t \phi \cdot \mathbf{n}) \right) dS dt \right| \\
& \le \left| \int_{\p \om_\ep} \left( \n_\ep^\alpha (\phi \cdot \mathbf{n}) (T) - \n_\ep^\alpha (\phi \cdot \mathbf{n}) (0) \right) dS \right|
+ C |\p \om_\ep| \\
& \le C \ep^{N-1}.
\ea\ee
Combining (\ref{pf115}) and (\ref{pf117}) leads to
\be\la{pf118}\ba
\int_0^T \int_{\p \om_\ep} \n_\ep^\alpha \mathbf{n} \cdot \mathbb{D}\mathbf{u}_\ep \cdot \phi dS dt
+ (\alpha-1) \int_0^T \int_{\p \om_\ep} \n_\ep^\alpha \div \mathbf{u}_\ep (\phi \cdot \mathbf{n}) dS dt \to 0 \text{ as } \ep \to 0.
\ea\ee
By the same argument, we arrive at
\be\la{pf119}\ba
\ep \int_0^T \int_{\p \om_\ep} \n_\ep^\theta \mathbf{n} \cdot \mathbb{D}\mathbf{u}_\ep \cdot \phi dS dt
+ (\theta-1) \ep \int_0^T \int_{\p \om_\ep} \n_\ep^\theta \div \mathbf{u}_\ep (\phi \cdot \mathbf{n}) dS dt \to 0 \text{ as } \ep \to 0.
\ea\ee
Passing to the limit in (\ref{pf11}) by means of (\ref{pf13}), (\ref{pf14}), (\ref{pf16}), (\ref{pf18}), (\ref{pf19}), (\ref{pf111}), (\ref{pf112}), (\ref{pf118}), and (\ref{pf119}), we conclude that $(\sqrt{\n},\sqrt{\n}\mathbf{u})$ satisfies (\ref{rjdf3}).
This completes the proof of Theorem \ref{th1} for $\om = \mathbb{R}^N$.

Next, we prove the global existence of weak solutions in the bounded domain case.
For $0<\ep<\frac{R}{2}$, we consider the approximate system (\ref{bjns}), (\ref{bjnxxs}), and (\ref{rgylt1}) with smooth initial data in the truncated domain
\be\nonumber\ba
\om_\ep^R \triangleq \{ x \in \mathbb{R}^N \mid \ep<|x|<R \}.
\ea\ee
Let $(\n_\ep,\mathbf{u}_\ep)$ be the local smooth solution to the approximate system.
We first establish some a priori estimates that allow the local solution to be extended globally in time.

We assume that
\be\la{yjyzl}\ba
\int_\ep^R \n_{0,\ep}(r) r^{N-1} dr = M_0,
\ea\ee
where $M_0>0$ is a constant independent of $\ep$.

Arguing as in Lemmas \ref{bjl1} and \ref{bjl2}, we can obtain the standard energy estimates and BD entropy estimates.
\begin{proposition}\la{bdbjl1}
There exists a positive constant $C$ independent of $\ep$ and $T$ such that
\be\la{bdbj01}\ba
& \sup_{0\le t \le T} \int_\ep^R \left( \n_\ep |u_\ep|^2 + \n_\ep^\ga + |\p_r \n_\ep^{\alpha-\frac{1}{2}}|^2 + \ep^2 |\p_r \n_\ep^{\theta-\frac{1}{2}}|^2 \right) r^{N-1} dr \\
& \quad + \int_0^T \int_\ep^R \left( \n_\ep^{\alpha} \frac{|u_\ep|^2}{r^2} + \n_\ep^{\alpha} |\p_r u_\ep|^2 + |\p_r \n_\ep^{\frac{\ga+\alpha-1}{2}}|^{2} \right) r^{N-1} dr dt \\
& \quad + \ep \int_0^T \int_\ep^R \left( \n_\ep^{\theta} \frac{|u_\ep|^2}{r^2}
+ \n_\ep^{\theta} |\p_r u_\ep|^2 + |\p_r \n_\ep^{\frac{\ga+\theta-1}{2}}|^{2} \right) r^{N-1} dr dt
\le C.
\ea\ee
\end{proposition}

We next derive an upper bound for the density away from the origin, together with the corresponding $r$-weighted estimates.
\begin{proposition}\la{bdtl2}
There exists a positive constant $\tilde{C}_1$ independent of $\ep$ and $T$ such that
\be\la{bdt02}\ba
\n_\ep(r,t) \le \tilde{C}_1 \quad \textnormal{ for all } (r,t) \in \left[\frac{R}{2},R\right] \times (0,T).
\ea\ee
For any $2 \le p \le 6$, there exist positive constants $\tilde{C}_2$ and $C$ independent of $\ep$ and $T$ such that
\be\la{bdt02a}\ba
\left( \int_\ep^{R} \chi_{(\n_\ep \le \tilde{C}_2)} r^{N-1} dr \right)^{\frac{2}{p}}
\le C \int_\ep^{R} |\p_r \n_\ep^{\frac{\ga+\alpha-1}{2}}|^{2} r^{N-1} dr.
\ea\ee
Moreover, if $N=2$, for any $\xi>0$, there exists a positive constant $C$ independent of $\ep$ and $T$ such that
\be\la{bdt002}\ba
\sup_{0 \le t \le T} \| \n_\ep r^{\xi} \|_{L^\infty(\ep,R)} \le C.
\ea\ee
If $N=3$, there exists a positive constant $C$ independent of $\ep$ and $T$ such that
\be\la{bdt0002}\ba
\sup_{0 \le t \le T} \| \n_\ep^{2\alpha-1} r \|_{L^\infty(\ep,R)}
+ \int_0^T \| \chi_{(\n_\ep > 2\tilde{C}_1)} \n_\ep^{\ga+\alpha-1} r \|_{L^\infty(\ep,R)} dt \le C,
\ea\ee
and for any $t \in [0,T]$,
\be\la{bdt012}\ba
\int_\ep^{R} \chi_{(\n_\ep > 2\tilde{C}_1)} \n_\ep^{\ga+\alpha-1}(r,t) dr
\le C \int_\ep^{R} |\p_r \n_\ep^{\frac{\ga+\alpha-1}{2}}(r,t)|^2 r^2 dr.
\ea\ee
\end{proposition}
\begin{proof}
First, from (\ref{bdbj01}), we conclude that there exists a positive constant $E_0$ such that
\be\la{bd21}\ba
\int_{\ep}^R \n_\ep^\ga(r,t) r^{N-1} dr \le E_0.
\ea\ee
Thus, for any fixed $t \in [0,T]$, there exists $r(t) \in [R/2,R]$ such that
\be\la{bd22}\ba
\n_\ep^\ga(r(t),t) (r(t))^{N-1} \le \frac{2}{R} \int_{\frac{R}{2}}^R \n_\ep^\ga(r,t) r^{N-1} dr \le \frac{2}{R} E_0,
\ea\ee
which implies
\be\la{bd23}\ba
\n_\ep(r(t),t) \le C.
\ea\ee
The fundamental theorem of calculus, together with (\ref{bd23}), yields for any $r \in [R/2,R]$,
\be\la{bd24}\ba
\n_\ep^{\alpha-\frac{1}{2}}(r,t)
& \le C + \int_{\frac{R}{2}}^R |\p_r \n_\ep^{\alpha-\frac{1}{2}}(r,t)| dr \\
& \le C + C \int_{\frac{R}{2}}^R |\p_r \n_\ep^{\alpha-\frac{1}{2}}(r,t)|^2 dr \le C.
\ea\ee
Since $\alpha > \frac{1}{2}$, we may choose a positive constant $\tilde{C}_1$ such that (\ref{bdt02}) holds.

In addition, we set
\be\la{bd25}\ba
d_0 \triangleq \frac{M_0 N}{2 R^N}.
\ea\ee
By (\ref{yjyzl}) and (\ref{bd25}), we have
\be\la{bd26}\ba
\int_\ep^R \chi_{(\n_\ep \ge d_0)} \n_\ep r^{N-1} dr
& = \int_\ep^R \n_\ep r^{N-1} dr - \int_\ep^R \chi_{(\n_\ep < d_0)} \n_\ep r^{N-1} dr \\
& \ge M_0 - d_0 \frac{R^N}{N} \ge \frac{M_0}{2}.
\ea\ee
On the other hand, using (\ref{bd21}) and H\"older's inequality, we arrive at
\be\la{bd27}\ba
\int_\ep^R \chi_{(\n_\ep \ge d_0)} \n_\ep r^{N-1} dr
& \le \left( \int_\ep^R \chi_{(\n_\ep \ge d_0)} \n_\ep^\ga r^{N-1} dr \right)^{\frac{1}{\ga}}
\left( \int_\ep^R \chi_{(\n_\ep \ge d_0)} r^{N-1} dr \right)^{\frac{\ga-1}{\ga}} \\
& \le E_0^{\frac{1}{\ga}} \left( \int_\ep^R \chi_{(\n_\ep \ge d_0)} r^{N-1} dr \right)^{\frac{\ga-1}{\ga}}.
\ea\ee
Combining (\ref{bd26}) and (\ref{bd27}) leads to
\be\la{bd28}\ba
\int_\ep^R \chi_{(\n_\ep \ge d_0)} r^{N-1} dr \ge \left( \frac{M_0}{ 2 E_0^{\frac{1}{\ga}} } \right)^{\frac{\ga}{\ga-1}} \triangleq h_0 > 0.
\ea\ee
As in the whole-space case, we extend the density continuously to $(0,R)$ by
\be\la{bd29}\ba
\tilde{\n}_\ep(r,t) \triangleq
\begin{cases}
\n_\ep(r,t) \quad & \text{ if } \ep < r < R, \\
\n_\ep(2\ep-r,t) \quad & \text{ if } 0 < r \le \ep.
\end{cases}
\ea\ee
The change of variables yields
\be\la{bd210}\ba
\int_0^\ep |\p_r \tilde{\n}_\ep^{\frac{\ga+\alpha-1}{2}}|^{2} r^{N-1} dr
\le \int_\ep^{2 \ep} |\p_r \n_\ep^{\frac{\ga+\alpha-1}{2}}|^{2} r^{N-1} dr.
\ea\ee
Define
\be\la{bd211}\ba
\Gamma(x,t) \triangleq (d_0^{\frac{\ga+\alpha-1}{2}} - \tilde{\n}_\ep^{\frac{\ga+\alpha-1}{2}}(|x|,t) )_+,
\ea\ee
and
\be\la{bd212}\ba
\ol{\Gamma} \triangleq \frac{1}{|B_R|} \int_{B_R} \Gamma(x,t) dx.
\ea\ee
It follows from (\ref{bd28}), (\ref{bd211}), and Poincar\'e's inequality that
\be\la{bd213}\ba
|\ol{\Gamma}| & \le \| \Gamma - \ol{\Gamma} \|_{L^2(\n_\ep \ge d_0)} \left( |\mathbb{S}^{N-1}| h_0 \right)^{-\frac{1}{2}} \\
& \le \| \Gamma - \ol{\Gamma} \|_{L^2(B_R)} \left( |\mathbb{S}^{N-1}| h_0 \right)^{-\frac{1}{2}}
\le C \| \na \Gamma \|_{L^2(B_R)}.
\ea\ee
This, together with the Sobolev-Poincar\'e inequality, implies that for any $2 \le p \le 6$,
\be\la{bd214}\ba
\| \Gamma \|_{L^p(B_R)}^2 \le C \| \Gamma - \ol{\Gamma} \|_{L^p(B_R)}^2
+ C |\ol{\Gamma}|^2 \le C \| \na \Gamma \|_{L^2(B_R)}^2.
\ea\ee
Set
\be\la{bd215}\ba
\tilde{C}_2 \triangleq \frac{d_0}{2} = \frac{M_0 N}{4 R^N}.
\ea\ee
Consequently,
\be\la{bd216}\ba
\chi_{(\n_\ep \le \tilde{C}_2)} \left( d_0^{\frac{\ga+\alpha-1}{2}} - \frac{d_0}{2}^{\frac{\ga+\alpha-1}{2}} \right)
\le \chi_{(\n_\ep \le \tilde{C}_2)} (d_0^{\frac{\ga+\alpha-1}{2}} - \n_\ep^{\frac{\ga+\alpha-1}{2}})_+.
\ea\ee
For any $2 \le p \le 6$, we use (\ref{bd210}), (\ref{bd214}), and (\ref{bd216}) to derive
\be\la{bd217}\ba
\left( \int_\ep^{R} \chi_{(\n_\ep \le \tilde{C}_2)} r^{N-1} dr \right)^{\frac{2}{p}}
& \le C \left( \int_\ep^{R} \chi_{(\n_\ep \le \tilde{C}_2)} \Gamma^p r^{N-1} dr \right)^{\frac{2}{p}} \\
& \le C \left( \int_0^{R} \Gamma^p r^{N-1} dr \right)^{\frac{2}{p}} \\
& \le C \| \Gamma \|_{L^p(B_R)}^2 \le C \| \na \Gamma \|_{L^2(B_R)}^2 \\
& \le C \int_\ep^{R} |\p_r \n_\ep^{\frac{\ga+\alpha-1}{2}}|^{2} r^{N-1} dr,
\ea\ee
which gives (\ref{bdt02a}).

Finally, using (\ref{bdbj01}) and (\ref{bdt02}), and adapting the arguments in Lemma \ref{cptl2}, we can obtain (\ref{bdt002}) and (\ref{bdt0002}).
This completes the proof of Proposition \ref{bdtl2}.
\end{proof}

Using Propositions \ref{bdbjl1} and \ref{bdtl2} and following arguments similar to those in Sections 4 and 5, we can obtain an $\ep$-uniform upper bound and a strictly positive lower bound for the density, provided that $\alpha$ and $\ga$ satisfy the conditions in Theorem \ref{th1}.
Combined with the standard argument, these estimates imply that the approximate system has a unique global smooth solution.

Moreover, adapting the compactness arguments in Lemmas \ref{jxl1} and \ref{jxl2}, we can obtain the compactness properties required for passing to the limit as $\ep \to 0$.
Arguing as in the whole space case, we can prove the global existence of weak solutions in the ball $B_R$.
This completes the proof of Theorem \ref{th1}.

\subsection{Proof of Theorem \ref{th3}}

\noindent\textbf{Proof of Theorem \ref{th3}.}
We prove Theorem \ref{th3} only in the whole space case; the bounded domain case can be treated similarly.
From (\ref{bj01}), (\ref{bj02}), (\ref{cpt01}), (\ref{cpt001}), (\ref{t04}), (\ref{t05}), and the weak lower semicontinuity, we deduce that, for $N<s_2 \le \min\{s_1,4\}$, with $s_1$ as in Theorem \ref{th3},
\be\la{vv1}\ba
& \sup_{0 \le t < \infty} \left( \| \sqrt{\n} \mathbf{u} \|_{L^2(\mathbb{R}^N)} + \| \n - 1 \|_{L^2(\mathbb{R}^N) \cap L^\infty(\mathbb{R}^N)} + \| \na \n^{\alpha-\frac{1}{2}} \|_{L^2(\mathbb{R}^N)} + \| \na \n^{\alpha-1+\frac{1}{s_2}} \|_{L^{s_2}(\mathbb{R}^N)} \right) \\
& \quad + \int_0^\infty \| \na \n^{\frac{\ga+\alpha-1}{2}} \|_{L^2(\mathbb{R}^N)}^2 dt \le C.
\ea\ee
Choose $b$ sufficiently large such that
\be\la{vv1a}\ba
b > \max\left\{ 1,\alpha,\frac{\ga+\alpha-1}{2} \right\}.
\ea\ee
Combining (\ref{vv1}) and (\ref{vv1a}) yields
\be\la{vv2}\ba
\sup_{0 \le t < \infty} \left( \| \na \n^{b} \|_{L^2(\mathbb{R}^N)} + \| \na \n^{b-\frac{1}{2}} \|_{L^2(\mathbb{R}^N)} + \| \na \n^{b} \|_{L^{s_2}(\mathbb{R}^N)} \right)
+ \int_0^\infty \| \na \n^b \|_{L^2(\mathbb{R}^N)}^{2} dt \le C.
\ea\ee
On the one hand, by (\ref{vv2}) and the Gagliardo-Nirenberg inequality, we have
\be\la{vv3}\ba
\int_0^\infty \| \n^b - 1 \|_{L^{4}(\mathbb{R}^N)}^{4} dt
& \le C \int_0^\infty \| \n^b - 1 \|_{L^{N}(\mathbb{R}^N)}^{2} \| \na \n^b \|_{L^2(\mathbb{R}^N)}^{2} dt \\
& \le C \int_0^\infty \| \na \n^b \|_{L^2(\mathbb{R}^N)}^{2} dt \le C,
\ea\ee
where we have used the estimate
\be\nonumber\ba
\sup_{0 \le t < \infty} \| \n^b - 1 \|_{L^{N}(\mathbb{R}^N)} \le C.
\ea\ee
On the other hand, it follows from $(\ref{ns})_1$, (\ref{vv1}), and H\"older's inequality that
\be\la{vv4}\ba
& \frac{d}{dt} \left( \| \n^b - 1 \|_{L^{4}(\mathbb{R}^N)}^{4} \right) \\
& = 4 b \left< (\n^b - 1)^3 \n^{b-1}, \n_t \right> \\
& = 4 b \int_{\mathbb{R}^N} \na \left( (\n^b - 1)^3 \n^{b-1} \right) \cdot \sqrt{\n} \sqrt{\n} \mathbf{u} dx \\
& \le C \int_{\mathbb{R}^N} \left( |\na \n^b| |\n^b - 1|^{2} |\n^{b-\frac{1}{2}}| |\sqrt{\n} \mathbf{u}|
+ |\n^b - 1|^{3} |\na \n^{b-\frac{1}{2}}| |\sqrt{\n} \mathbf{u}| \right) dx \\
& \le C \| \na \n^{b} \|_{L^2(\mathbb{R}^N)} \| \sqrt{\n} \mathbf{u} \|_{L^2(\mathbb{R}^N)}
+ C \| \na \n^{b-\frac{1}{2}} \|_{L^2(\mathbb{R}^N)} \| \sqrt{\n} \mathbf{u} \|_{L^2(\mathbb{R}^N)}.
\ea\ee
Combining (\ref{vv1}), (\ref{vv2}), (\ref{vv4}), and H\"older's inequality leads to
\be\la{vv5}\ba
\sup_{0 \le t < \infty} \left| \frac{d}{dt} \left( \| \n^b - 1 \|_{L^{4}(\mathbb{R}^N)}^{4} \right) \right| \le C.
\ea\ee
From (\ref{vv3}), (\ref{vv5}) and Lemma \ref{zkxsyl}, we conclude that
\be\la{vv6}\ba
\| \n^b - 1 \|_{L^{4}(\mathbb{R}^N)}^{4} \to 0 \ \text{ as } t \to \infty.
\ea\ee
Using (\ref{vv6}), (\ref{vv2}), and the Gagliardo-Nirenberg inequality, we obtain
\be\la{vv7}\ba
\| \n^b - 1 \|_{L^\infty(\mathbb{R}^N)}
& \le C \| \n^b - 1 \|_{L^{4}(\mathbb{R}^N)}^{ \frac{4(s_2-N)}{s_2 N+4(s_2-N)} }
\| \na \n^b \|_{L^{s_2}(\mathbb{R}^N)}^{\frac{s_2 N}{s_2 N+4(s_2-N)}} \\
& \le C \| \n^b - 1 \|_{L^{4}(\mathbb{R}^N)}^{\frac{4(s_2-N)}{s_2 N+4(s_2-N)}} \to 0 \ \text{ as } t \to \infty.
\ea\ee
This implies that there exist positive constants $T_0$ and $\n_{-}$ such that
\be\la{vv9}\ba
\inf_{x \in \mathbb{R}^N} \n(x,t) \ge \n_{-}>0, \quad \textnormal{ for } t \in [T_0,\infty),
\ea\ee
which gives (\ref{th36}) and completes the proof of Theorem \ref{th3}.

\section{Proof of Theorem \ref{th4}}

In this section, we prove Theorem \ref{th4}.
Under the assumptions of Theorem \ref{th4}, we deduce from \cite{CZZ1,CZZ2} that the system (\ref{ns}), (\ref{i1}), (\ref{cz}), (\ref{nxxs}) with the boundary conditions (\ref{qkjbjtj}) or (\ref{yjybjtj}) has a unique local spherically symmetric classical solution.
Thus, let $T>0$ be a fixed time and $(\n,\mathbf{u})$ be a spherically symmetric classical solution on $\OM \times (0,T]$ with initial data $(\n_0,\mathbf{u}_0)$ satisfying (\ref{th42}).
By spherical symmetry, $\n$ and $\mathbf{u}$ can be written as
\be\nonumber\ba
\n(x,t) = \n(r,t), \quad \mathbf{u}(x,t) = u(r,t)\frac{x}{r}, \quad r = |x|.
\ea\ee

\subsection{Proof of Theorem \ref{th4} for the Cauchy problem}

In this subsection, we prove Theorem \ref{th4} for the Cauchy problem.

First, following Lemmas \ref{bjl1}, \ref{bjl2}, and \ref{cptl1}, we can obtain the standard energy estimates and BD entropy estimates.
\begin{lemma}\la{ltbl1}
There exists a positive constant $C$ independent of $T$ such that
\be\la{ltb01}\ba
& \sup_{0\le t \le T} \int_0^\infty \left( \n |u|^2 + K(\n) + |\p_r \n^{\alpha-\frac{1}{2}}|^2 + |\n^{\alpha-\frac{1}{2}}-1|^2 \right) r^{N-1} dr \\
& \quad + \int_0^T \int_0^\infty \left( \n^{\alpha} \frac{|u|^2}{r^2} + \n^{\alpha} |\p_r u|^2 + |\p_r \n^{\frac{\ga+\alpha-1}{2}}|^{2} \right) r^{N-1} dr dt
\le C.
\ea\ee
\end{lemma}

Arguing as in Lemma \ref{cptl2}, we can establish the following upper bound for the density away from the origin and the corresponding $r$-weighted estimates.
\begin{lemma}\la{ltbl2}
There exist positive constants $\hat{R}_0 \ge 2$, $\hat{c}_1$, and $\hat{c}_2$ independent of $T$ such that
\be\la{ltb02}\ba
0 < \hat{c}_1 \le \n(r,t) \le \hat{c}_2 \quad \textnormal{ for all } (r,t) \in [\hat{R}_0,\infty) \times (0,T).
\ea\ee
For any $2 \le p \le 6$, there exists a positive constant $C$ independent of $T$ such that
\be\la{ltb020}\ba
\left( \int_0^{\hat{R}_0} \chi_{(\n \le \frac{\hat{c}_1}{2})}(r,t) r^{N-1} dr \right)^{\frac{2}{p}}
\le C \int_0^{\hat{R}_0} |\p_r \n^{\frac{\ga+\alpha-1}{2}}(r,t)|^{2} r^{N-1} dr.
\ea\ee
Moreover, if $N=2$, for any $\xi>0$, there exists a positive constant $C$ independent of $T$ such that
\be\la{ltb02a}\ba
\sup_{0 \le t \le T} \| \n r^{\xi} \|_{L^\infty(0,\hat{R}_0)} \le C.
\ea\ee
If $N=3$, there exists a positive constant $C$ independent of $T$ such that
\be\la{ltb02b}\ba
\sup_{0 \le t \le T} \| \n^{2\alpha-1} r \|_{L^\infty(0,\hat{R}_0)} \le C,
\ea\ee
and for any $t \in [0,T]$,
\be\la{ltb02c}\ba
\int_0^{\hat{R}_0} \chi_{(\n > 2\hat{c}_2)} \n^{\ga+\alpha-1}(r,t) dr
\le C \int_0^{\hat{R}_0} |\p_r \n^{\frac{\ga+\alpha-1}{2}}(r,t)|^2 r^2 dr.
\ea\ee
\end{lemma}

Using Lemmas \ref{ltbl1} and \ref{ltbl2} and adapting the arguments in Lemma \ref{tl4}, we can obtain the following estimates.
\begin{lemma}\la{ltbl3}
Let $k$ satisfy
\be\la{ltb030}\ba
N < k < \mathcal{P}_N(\alpha), \quad k \le 4, \quad \alpha-1+\frac{1}{k} > 0.
\ea\ee
There exists a positive constant $\eta_1 \in (0,1]$ depending only on $\alpha$, $\mu$, and the initial data such that if $\alpha$ and $\ga$ satisfy \eqref{th41}, then there exists a positive constant $C$ independent of $T$ such that
\be\la{ltb03}\ba
& \sup_{0\le t \le T} \int_0^\infty \left( \n |u|^{k} + |\p_r \n^{ \alpha-1+\frac{1}{k} }|^{k} \right) r^{N-1} dr \\
& + \int_0^T \int_0^\infty \left( \n^\alpha \frac{|u|^{k}}{r^2} + \n^\alpha |u|^{k-2} |\p_r u|^2 \right) r^{N-1} dr dt
\le C.
\ea\ee
\end{lemma}

We next prove the time-uniform upper bound for the density.
\begin{lemma}\la{ltbl4}
There exists a positive constant $C$ independent of $T$ such that
\be\la{ltb04}\ba
\sup_{0\le t \le T} \| \n \|_{L^\infty(0,\infty)} \le C.
\ea\ee
\end{lemma}
\begin{proof}
Choosing $k>N$ satisfying (\ref{ltb030}), we obtain from (\ref{ltb03}) that
\be\la{ltb41}\ba
\sup_{0\le t \le T} \int_0^\infty |\p_r \n^{ \alpha-1+\frac{1}{k} }|^{k} r^{N-1} dr \le C.
\ea\ee
For $\hat{R}_0$ as in Lemma \ref{ltbl2}, using (\ref{ltb02}), (\ref{ltb030}), (\ref{ltb41}), H\"older's inequality, and the fundamental theorem of calculus, we derive
\be\la{ltb42}\ba
\n^{\alpha-1+\frac{1}{k}}(r)
& \le \n^{\alpha-1+\frac{1}{k}}(\hat{R}_0) - \int_r^{\hat{R}_0} \p_s \n^{\alpha-1+\frac{1}{k}}(s) ds \\
& \le C + \int_0^{\hat{R}_0} |\p_r \n^{\alpha-1+\frac{1}{k}}(r)| dr \\
& \le C + C \left( \int_0^{\hat{R}_0} |\p_r \n^{\alpha-1+\frac{1}{k}}(r)|^{k} r^{N-1} dr \right)^{\frac{1}{k}}
\left( \int_0^{\hat{R}_0} r^{ -\frac{N-1}{k-1} } dr \right)^{\frac{k-1}{k}} \\
& \le C,
\ea\ee
which gives
\be\la{ltb43}\ba
\| \n \|_{L^\infty(0,\hat{R}_0)} \le C.
\ea\ee
Combining (\ref{ltb43}) and (\ref{ltb02}), we arrive at (\ref{ltb04}) and complete the proof of Lemma \ref{ltbl4}.
\end{proof}

Using the time-uniform upper bound for the density and following the arguments as in \cite{CZZ1,CZZ2}, we can conclude that the problem (\ref{ns}), (\ref{i1}), (\ref{cz}), (\ref{nxxs}) with the far-field conditions (\ref{qkjbjtj}) has a unique global spherically symmetric classical solution $(\n,\mathbf{u})$ on $\om \times (0,\infty)$ satisfying (\ref{th44}).

Furthermore, using (\ref{ltb01}), (\ref{ltb03}), and (\ref{ltb04}), and arguing as in the proof of Theorem \ref{th3}, we can obtain the following time-uniform lower bound for the density.
\begin{lemma}\la{ltbl40}
There exists a positive constant $C$ independent of $T$ such that
\be\la{ltb040}\ba
\sup_{0\le t \le T} \| \n^{-1} \|_{L^\infty(0,\infty)} \le C.
\ea\ee
\end{lemma}

\begin{lemma}\la{ltbl5}
There exists a positive constant $C$ independent of $T$ such that
\be\la{ltb05}\ba
& \sup_{0 \le t \le T} \left( \| \na \mathbf{u} \|^2_{L^2(\mathbb{R}^N)} + \| \n_t \|^2_{L^2(\mathbb{R}^N)} \right) \\
& \quad + \int_0^T \left( \| \mathbf{u}_t \|^2_{L^2(\mathbb{R}^N)} + \| \na^2 \mathbf{u} \|^2_{L^2(\mathbb{R}^N)} + \| \n_t \|^2_{L^2(\mathbb{R}^N)} \right) dt \le C.
\ea\ee
\end{lemma}
\begin{proof}
First, by (\ref{ltb01}), (\ref{ltb03}), (\ref{ltb04}), and (\ref{ltb040}), we obtain that, for some $q \in (N,4]$,
\be\la{ltb51}\ba
& \sup_{0 \le t \le T} \left( \| \n - 1 \|_{L^2(\mathbb{R}^N) \cap L^\infty(\mathbb{R}^N)} + \| \n^{-1} \|_{L^\infty(\mathbb{R}^N)} + \| \na \n \|_{L^2(\mathbb{R}^N)} + \| \na \n \|_{L^q(\mathbb{R}^N)} \right) \\
& \quad + \sup_{0 \le t \le T} \left( \| \mathbf{u} \|_{L^2(\mathbb{R}^N)} + \| \mathbf{u} \|_{L^q(\mathbb{R}^N)} \right) + \int_0^T \| \na \n \|^2_{L^2(\mathbb{R}^N)} dt
\le C.
\ea\ee
Noticing that
\be\la{ltb52}\ba
|\na \mathbf{u}|^2 = |\p_r u|^2 + (N-1) \frac{1}{r^2} u^2,
\ea\ee
we deduce from (\ref{ltb01}) that
\be\la{ltb53}\ba
\int_0^T \|\na \mathbf{u}\|^2_{L^2(\mathbb{R}^N)} dt \le C.
\ea\ee
Since $\n>0$ and in the radially symmetric setting $\mathbb{D} \mathbf{u} = \na \mathbf{u}$, we get from $(\ref{ns})_2$ that $\mathbf{u}$ satisfies the following system
\be\la{ltb54}\ba
\begin{cases}
\Delta \mathbf{u} + (\alpha - 1) \na \div \mathbf{u}
= \n^{-\alpha} \left( \n \mathbf{u}_t + \n \mathbf{u} \cdot \na \mathbf{u} + \na P - \na \mu(\n) \cdot \na \mathbf{u} - \na \lambda(\n) \div \mathbf{u} \right), \\
\mathbf{u} \to 0 \ \text{ as } \ |x| \to \infty.
\end{cases}
\ea\ee
The standard $L^p$ estimate of elliptic equations (see \cite{NS}) shows that for any integer $l \ge 0$ and any $p \in (1,\infty)$,
\be\nonumber\ba
\| \na^2 \mathbf{u} \|_{W^{l,p}(\mathbb{R}^N)} \le C \| \n^{-\alpha} \left( \n \mathbf{u}_t + \n \mathbf{u} \cdot \na \mathbf{u} + \na P - \na \mu(\n) \cdot \na \mathbf{u} - \na \lambda(\n) \div \mathbf{u} \right) \|_{W^{l,p}(\mathbb{R}^N)}.
\ea\ee
In particular, using (\ref{gn11}), (\ref{ltb51}), and Young's inequality, we arrive at
\be\nonumber\ba
& \| \na^2 \mathbf{u} \|_{L^2(\mathbb{R}^N)} \\
& \le C \| \n^{-\alpha} \left( \n \mathbf{u}_t + \n \mathbf{u} \cdot \na \mathbf{u} + \na P - \na \mu(\n) \cdot \na \mathbf{u} - \na \lambda(\n) \div \mathbf{u} \right) \|_{L^2(\mathbb{R}^N)} \\
& \le C \left( \| \mathbf{u}_t \|_{L^2(\mathbb{R}^N)} + \| |\mathbf{u}| |\na \mathbf{u}| \|_{L^2(\mathbb{R}^N)} + \| \na \n \|_{L^2(\mathbb{R}^N)} + \| |\na \n| |\na \mathbf{u}| \|_{L^2(\mathbb{R}^N)} \right) \\
& \le C \left( \| \mathbf{u}_t \|_{L^2(\mathbb{R}^N)} + \| \na \n \|_{L^2(\mathbb{R}^N)} + \| \mathbf{u} \|_{L^q(\mathbb{R}^N)} \| \na \mathbf{u} \|_{L^\frac{2q}{q-2}(\mathbb{R}^N)} + \| \na \n \|_{L^q(\mathbb{R}^N)} \| \na \mathbf{u} \|_{L^\frac{2q}{q-2}(\mathbb{R}^N)} \right) \\
& \le C \left( \| \mathbf{u}_t \|_{L^2(\mathbb{R}^N)} + \| \na \n \|_{L^2(\mathbb{R}^N)}
+ \| \na \mathbf{u} \|_{L^2(\mathbb{R}^N)}^{\frac{q-N}{q}} \| \na^2 \mathbf{u} \|_{L^2(\mathbb{R}^N)}^{\frac{N}{q}} \right) \\
& \le \frac{1}{2} \| \na^2 \mathbf{u} \|_{L^2(\mathbb{R}^N)}
+ C \left( \| \mathbf{u}_t \|_{L^2(\mathbb{R}^N)} + \| \na \n \|_{L^2(\mathbb{R}^N)} + \| \na \mathbf{u} \|_{L^2(\mathbb{R}^N)} \right),
\ea\ee
which gives
\be\la{ltb57}\ba
\| \na^2 \mathbf{u} \|_{L^2(\mathbb{R}^N)}
\le C \left( \| \mathbf{u}_t \|_{L^2(\mathbb{R}^N)} + \| \na \n \|_{L^2(\mathbb{R}^N)} + \| \na \mathbf{u} \|_{L^2(\mathbb{R}^N)} \right).
\ea\ee
Multiplying $(\ref{ltb54})_1$ by $2 \mathbf{u}_t$, integrating by parts over $\mathbb{R}^N$, and using (\ref{ltb51}), (\ref{ltb57}), (\ref{gn11}), and Young's inequality, we derive
\be\la{ltb58}\ba
& \frac{d}{dt} \left( \| \na \mathbf{u} \|_{L^2(\mathbb{R}^N)}^2 + (\alpha-1) \| \div \mathbf{u} \|^2_{L^2(\mathbb{R}^N)} \right)
+ 2 \int_{\mathbb{R}^N} \n^{1-\alpha} |\mathbf{u}_t|^2 dx \\
& \le C \int_{\mathbb{R}^N} \left( |\mathbf{u}| |\na \mathbf{u}| + |\na \n| + |\na \n| |\na \mathbf{u}| \right) |\mathbf{u}_t| dx \\
& \le C \left( \| \mathbf{u} \|_{L^q(\mathbb{R}^N)} \| \na \mathbf{u} \|_{L^\frac{2q}{q-2}(\mathbb{R}^N)} + \| \na \n \|_{L^2(\mathbb{R}^N)}
+ \| \na \n \|_{L^q(\mathbb{R}^N)} \| \na \mathbf{u} \|_{L^\frac{2q}{q-2}(\mathbb{R}^N)} \right) \| \mathbf{u}_t \|_{L^2(\mathbb{R}^N)} \\
& \le C \left( \| \na \n \|_{L^2(\mathbb{R}^N)} + \| \na \mathbf{u} \|_{L^2(\mathbb{R}^N)}^{\frac{q-N}{q}} \| \na^2 \mathbf{u} \|_{L^2(\mathbb{R}^N)}^{\frac{N}{q}} \right) \| \mathbf{u}_t \|_{L^2(\mathbb{R}^N)} \\
& \le C \left( \| \na \n \|_{L^2(\mathbb{R}^N)} + \| \na \mathbf{u} \|_{L^2(\mathbb{R}^N)} + \| \na \mathbf{u} \|_{L^2(\mathbb{R}^N)}^{\frac{q-N}{q}} \| \mathbf{u}_t \|_{L^2(\mathbb{R}^N)}^{\frac{N}{q}} \right) \| \mathbf{u}_t \|_{L^2(\mathbb{R}^N)} \\
& \le \int_{\mathbb{R}^N} \n^{1-\alpha} |\mathbf{u}_t|^2 dx + C \| \na \n \|_{L^2(\mathbb{R}^N)}^2
+ C \| \na \mathbf{u} \|_{L^2(\mathbb{R}^N)}^2,
\ea\ee
which gives
\be\la{ltb59}\ba
& \frac{d}{dt} \left( \| \na \mathbf{u} \|_{L^2(\mathbb{R}^N)}^2 + (\alpha-1) \| \div \mathbf{u} \|^2_{L^2(\mathbb{R}^N)} \right)
+ \int_{\mathbb{R}^N} \n^{1-\alpha} |\mathbf{u}_t|^2 dx \\
& \le C \| \na \n \|_{L^2(\mathbb{R}^N)}^2 + C \| \na \mathbf{u} \|_{L^2(\mathbb{R}^N)}^2.
\ea\ee
Integrating (\ref{ltb59}) over $(0,T)$ and using (\ref{ltb51}) and (\ref{ltb53}), one obtains
\be\la{ltb510}\ba
\sup_{0 \le t \le T} \left( \| \na \mathbf{u} \|_{L^2(\mathbb{R}^N)}^2 + (\alpha-1) \| \div \mathbf{u} \|_{L^2(\mathbb{R}^N)}^2 \right)
+ \int_0^T\| \mathbf{u}_t \|_{L^2(\mathbb{R}^N)}^2 dt \le C.
\ea\ee
For any vector-valued function $\mathbf{v}$, we have
\be\la{ltb511}\ba
(\div \mathbf{v})^2 \le N |\na \mathbf{v}|^2,
\ea\ee
which together with the fact that $\frac{N-1}{N} < \alpha \le 1$ yields
\be\la{ltb512}\ba
|\na \mathbf{v}|^2 + (\alpha-1) (\div \mathbf{v})^2
\ge (\alpha N - N + 1) |\na \mathbf{v}|^2 \triangleq N_{\alpha} |\na \mathbf{v}|^2,
\ea\ee
where $N_{\alpha} = \alpha N - N + 1 > 0$.
Thus, 
\be\la{ltb513}\ba
|\na \mathbf{u}|^2 + (\alpha-1) (\div \mathbf{u})^2
\ge N_{\alpha} |\na \mathbf{u}|^2.
\ea\ee
Combining (\ref{ltb510}), (\ref{ltb513}), and (\ref{ltb57}) leads to
\be\la{ltb514}\ba
\sup_{0 \le t \le T} \| \na \mathbf{u} \|_{L^2(\mathbb{R}^N)}^2
+ \int_0^T \left( \| \mathbf{u}_t \|_{L^2(\mathbb{R}^N)}^2 + \| \na^2 \mathbf{u} \|^2_{L^2(\mathbb{R}^N)} \right) dt \le C.
\ea\ee
In addition, by $(\ref{ns})_1$, (\ref{ltb51}), (\ref{ltb514}), and H\"older's inequality, we have
\be\la{ltb515}\ba
\| \n_t \|_{L^2(\mathbb{R}^N)}
& \le \| \n \div \mathbf{u} \|_{L^2(\mathbb{R}^N)} + \| \mathbf{u} \cdot \na \n \|_{L^2(\mathbb{R}^N)} \\
& \le C \| \na \mathbf{u} \|_{L^2(\mathbb{R}^N)} + \| \mathbf{u} \cdot \na \n \|_{L^2(\mathbb{R}^N)}.
\ea\ee
If $N=2$, integration by parts together with (\ref{ltb52}) shows for any $r \in (0,\infty)$,
\be\la{ltb516}\ba
(u(r))^2 & = 2 \int_0^r u (\p_s u) ds \le 2 \int_0^\infty |u| |\p_r u| dr \\
& \le \int_0^\infty \left( \frac{|u|^2}{r^2} + |\p_r u|^2 \right) r dr
\le C \| \na \mathbf{u} \|^2_{L^2(\mathbb{R}^2)},
\ea\ee
which yields
\be\la{ltb517}\ba
\| \mathbf{u} \|^2_{L^\infty(\mathbb{R}^2)} \le C \| \na \mathbf{u} \|^2_{L^2(\mathbb{R}^2)}.
\ea\ee
Thus, from (\ref{ltb51}) and (\ref{ltb517}), we conclude that
\be\la{ltb518}\ba
\| \mathbf{u} \cdot \na \n \|_{L^2(\mathbb{R}^2)} \le \| \mathbf{u} \|_{L^\infty(\mathbb{R}^2)} \| \na \n \|_{L^2(\mathbb{R}^2)} \le C \| \na \mathbf{u} \|_{L^2(\mathbb{R}^2)}.
\ea\ee
If $N=3$, using the Sobolev inequality, we obtain
\be\la{ltb519}\ba
\| \mathbf{u} \|^2_{L^6(\mathbb{R}^3)} \le C \| \na \mathbf{u} \|^2_{L^2(\mathbb{R}^3)},
\ea\ee
which together with (\ref{ltb51}) and H\"older's inequality implies
\be\la{ltb520}\ba
\| \mathbf{u} \cdot \na \n \|_{L^2(\mathbb{R}^3)} \le \| \mathbf{u} \|_{L^6(\mathbb{R}^3)} \| \na \n \|_{L^3(\mathbb{R}^3)} \le C \| \na \mathbf{u} \|_{L^2(\mathbb{R}^3)}.
\ea\ee
It follows from (\ref{ltb515}), (\ref{ltb518}), and (\ref{ltb520}) that
\be\la{ltb521}\ba
\| \n_t \|_{L^2(\mathbb{R}^N)} \le C \| \na \mathbf{u} \|_{L^2(\mathbb{R}^N)},
\ea\ee
which along with (\ref{ltb53}) and (\ref{ltb514}) leads to
\be\la{ltb522}\ba
\sup_{0 \le t \le T} \| \n_t \|^2_{L^2(\mathbb{R}^N)} + \int_0^T \| \n_t \|^2_{L^2(\mathbb{R}^N)} dt
\le C + C \int_0^T \| \na \mathbf{u} \|^2_{L^2(\mathbb{R}^N)} dt \le C.
\ea\ee
This, combined with (\ref{ltb514}), gives (\ref{ltb05}) and completes the proof of Lemma \ref{ltbl5}.
\end{proof}

\begin{lemma}\la{ltbl6}
There exists a positive constant $C$ independent of $T$ such that
\be\la{ltb06}\ba
\sup_{0 \le t \le T} \left( \| \mathbf{u}_t \|_{L^2(\mathbb{R}^N)}^2 + \| \na^2 \mathbf{u} \|^2_{L^2(\mathbb{R}^N)} \right)
+ \int_0^T \| \na \mathbf{u}_t \|_{L^2(\mathbb{R}^N)}^2 dt \le C.
\ea\ee
\end{lemma}
\begin{proof}
First, differentiating $(\ref{ns})_2$ with respect to $t$ gives
\be\la{ltb61}\ba
& \n \mathbf{u}_{tt} + \n_t \mathbf{u}_t + \n \mathbf{u} \cdot \na \mathbf{u}_t
- \div(\n^\alpha \na \mathbf{u}_t)
- (\alpha-1) \na(\n^\alpha \div \mathbf{u}_t) \\
& = - \n_t \mathbf{u} \cdot \na \mathbf{u} - \n \mathbf{u}_t \cdot \na \mathbf{u}
+ \div( (\n^\alpha)_t \na \mathbf{u}) + (\alpha-1) \na((\n^\alpha)_t \div \mathbf{u}) - \na P_t.
\ea\ee
Multiplying (\ref{ltb61}) by $\mathbf{u}_t$, integrating over $\mathbb{R}^N$, and using (\ref{ltb05}), (\ref{ltb51}), (\ref{ltb57}), (\ref{gn11}), and Young's inequality, we obtain
\be\la{ltb62}\ba
& \frac{1}{2} \frac{d}{dt} \int_{\mathbb{R}^N} \n |\mathbf{u}_t|^2 dx + \int_{\mathbb{R}^N} \n^\alpha ( |\na \mathbf{u}_t|^2 + (\alpha-1) (\div \mathbf{u}_t)^2 ) dx \\
& = - \int_{\mathbb{R}^N} \n_t |\mathbf{u}_t|^2 dx
- \int_{\mathbb{R}^N} \left( \n_t \mathbf{u} \cdot \na \mathbf{u} \cdot \mathbf{u}_t
+ \n \mathbf{u}_t \cdot \na \mathbf{u} \cdot \mathbf{u}_t \right) dx \\
& \quad - \int_{\mathbb{R}^N} (\n^\alpha)_t \na \mathbf{u} : \na \mathbf{u}_t dx
- (\alpha-1) \int_{\mathbb{R}^N} (\n^\alpha)_t \div \mathbf{u} \div \mathbf{u}_t dx
+ \int_{\mathbb{R}^N} P_t \div \mathbf{u}_t dx \\
& \le C \| \n_t \|_{L^2(\mathbb{R}^N)} \| \mathbf{u}_t \|_{L^4(\mathbb{R}^N)}^2
+ C \| \na \mathbf{u} \|_{L^2(\mathbb{R}^N)} \| \mathbf{u}_t \|_{L^4(\mathbb{R}^N)}^2 \\
& \quad + C \| \n_t \|_{L^q(\mathbb{R}^N)} \| \na \mathbf{u} \|_{L^{\frac{2q}{q-2}}(\mathbb{R}^N)} \| \na \mathbf{u}_t \|_{L^2(\mathbb{R}^N)}
+ C \| \n_t \|_{L^2(\mathbb{R}^N)} \| \na \mathbf{u}_t \|_{L^2(\mathbb{R}^N)} \\
& \quad + C \| \n_t \|_{L^2(\mathbb{R}^N)} \| \mathbf{u} \|_{L^6(\mathbb{R}^N)} \| \na \mathbf{u} \|_{L^6(\mathbb{R}^N)} \| \mathbf{u}_t \|_{L^6(\mathbb{R}^N)} \\
& \le C \| \mathbf{u}_t \|_{L^2(\mathbb{R}^N)}^{\frac{4-N}{2}} \| \na \mathbf{u}_t \|_{L^2(\mathbb{R}^N)}^{\frac{N}{2}}
+ C (1 + \| \mathbf{u}_t \|_{L^2(\mathbb{R}^N)}) (\| \na \mathbf{u} \|_{L^2(\mathbb{R}^N)} + \| \na^2 \mathbf{u} \|_{L^2(\mathbb{R}^N)}) \| \na \mathbf{u}_t \|_{L^2(\mathbb{R}^N)} \\
& \quad + C \left( \| \na \mathbf{u} \|_{L^2(\mathbb{R}^N)} + \| \na^2 \mathbf{u} \|_{L^2(\mathbb{R}^N)} + \| \n_t \|_{L^2(\mathbb{R}^N)} \right) \left( \| \mathbf{u}_t \|_{L^2(\mathbb{R}^N)} + \| \na \mathbf{u}_t \|_{L^2(\mathbb{R}^N)} \right) \\
& \le \frac{N_\alpha}{2} \int_{\mathbb{R}^N} \n^\alpha |\na \mathbf{u}_t|^2 dx
+ C \left( \| \mathbf{u}_t \|^2_{L^2(\mathbb{R}^N)} + \| \na \mathbf{u} \|^2_{L^2(\mathbb{R}^N)} + \| \na^2 \mathbf{u} \|^2_{L^2(\mathbb{R}^N)} + \| \n_t \|^2_{L^2(\mathbb{R}^N)} \right) \\
& \quad + C \| \sqrt{\n} \mathbf{u}_t \|^2_{L^2(\mathbb{R}^N)} \left(\| \na \mathbf{u} \|^2_{L^2(\mathbb{R}^N)} + \| \na^2 \mathbf{u} \|^2_{L^2(\mathbb{R}^N)} \right),
\ea\ee
where we have used the following estimate
\be\la{ltb63}\ba
\| \n_t \|_{L^q(\mathbb{R}^N)}
& \le \left( \| \n \div \mathbf{u} \|_{L^q(\mathbb{R}^N)} + \| \mathbf{u} \cdot \na \n \|_{L^q(\mathbb{R}^N)} \right) \\
& \le C \left( \| \na \mathbf{u} \|_{H^1(\mathbb{R}^N)} + \| \mathbf{u} \|_{H^2(\mathbb{R}^N)} \| \na \n \|_{L^q(\mathbb{R}^N)} \right) \\
& \le C + C \| \mathbf{u}_t \|_{L^2(\mathbb{R}^N)},
\ea\ee
owing to $(\ref{ns})_1$, (\ref{ltb05}), (\ref{ltb51}), and (\ref{ltb57}).

Applying Gr\"onwall's inequality to (\ref{ltb62}) and using (\ref{ltb05}), (\ref{ltb51}), and (\ref{ltb512}), we arrive at
\be\la{ltb64}\ba
\sup_{0 \le t \le T} \| \mathbf{u}_t \|_{L^2(\mathbb{R}^N)}^2
+ \int_0^T \| \na \mathbf{u}_t \|_{L^2(\mathbb{R}^N)}^2 dt \le C,
\ea\ee
which together with (\ref{ltb63}), (\ref{ltb05}), and (\ref{ltb57}) yields
\be\la{ltb65}\ba
\sup_{0 \le t \le T} \left( \| \na^2 \mathbf{u} \|^2_{L^2(\mathbb{R}^N)} + \| \n_t \|_{L^q(\mathbb{R}^N)} \right) \le C.
\ea\ee
From (\ref{ltb65}) and (\ref{ltb64}), we get (\ref{ltb06}) and complete the proof of Lemma \ref{ltbl6}.
\end{proof}

\noindent\textbf{Proof of (\ref{th45}).}
First, using (\ref{ltb01}), (\ref{ltb03}), and (\ref{ltb04}), and arguing as in the proof of Theorem \ref{th3}, we conclude that
\be\la{ltbp11}\ba
\| \n - 1 \|_{L^4(\mathbb{R}^N)} \to 0 \ \text{ as } t \to \infty.
\ea\ee
This, together with (\ref{ltb51}) and H\"older's inequality, implies for any $s \in (2,\infty)$,
\be\la{ltbp12}\ba
\| \n - 1 \|_{L^s(\mathbb{R}^N)} \to 0 \ \text{ as } t \to \infty.
\ea\ee
It follows from (\ref{ltb53}) and (\ref{ltb06}) that
\be\la{ltbp13}\ba
\int_0^\infty \| \na \mathbf{u} \|_{L^2(\mathbb{R}^N)}^2 dt \le C,
\ea\ee
and
\be\la{ltbp14}\ba
\int_0^\infty \left| \frac{d}{dt} \int_{\mathbb{R}^N} |\na \mathbf{u}|^2 dx \right| dt
& \le C \int_0^\infty \int_{\mathbb{R}^N} |\na \mathbf{u}| |\na \mathbf{u}_t| dx dt \\
& \le C \int_0^\infty \int_{\mathbb{R}^N} \left( |\na \mathbf{u}|^2 + |\na \mathbf{u}_t|^2 \right) dx dt \le C.
\ea\ee
Thus,
\be\la{ltbp15}\ba
\| \na \mathbf{u} \|_{L^2(\mathbb{R}^N)} \to 0 \ \text{ as } t \to \infty,
\ea\ee
which along with (\ref{ltb521}) gives
\be\la{ltbp15a}\ba
\| \n_t \|_{L^2(\mathbb{R}^N)} \to 0 \ \text{ as } t \to \infty.
\ea\ee
Recalling (\ref{ltb62}), we have
\be\la{ltbp16}\ba
& \frac{1}{2} \frac{d}{dt} \int_{\mathbb{R}^N} \n |\mathbf{u}_t|^2 dx
+ \frac{1}{2} \int_{\mathbb{R}^N} \n^\alpha ( |\na \mathbf{u}_t|^2 + (\alpha-1) (\div \mathbf{u}_t)^2 ) dx \\
& \le C \left( \| \mathbf{u}_t \|^2_{L^2(\mathbb{R}^N)} + \| \na \mathbf{u} \|^2_{L^2(\mathbb{R}^N)} + \| \na^2 \mathbf{u} \|^2_{L^2(\mathbb{R}^N)} + \| \n_t \|^2_{L^2(\mathbb{R}^N)} \right) \\
& \quad + C \| \sqrt{\n} \mathbf{u}_t \|^2_{L^2(\mathbb{R}^N)} (\| \na \mathbf{u} \|^2_{L^2(\mathbb{R}^N)} + \| \na^2 \mathbf{u} \|^2_{L^2(\mathbb{R}^N)}) \\
& \le C \left( \| \mathbf{u}_t \|^2_{L^2(\mathbb{R}^N)} + \| \na \mathbf{u} \|^2_{L^2(\mathbb{R}^N)} + \| \na^2 \mathbf{u} \|^2_{L^2(\mathbb{R}^N)} + \| \n_t \|^2_{L^2(\mathbb{R}^N)} \right),
\ea\ee
where in the second inequality we have used (\ref{ltb51}) and (\ref{ltb06}).

From (\ref{ltbp16}), (\ref{ltb05}), (\ref{ltb06}), and (\ref{ltb53}), we deduce that
\be\la{ltbp17}\ba
& \int_0^\infty \left| \frac{d}{dt} \int_{\mathbb{R}^N} \n |\mathbf{u}_t|^2 dx \right| dt \\
& \le C \int_0^\infty \left( \| \na \mathbf{u}_t \|^2_{L^2(\mathbb{R}^N)} + \| \mathbf{u}_t \|^2_{L^2(\mathbb{R}^N)} + \| \na \mathbf{u} \|^2_{L^2(\mathbb{R}^N)} + \| \na^2 \mathbf{u} \|^2_{L^2(\mathbb{R}^N)} + \| \n_t \|^2_{L^2(\mathbb{R}^N)} \right) dt \\
& \le C.
\ea\ee
On the other hand, by (\ref{ltb05}) and (\ref{ltb51}), we arrive at
\be\la{ltbp18}\ba
\int_0^\infty \int_{\mathbb{R}^N} \n |\mathbf{u}_t|^2 dx dt
\le C \int_0^\infty \int_{\mathbb{R}^N} |\mathbf{u}_t|^2 dx dt \le C,
\ea\ee
which together with (\ref{ltbp17}) shows
\be\la{ltbp19}\ba
\| \sqrt{\n} \mathbf{u}_t \|_{L^2(\mathbb{R}^N)} \to 0 \ \text{ as } t \to \infty.
\ea\ee
Consequently, we obtain from (\ref{ltb51}) and (\ref{ltbp19}) that
\be\la{ltbp110}\ba
\| \mathbf{u}_t \|_{L^2(\mathbb{R}^N)} \to 0 \ \text{ as } t \to \infty.
\ea\ee
Applying $\na$ to $(\ref{ns})_1$ yields
\be\la{ltbp111}\ba
& \na \n_t + \na \n \div \mathbf{u} + \n \na \div \mathbf{u} + \na \mathbf{u}^j \p_j \n + \mathbf{u}^j \p_j \na \n = 0.
\ea\ee
Multiplying (\ref{ltbp111}) by $2 \na \n$, integrating by parts over $\mathbb{R}^N$, and using (\ref{ltb05}), (\ref{ltb51}), and (\ref{ltb06}), we derive
\be\la{ltbp112}\ba
\frac{d}{dt} \int_{\mathbb{R}^N} |\na \n|^2 dx
& = - \int_{\mathbb{R}^N} \left( |\nabla \n|^2 \div \mathbf{u} + 2 \n \na \n \cdot \na \div \mathbf{u} + 2 \na \n \cdot \na \mathbf{u}^j \p_j \n \right) dx \\
& \le C \int_{\mathbb{R}^N} \left( |\nabla \n|^2 |\na \mathbf{u}| + |\na \n| |\na^2 \mathbf{u}| \right) dx \\
& \le C \int_{\mathbb{R}^N} \left( |\nabla \n|^q + |\na \mathbf{u}|^{\frac{q}{q-2}} + |\na \n|^2 + |\na^2 \mathbf{u}|^2 \right) dx \\
& \le C + C \| \na \mathbf{u} \|_{H^1(\mathbb{R}^N)}^{\frac{q}{q-2}} \le C.
\ea\ee
This implies that
\be\la{ltbp113}\ba
\sup_{0 \le t <\infty} \left| \frac{d}{dt} \| \na \n \|_{L^{2}(\mathbb{R}^N)}^{2} \right| \le C,
\ea\ee
which together with (\ref{ltb51}), H\"older's inequality, and Lemma \ref{zkxsyl} yields, for any $2 < q_0 < q$,
\be\la{ltbp114}\ba
\| \na \n \|_{L^2(\mathbb{R}^N)} + \| \na \n \|_{L^{q_0}(\mathbb{R}^N)} \to 0 \ \text{ as } t \to \infty.
\ea\ee
Finally, it follows from (\ref{ltb57}), (\ref{ltbp15}), (\ref{ltbp110}), and (\ref{ltbp114}) that
\be\la{ltbp115}\ba
\| \na^2 \mathbf{u} \|_{L^2(\mathbb{R}^N)} \to 0 \ \text{ as } t \to \infty.
\ea\ee
Combining (\ref{ltbp12}), (\ref{ltbp15}), (\ref{ltbp15a}), (\ref{ltbp110}), (\ref{ltbp114}), and (\ref{ltbp115}), we obtain (\ref{th45}) and complete the proof of Theorem \ref{th4} for $\om = \mathbb{R}^N$.

\subsection{Proof of Theorem \ref{th4} on bounded domains}

In this subsection, we assume that $\om=B_R$ and prove (\ref{th46}) and (\ref{th46a}).

\noindent\textbf{Proof of (\ref{th46}) and (\ref{th46a}).}
First, arguing as in the whole space case, we can obtain for some $q \in (N,4]$,
\be\la{ltbp21}\ba
& \sup_{0 \le t \le T} \left( \| \n \|_{L^\infty(B_R)} + \| \n^{-1} \|_{L^\infty(B_R)} + \| \n \|_{W^{1,q}(B_R)} + \| \n_t \|_{L^2(B_R)} + \| \mathbf{u}_t \|_{L^2(B_R)} + \| \mathbf{u} \|_{H^2(B_R)} \right) \\
& \quad + \int_0^T \left( \| \na \n \|^2_{L^2(B_R)} + \| \n_t \|^2_{L^2(B_R)} + \| \mathbf{u} \|^2_{H^2(B_R)} + \| \mathbf{u}_t \|_{H^1(B_R)}^2 \right) dt \le C.
\ea\ee
Define
\be\nonumber\ba
\ol{\n} \triangleq \frac{1}{|B_R|} \int_{B_R} \n dx,
\ea\ee
and
\be\la{ltbp22}\ba
H(\n,\ol{\n}) = \n \int^\n_{\ol{\n}} \frac{P(s)-P(\ol{\n})}{s^2} ds.
\ea\ee
By the continuity equation $(\ref{ns})_1$, we have
\be\la{ltbp23}\ba
\int_{B_R} \n dx = \int_{B_R}\n_0 dx.
\ea\ee
From (\ref{ltbp21}) and (\ref{ltbp22}), we conclude that there exist positive constants $\hat{C}_1$ and $\hat{C}_2$ independent of $T$ such that
\be\la{ltbp24}\ba
\hat{C}_1 (\n-\ol{\n})^2 \le H(\n,\ol{\n})  \le \hat{C}_2 (\n^\ga -\ol{\n}^\ga)( \n-\ol{\n}).
\ea\ee
The definition of $H(\n,\ol{\n})$ together with the continuity equation $(\ref{ns})_1$ implies
\be\la{ltbp25}\ba
\p_t( H(\n,\ol{\n}) ) + \div( H(\n,\ol{\n}) \mathbf{u} ) + (P(\n)-P(\ol{\n})) \div\mathbf{u} = 0.
\ea\ee
Integrating (\ref{ltbp25}) over $B_R$, and using integration by parts and the boundary condition $\mathbf{u} = 0$ on $\p B_R$, we arrive at
\be\la{ltbp26}\ba
\frac{d}{dt} \int_{B_R} H(\n,\ol{\n}) dx = \int_{B_R} \na P \cdot \mathbf{u} dx.
\ea\ee
Multiplying $(\ref{ns})_2$ by $\mathbf{u}$, integrating by parts over $B_R$, and using (\ref{ltbp26}), we get
\be\la{ltbp27}\ba
\frac{d}{dt} \int_{B_R} \left( \frac{1}{2} \n |\mathbf{u}|^2 + H(\n,\ol{\n}) \right) dx
+ \int_{B_R} \n^\alpha \left( |\na \mathbf{u}|^2 + (\alpha-1) (\div \mathbf{u})^2 \right) dx = 0.
\ea\ee
Set
\be\la{ltbp28}\ba
\mathbf{w} \triangleq \mathbf{u} + \n^{-1} \na \n^\alpha.
\ea\ee
From $(\ref{ns})_1$ and $(\ref{ns})_2$, we deduce that $\mathbf{w}$ satisfies
\be\la{ltbp29}\ba
\n(\mathbf{w}_t + \mathbf{u} \cdot \na \mathbf{w}) + \na P = 0.
\ea\ee
Multiplying (\ref{ltbp29}) by $\mathbf{w}$, integrating by parts over $B_R$, and using (\ref{ltbp26}), we derive
\be\la{ltbp210}\ba
\frac{d}{dt} \int_{B_R} \frac{1}{2} \n |\mathbf{w}|^2 dx
& = - \int_{B_R} \na P \cdot \mathbf{w} dx
= - \int_{B_R} \na P \cdot \mathbf{u} dx
- \int_{B_R} \n^{-1} \na P \cdot \na \n^{\alpha} dx \\
& = - \frac{d}{dt} \int_{B_R} H(\n,\ol{\n}) dx - \alpha \ga \int_{B_R} \n^{\ga+\alpha-3} |\na \n|^2 dx,
\ea\ee
which yields
\be\la{ltbp211}\ba
\frac{d}{dt} \int_{B_R} \left( \frac{1}{2} \n |\mathbf{w}|^2 + H(\n,\ol{\n}) \right) dx
+ \alpha \ga \int_{B_R} \n^{\ga+\alpha-3} |\na \n|^2 dx = 0.
\ea\ee
Adding (\ref{ltbp27}) and (\ref{ltbp211}) leads to
\be\la{ltbp212}\ba
\frac{d}{dt} E_1(t) + E_2(t) = 0,
\ea\ee
where
\be\la{ltbp213}\ba
E_1(t) \triangleq \int_{B_R} \left( \frac{1}{2} \n |\mathbf{u}|^2 + \frac{1}{2} \n |\mathbf{w}|^2 + 2 H(\n,\ol{\n}) \right) dx,
\ea\ee
and
\be\la{ltbp214}\ba
E_2(t) \triangleq \int_{B_R} \left( \n^\alpha \left( |\na \mathbf{u}|^2 + (\alpha-1) (\div \mathbf{u})^2 \right) + \alpha \ga \n^{\ga+\alpha-3} |\na \n|^2 \right) dx.
\ea\ee

On the other hand, multiplying $(\ref{ns})_2$ by $\mathcal{B}[\n-\ol{\n} ]$, integrating over $B_R$, and using $(\ref{ns})_1$, (\ref{ltbp21}), (\ref{ltbp23}), Poincar\'e's inequality, Young's inequality, and Lemma \ref{iod}, we obtain
\be\nonumber\ba
& \int_{B_R} (\n^\ga - \ol{\n}^\ga) (\n-\ol{\n}) dx \\
& = \frac{d}{dt} \int_{B_R} \rho \mathbf{u} \cdot \mathcal{B}[\n-\ol{\n}] dx
- \int_{B_R} \rho \mathbf{u} \cdot \mathcal{B}[(\n-\ol{\n})_t] dx
- \int_{B_R} \rho \mathbf{u} \cdot \nabla \mathcal{B}[\n-\ol{\n}] \cdot \mathbf{u} dx \\
& \quad + \int_{B_R} \n^\alpha \p_i \mathbf{u} \cdot \p_i \mathcal{B}[\n-\ol{\n}] dx
+ (\alpha-1) \int_{B_R} \n^\alpha \div \mathbf{u} (\n-\ol{\n}) dx \\
& \le \frac{d}{dt} \int_{B_R} \rho \mathbf{u} \cdot \mathcal{B}[\n-\ol{\n}] dx
+ \int_{B_R} \rho \mathbf{u} \cdot \mathcal{B}[\div (\n \mathbf{u})] dx
+ C \| \mathbf{u} \|^2_{L^4(B_R)} \| \na \mathcal{B}[\n-\ol{\n}] \|_{L^2(B_R)} \\
& \quad + C \| \na \mathbf{u} \|_{L^2(B_R)} \| \na \mathcal{B}[\n-\ol{\n}] \|_{L^2(B_R)}
+ C \| \div \mathbf{u} \|_{L^2(B_R)} \| \n - \ol{\n} \|_{L^2(B_R)} \\
& \le \frac{d}{dt} \int_{B_R} \rho \mathbf{u} \cdot \mathcal{B}[\n-\ol{\n}] dx
+ C \| \n  \mathbf{u} \|^2_{L^2(B_R)} + C \left( \| \na \mathbf{u} \|^2_{L^2(B_R)} + \| \na \mathbf{u} \|_{L^2(B_R)} \right) \| \n-\ol{\n} \|_{L^2(B_R)} \\
& \leq \frac{d}{dt} \int_{B_R} \rho \mathbf{u} \cdot \mathcal{B}[\n-\ol{\n}] dx
+ \frac{1}{2} \int_{B_R} (\n^\ga - \ol{\n}^\ga) (\n-\ol{\n}) dx + C \| \na \mathbf{u} \|^2_{L^2(B_R)},
\ea\ee
which together with (\ref{ltb513}) and (\ref{ltbp214}) yields
\be\la{ltbp216}\ba
\int_{B_R} (\n^\ga - \ol{\n}^\ga) (\n-\ol{\n}) dx
\le 2 \frac{d}{dt} \int_{B_R} \rho \mathbf{u} \cdot \mathcal{B}[\n-\ol{\n}] dx + \hat{C}_3 E_2(t).
\ea\ee
Moreover, by (\ref{ltbp24}) and H\"older's inequality, we have
\be\la{ltbp217}\ba
\left| \int_{B_R} \rho \mathbf{u} \cdot \mathcal{B}[\n-\ol{\n}] dx \right|
& \le C \| \sqrt{\n} \mathbf{u} \|_{L^2(B_R)} \| \n-\ol{\n} \|_{L^2(B_R)} \\
& \le \hat{C}_4 \int_{B_R} \left( \frac{1}{2} \rho |\mathbf{u}|^2 + 2 H(\n,\ol{\n}) \right) dx.
\ea\ee
Set
\be\la{ltbp218}\ba
\hat{\si} \triangleq \min\left\{ \frac{1}{4 \hat{C}_4}, \frac{1}{2 \hat{C}_3} \right\}.
\ea\ee
Multiplying (\ref{ltbp216}) by $\hat{\si}$ and adding the resulting inequality to (\ref{ltbp212}), we arrive at
\be\la{ltbp219}\ba
\frac{d}{dt} \tilde{E}_1(t) + \frac{1}{2} E_2(t) + \hat{\si} \int_{B_R} (\n^\ga - \ol{\n}^\ga) (\n-\ol{\n}) dx
\le 0,
\ea\ee
where
\be\la{ltbp220}\ba
\tilde{E}_1(t) \triangleq \int_{B_R} \left( \frac{1}{2} \n |\mathbf{u}|^2 + \frac{1}{2} \n |\mathbf{w}|^2 + 2 H(\n,\ol{\n}) \right) dx
- 2 \hat{\si} \int_{B_R} \rho \mathbf{u} \cdot \mathcal{B}[\n-\ol{\n}] dx.
\ea\ee
From (\ref{ltbp213}), (\ref{ltbp217}), (\ref{ltbp218}), and (\ref{ltbp220}), we deduce that
\be\la{ltbp221}\ba
\frac{1}{2} E_1(t) \le \tilde{E}_1(t) \le 2 E_1(t).
\ea\ee
Moreover, using (\ref{ltb513}), (\ref{ltbp21}), (\ref{ltbp24}), (\ref{ltbp28}), and Poincar\'e's inequality, we obtain
\be\la{ltbp222}\ba
\int_{B_R} \left( |\na \mathbf{u}|^2 + |\na \n|^2 \right) dx
\le \hat{C}_5 E_2(t),
\ea\ee
and
\be\la{ltbp223}\ba
\tilde{E}_1(t)
& \le 2 \int_{B_R} \left( \frac{1}{2} \n |\mathbf{u}|^2 + \frac{1}{2} \n |\mathbf{w}|^2 + 2 H(\n,\ol{\n}) \right) dx \\
& \le C \int_{B_R} \n \left( |\mathbf{u}|^2 + \n^{-2} |\na \n^\alpha|^2 \right) dx
+ 4 \hat{C}_2 \int_{B_R} (\n^\ga -\ol{\n}^\ga)( \n-\ol{\n}) dx \\
& \le C \int_{B_R} \left( |\na \mathbf{u}|^2 + |\na \n|^2 \right) dx
+ 4 \hat{C}_2 \int_{B_R} (\n^\ga -\ol{\n}^\ga)( \n-\ol{\n}) dx \\
& \le \hat{C}_6 \left( \frac{1}{4} E_2(t) + \hat{\si} \int_{B_R} (\n^\ga -\ol{\n}^\ga)( \n-\ol{\n}) dx \right).
\ea\ee
Set
\be\la{ltbp224}\ba
\hat{\delta} \triangleq \min\left\{ \frac{1}{4 \hat{C}_5}, \frac{1}{\hat{C}_6} \right\}.
\ea\ee
It follows from (\ref{ltbp219}), (\ref{ltbp222}), (\ref{ltbp223}), and (\ref{ltbp224}) that
\be\la{ltbp225}\ba
\frac{d}{dt} \tilde{E}_1(t) + \hat{\de} \tilde{E}_1(t) + \hat{\de} \left( \| \na \mathbf{u} \|_{L^2(B_R)}^2 + \| \na \n \|_{L^2(B_R)}^2 \right)
\le 0,
\ea\ee
which gives
\be\la{ltbp226}\ba
\frac{d}{dt} \left( e^{\hat{\de} t} \tilde{E}_1(t) \right)
+ \hat{\de} e^{\hat{\de} t} \left( \| \na \mathbf{u} \|_{L^2(B_R)}^2 + \| \na \n \|_{L^2(B_R)}^2 \right)
\le 0.
\ea\ee
Integrating (\ref{ltbp226}) over $(0,T)$ and using (\ref{ltbp21}), (\ref{ltbp28}), (\ref{ltbp213}), and (\ref{ltbp221}) lead to
\be\la{ltbp227}\ba
\sup_{0 \le t \le T} \left( e^{\hat{\de} t} \left( \| \mathbf{u} \|^2_{L^2(B_R)} + \| \na \n \|^2_{L^2(B_R)} \right) \right)
+ \int_0^T e^{\hat{\de} t} \left( \| \na \mathbf{u} \|_{L^2(B_R)}^2 + \| \na \n \|_{L^2(B_R)}^2 \right) dt \le C.
\ea\ee
This, together with (\ref{ltbp21}), (\ref{ltbp23}), and H\"older's and Poincar\'e's inequalities, gives, for any $2<q_0<q$,
\be\la{ltbbp227a}\ba
\sup_{0 \le t \le T} e^{\tilde{\de} t} \| \n - \ol{\n_0} \|_{W^{1,q_0}(B_R)} \le C,
\ea\ee
where $\tilde{\de}>0$ depends on $q_0$ but is independent of $T$.

Arguing as in the whole space case, we have
\be\la{ltbp228}\ba
\| \n_t \|_{L^2(B_R)} \le C \| \na \mathbf{u} \|_{L^2(B_R)},
\ea\ee
and
\be\la{ltbp229}\ba
\| \na^2 \mathbf{u} \|_{L^2(B_R)}
\le C \left( \| \mathbf{u}_t \|_{L^2(B_R)} + \| \na \n \|_{L^2(B_R)} + \| \na \mathbf{u} \|_{L^2(B_R)} \right).
\ea\ee
In addition, by an argument similar to that in Lemma \ref{ltbl5}, we arrive at
\be\la{ltbp230}\ba
& \frac{d}{dt} \left( \| \na \mathbf{u} \|_{L^2(B_R)}^2 + (\alpha-1) \| \div \mathbf{u} \|^2_{L^2(B_R)} \right)
+ \int_{B_R} \n^{1-\alpha} |\mathbf{u}_t|^2 dx \\
& \le C \| \na \n \|_{L^2(B_R)}^2 + C \| \na \mathbf{u} \|_{L^2(B_R)}^2.
\ea\ee
Multiplying (\ref{ltbp230}) by $e^{ \hat{\de} t }$, integrating over $(0,T)$, and applying (\ref{ltbp227}), (\ref{ltb512}), (\ref{ltbp21}), (\ref{ltbp228}), and (\ref{ltbp229}), we get
\be\la{ltbp231}\ba
\sup_{0 \le t \le T} \left( e^{\hat{\de} t} \left( \| \na \mathbf{u} \|^2_{L^2(B_R)} + \| \n_t \|^2_{L^2(B_R)} \right) \right)
+ \int_0^T e^{\hat{\de} t} \left( \| \mathbf{u}_t \|_{L^2(B_R)}^2 + \| \na^2 \mathbf{u} \|_{L^2(B_R)}^2 \right) dt \le C.
\ea\ee
Proceeding as in Lemma \ref{ltbl6} and using (\ref{ltbp21}), (\ref{ltbp228}), and (\ref{ltbp229}), we derive
\be\la{ltbp232}\ba
& \frac{1}{2} \frac{d}{dt} \int_{B_R} \n |\mathbf{u}_t|^2 dx
+ \frac{1}{2} \int_{B_R} \n^\alpha ( |\na \mathbf{u}_t|^2 + (\alpha-1) (\div \mathbf{u}_t)^2 ) dx \\
& \le C \left( \| \mathbf{u}_t \|^2_{L^2(B_R)} + \| \na \mathbf{u} \|^2_{L^2(B_R)} + \| \na \n \|^2_{L^2(B_R)} \right).
\ea\ee
Multiplying (\ref{ltbp232}) by $e^{ \hat{\de} t }$, integrating over $(0,T)$, and using (\ref{ltb512}), (\ref{ltbp21}), (\ref{ltbp227}), (\ref{ltbp228}), (\ref{ltbp229}), and (\ref{ltbp231}), we obtain
\be\la{ltbp233}\ba
\sup_{0 \le t \le T} \left( e^{\hat{\de} t} \left( \| \mathbf{u}_t \|^2_{L^2(B_R)} + \| \na^2 \mathbf{u} \|^2_{L^2(B_R)} \right) \right)
+ \int_0^T e^{\hat{\de} t} \| \na \mathbf{u}_t \|_{L^2(B_R)}^2 dt \le C.
\ea\ee
Combining (\ref{ltbp227}), (\ref{ltbp231}), (\ref{ltbp233}), and Poincar\'e's inequality yields (\ref{th46}) and completes the proof of Theorem \ref{th4}.

\section{Appendix: Construction of the smooth approximate initial data}

In this section, we give a detailed construction of the smooth approximate initial data.

\begin{proposition}\la{caid}
Let $N=2$ or $3$.
Assume that
\be\la{caid01}\ba
\alpha>\frac{N-1}{N}, \quad N < s_1 \le p_1, \quad \alpha-1+\frac{1}{s_1}>0.
\ea\ee
Suppose that the spherically symmetric functions $(\n_0,\mathbf{m}_0)$ satisfy
\be\la{caid02}\ba
\begin{cases}
0 \le \n_0 \in L^\infty(\mathbb{R}^N), \quad \n_0 \not\equiv 0, \quad \n_0 - 1 \in L^2(\mathbb{R}^N), \quad \na \n_0^{\alpha-\frac{1}{2}} \in L^2(\mathbb{R}^N), \\
\n_0^{-1}|\mathbf{m}_0|^2 \in L^1(\mathbb{R}^N), \quad \mathbf{m}_0 = 0 \textnormal{ a.e. on } \om_0, \\
\na \n_0^{\alpha-1+\frac{1}{s_1}} \in L^{s_1}(\mathbb{R}^N), \quad \n_0^{-p_1+1}|\mathbf{m}_0|^{p_1} \in L^1(\mathbb{R}^N),
\end{cases}
\ea\ee
where
\be\nonumber\ba
\om_0 \triangleq \{x \in \mathbb{R}^N \mid \n_0(x)=0 \},
\ea\ee
denotes the vacuum set of $\n_0$.
We use the convention that
\be\nonumber\ba
\n_0^{-1}|\mathbf{m}_0|^2 = \n_0^{-p_1+1}|\mathbf{m}_0|^{p_1} = 0 \quad \textnormal{ a.e. on } \om_0.
\ea\ee
Then, for every sufficiently small $\ep \in (0,1)$, there exist spherically symmetric functions
\be\nonumber\ba
\n_{0,\ep} \in C^\infty(\mathbb{R}^N), \quad \mathbf{m}_{0,\ep} \in C_c^\infty(\mathbb{R}^N;\mathbb{R}^N),
\ea\ee
such that, with
\be\nonumber\ba
\mathbf{u}_{0,\ep} \triangleq \n_{0,\ep}^{-1} \mathbf{m}_{0,\ep},
\ea\ee
the following properties hold:

1.
\be\la{caid03}\ba
(\n_{0,\ep},\mathbf{u}_{0,\ep})(x) \to (1,0) \ \textnormal{ as } |x| \to \infty.
\ea\ee

2.
As $\ep \to 0$,
\be\la{caid04}\ba
\begin{cases}
\n_{0,\ep} - 1 \to \n_0 - 1 \textnormal{ in } L^2(\mathbb{R}^N) \cap L^\infty(\mathbb{R}^N), \\
\na \n_{0,\ep}^{\alpha-\frac{1}{2}} \to \na \n_0^{\alpha-\frac{1}{2}} \textnormal{ in } L^2(\mathbb{R}^N), \quad
\na \n_{0,\ep}^{\alpha-1+\frac{1}{s_1}} \to \na \n_{0}^{\alpha-1+\frac{1}{s_1}} \textnormal{ in } L^{s_1}(\mathbb{R}^N), \\
\mathbf{m}_{0,\ep} \to \mathbf{m}_{0} \textnormal{ in } L^2(\mathbb{R}^N), \quad
\n_{0,\ep}^{-1}|\mathbf{m}_{0,\ep}|^2 \to \n_{0}^{-1}|\mathbf{m}_{0}|^2 \textnormal{ in } L^1(\mathbb{R}^N), \\
\n_{0,\ep}^{-p_1+1}|\mathbf{m}_{0,\ep}|^{p_1} \to \n_0^{-p_1+1}|\mathbf{m}_0|^{p_1} \textnormal{ in } L^1(\mathbb{R}^N).
\end{cases}
\ea\ee

3.
\be\la{caid05}\ba
\mathbf{u}_{0,\ep} = 0 \textnormal{ in } \ol{B_\ep}.
\ea\ee

4. For $\theta \in (\frac{N-1}{N},1]$ as in \eqref{rgylt1},
\be\la{caid06}\ba
\n_{0,\ep} \ge
\begin{cases}
\ep \quad & \textnormal{ if } \alpha \le 1, \\
\ep^{\frac{1}{\alpha-\theta}} \quad & \textnormal{ if } \alpha>1.
\end{cases}
\ea\ee
\end{proposition}
\begin{proof}
We divide the proof into four steps.

\noindent\textit{Step 1. Construction of the approximate density.}

First, define
\be\nonumber\ba
\beta \triangleq \alpha-1+\frac{1}{s_1}, \quad H_0 \triangleq \n_0^\beta.
\ea\ee
Since $\n_0 \in L^\infty(\mathbb{R}^N)$ and $\beta>0$, we have
\be\la{caid1}\ba
|H_0 - 1| = |\n_0^\beta - 1| \le C |\n_0 - 1|,
\ea\ee
which together with (\ref{caid02}) yields
\be\la{caid2}\ba
H_0 - 1 \in L^2(\mathbb{R}^N) \cap L^\infty(\mathbb{R}^N).
\ea\ee
From (\ref{caid02}), (\ref{caid2}), and $s_1>2$, we conclude that
\be\la{caid3}\ba
H_0 - 1 \in W^{1,s_1}(\mathbb{R}^N).
\ea\ee
We next prove that
\be\la{caid4}\ba
\na H_0 \in L^2(\mathbb{R}^N).
\ea\ee
Set
\be\la{caid5}\ba
A_1 \triangleq \left\{ x \in \mathbb{R}^N: \n_0(x) \ge \frac{1}{2} \right\}, \quad A_2 \triangleq \left\{ x \in \mathbb{R}^N: \n_0(x) < \frac{1}{2} \right\}.
\ea\ee
A direct calculation shows
\be\la{caid6}\ba
\na H_0 = \na \n_0^{\alpha-1+\frac{1}{s_1}} = \frac{2(\alpha-1+\frac{1}{s_1})}{2\alpha-1} \n_0^{-\frac{1}{2}+\frac{1}{s_1}} \na \n_0^{\alpha-\frac{1}{2}} \quad \text{ in } A_1,
\ea\ee
which together with (\ref{caid02}) gives
\be\la{caid7}\ba
\| \na H_0 \|_{L^2(A_1)} \le C \| \na \n_0^{\alpha-\frac{1}{2}} \|_{L^2(A_1)} \le C.
\ea\ee
On the other hand, by (\ref{caid02}), we have
\be\la{caid8}\ba
|A_2| = \int_{\mathbb{R}^N} \chi_{(\n_0 < \frac{1}{2})} dx
= \int_{\mathbb{R}^N} \chi_{(1-\n_0 > \frac{1}{2})} dx
\le 4 \int_{\mathbb{R}^N} |\n_0 -1|^2 dx \le C.
\ea\ee
Combining (\ref{caid3}) and (\ref{caid8}), and using H\"older's inequality and $s_1>2$, we arrive at
\be\la{caid9}\ba
\| \na H_0 \|_{L^2(A_2)} \le |A_2|^{\frac{1}{2}-\frac{1}{s_1}} \| \na H_0 \|_{L^{s_1}(A_2)} \le C.
\ea\ee
It follows from (\ref{caid2}), (\ref{caid3}), (\ref{caid7}), and (\ref{caid9}) that
\be\la{caid10}\ba
H_0 - 1 \in H^1(\mathbb{R}^N) \cap W^{1,s_1}(\mathbb{R}^N).
\ea\ee
Since $s_1>N$, the Sobolev embedding theorem implies that $H_0$ has a bounded continuous representative.
Moreover, $H_0 - 1 \in L^{s_1}(\mathbb{R}^N)$, and hence
\be\la{caid11}\ba
H_0(x) \to 1 \text{ as } |x| \to \infty.
\ea\ee

Let $\eta$ be a standard non-negative spherically symmetric mollifier satisfying
\be\la{caid12}\ba
\eta(x) \in C_c^\infty(\mathbb{R}^N), \quad \int_{\mathbb{R}^N} \eta(x) dx = 1.
\ea\ee
Define
\be\la{caid13}\ba
\eta_\ep(x) \triangleq \ep^{-N} \eta\left(\frac{x}{\ep}\right),
\ea\ee
and
\be\la{caid13a}\ba
\tilde{H}_{0,\ep} \triangleq \eta_\ep * (H_0-1) + 1 = \eta_\ep * H_0.
\ea\ee
By the standard properties of mollification (see \cite{EL}),
\be\la{caid14}\ba
\| \tilde{H}_{0,\ep} - H_0 \|_{H^1(\mathbb{R}^N)} + \| \tilde{H}_{0,\ep} - H_0 \|_{W^{1,s_1}(\mathbb{R}^N)}
+ \| \tilde{H}_{0,\ep} - H_0 \|_{L^\infty(\mathbb{R}^N)} \to 0 \text{ as } \ep \to 0.
\ea\ee
In addition, it is straightforward to verify that
\be\la{caid15}\ba
\tilde{H}_{0,\ep} \ge 0, \quad \tilde{H}_{0,\ep}(x) \to 1 \text{ as } |x| \to \infty.
\ea\ee
Set
\be\la{caid16}\ba
\n_{0,\ep} \triangleq \left( ( 1 - \si(\ep) ) \tilde{H}_{0,\ep} + \si(\ep) \right)^{\frac{1}{\beta}},
\ea\ee
where
\be\la{caid17}\ba
\si(\ep) \triangleq
\begin{cases}
\ep^\beta \quad & \textnormal{ if } \alpha \le 1, \\
\ep^{\frac{\beta}{\alpha-\theta}} \quad & \textnormal{ if } \alpha>1,
\end{cases}
\ea\ee
with $\theta$ as in (\ref{rgylt1}).

The construction implies that $\n_{0,\ep} \in C^\infty(\mathbb{R}^N)$ is spherically symmetric and
\be\la{caid18}\ba
\n_{0,\ep} \ge (\si(\ep))^{\frac{1}{\beta}} =
\begin{cases}
\ep \quad & \textnormal{ if } \alpha \le 1, \\
\ep^{\frac{1}{\alpha-\theta}} \quad & \textnormal{ if } \alpha>1.
\end{cases}
\ea\ee

\noindent\textit{Step 2. Strong convergence of the approximate density.}

Since
\be\la{caid19}\ba
\si(\ep) \to 0 \quad \text{ as } \ep \to 0,
\ea\ee
it follows from (\ref{caid14}), (\ref{caid16}), and (\ref{caid19}) that
\be\la{caid20}\ba
\na \n_{0,\ep}^{\beta} = (1-\si(\ep)) \na \tilde{H}_{0,\ep}
\to \na H_0 = \na \n_{0}^{\beta} \ \text{ in } L^{2}(\mathbb{R}^N) \cap L^{s_1}(\mathbb{R}^N),
\ea\ee
and
\be\la{caid21}\ba
\n_{0,\ep} \to \n_0 \ \text{ in } L^\infty(\mathbb{R}^N).
\ea\ee

We next prove the $L^2$ convergence.
Let $A_1$ and $A_2$ be the sets defined in (\ref{caid5}).
On the one hand, using (\ref{caid8}) and (\ref{caid21}), we obtain
\be\la{caid22}\ba
\| \n_{0,\ep} - \n_0 \|_{L^2(A_2)}
\le |A_2|^{\frac{1}{2}} \| \n_{0,\ep} - \n_0 \|_{L^\infty(\mathbb{R}^N)} \to 0.
\ea\ee
From (\ref{caid21}), we conclude that for sufficiently small $\ep>0$,
\be\la{caid23}\ba
\frac{1}{4} \le \n_{0,\ep} \le \| \n_0 \|_{L^\infty(\mathbb{R}^N)} + 1 \quad \text{ on } A_1,
\ea\ee
which together with $\beta>0$ yields
\be\la{caid24}\ba
|\n_{0,\ep} - \n_0| \le C |\n_{0,\ep}^\beta - \n_0^\beta|.
\ea\ee
In view of (\ref{caid14}), (\ref{caid19}), and (\ref{caid24}), we derive
\be\la{caid25}\ba
\|\n_{0,\ep} - \n_0\|_{L^2(A_1)}
& \le C \|\n_{0,\ep}^\beta - \n_0^\beta\|_{L^2(A_1)}
= C \| ( 1 - \si(\ep) ) \tilde{H}_{0,\ep} + \si(\ep) - H_0\|_{L^2(A_1)} \\
& = C \| - \si(\ep) (\tilde{H}_{0,\ep}-1) + (\tilde{H}_{0,\ep} - H_0) \|_{L^2(A_1)} \\
& \le C \si(\ep) \| \tilde{H}_{0,\ep}-1 \|_{L^2(A_1)} + C \| \tilde{H}_{0,\ep} - H_0 \|_{L^2(A_1)} \to 0.
\ea\ee
Combining (\ref{caid22}) and (\ref{caid25}) leads to
\be\la{caid26}\ba
\n_{0,\ep} - 1 \to \n_0 - 1 \ \text{ in } L^2(\mathbb{R}^N).
\ea\ee
Since $s_1>2$, a direct calculation gives
\be\la{caid27}\ba
\na \n_{0,\ep}^{\alpha-\frac{1}{2}} = \frac{2\alpha-1}{2(\alpha-1+\frac{1}{s_1})} \n_{0,\ep}^{\frac{1}{2}-\frac{1}{s_1}} \na \n_{0,\ep}^{\alpha-1+\frac{1}{s_1}},
\ea\ee
and
\be\la{caid28}\ba
\na \n_{0}^{\alpha-\frac{1}{2}} = \frac{2\alpha-1}{2(\alpha-1+\frac{1}{s_1})} \n_{0}^{\frac{1}{2}-\frac{1}{s_1}} \na \n_{0}^{\alpha-1+\frac{1}{s_1}}.
\ea\ee
It follows from (\ref{caid27}), (\ref{caid28}), and H\"older's inequality that
\be\la{caid29}\ba
\| \na \n_{0,\ep}^{\alpha-\frac{1}{2}} - \na \n_{0}^{\alpha-\frac{1}{2}} \|_{L^2(\mathbb{R}^N)}
& \le C \| \n_{0,\ep}^{\frac{1}{2}-\frac{1}{s_1}} \|_{L^\infty(\mathbb{R}^N)}
\| \na \n_{0,\ep}^{\alpha-1+\frac{1}{s_1}} - \na \n_{0}^{\alpha-1+\frac{1}{s_1}} \|_{L^2(\mathbb{R}^N)} \\
& \quad + C \| \n_{0,\ep}^{\frac{1}{2}-\frac{1}{s_1}} - \n_{0}^{\frac{1}{2}-\frac{1}{s_1}} \|_{L^\infty(\mathbb{R}^N)}
\|
 \na \n_{0}^{\alpha-1+\frac{1}{s_1}} \|_{L^2(\mathbb{R}^N)}.
\ea\ee
By (\ref{caid20}) and (\ref{caid21}), the right-hand side of (\ref{caid29}) tends to zero.
Hence,
\be\la{caid30}\ba
\na \n_{0,\ep}^{\alpha-\frac{1}{2}} \to \na \n_{0}^{\alpha-\frac{1}{2}} \ \text{ in } L^2(\mathbb{R}^N).
\ea\ee

\noindent\textit{Step 3. Construction of the approximate momentum.}

Define
\be\la{caid31}\ba
\mathbf{G}_0(x) \triangleq
\begin{cases}
\n_0^{-1+\frac{1}{p_1}}(x) \mathbf{m}_0(x) \quad & \text{ if } \n_0(x)>0, \\
0 \quad & \text{ if } \n_0(x)=0,
\end{cases}
\ea\ee
which together with (\ref{caid02}) yields
\be\la{caid32}\ba
|\mathbf{G}_0|^{p_1} = \n_0^{-p_1+1} |\mathbf{m}_0|^{p_1} \in L^1(\mathbb{R}^N).
\ea\ee
Let $A_1$ and $A_2$ be the sets defined in (\ref{caid5}).
Using (\ref{caid02}), we obtain
\be\la{caid32a}\ba
\int_{A_1} |\mathbf{G}_0|^2 dx = \int_{A_1} \n_0^{-2+\frac{2}{p_1}} |\mathbf{m}_0|^2 dx \le C \int_{A_1} \n_0^{-1} |\mathbf{m}_0|^2 dx \le C.
\ea\ee
From (\ref{caid32}) and (\ref{caid8}), we deduce that
\be\la{caid32b}\ba
\int_{A_2} |\mathbf{G}_0|^2 dx \le |A_2|^{1-\frac{2}{p_1}} \left( \int_{A_2} |\mathbf{G}_0|^{p_1} dx \right)^{\frac{2}{p_1}} \le C.
\ea\ee
By (\ref{caid32}), (\ref{caid32a}), and (\ref{caid32b}), we have
\be\la{caid33}\ba
\mathbf{G}_0 \in L^2(\mathbb{R}^N) \cap L^{p_1}(\mathbb{R}^N).
\ea\ee
Choose smooth cutoff functions $\kappa_1,\kappa_2 \in C^\infty([0,\infty))$ satisfying
\be\la{caid34}\ba
0 \le \kappa_1(z) \le 1, \quad \kappa_1(z) =
\begin{cases}
0 \quad & \text{ if } z \le 2, \\
1 \quad & \text{ if } z \ge 3,
\end{cases}
\ea\ee
and
\be\la{caid35}\ba
0 \le \kappa_2(z) \le 1, \quad \kappa_2(z) =
\begin{cases}
1 \quad & \text{ if } z \le 1, \\
0 \quad & \text{ if } z \ge 2.
\end{cases}
\ea\ee
Set
\be\la{caid36}\ba
\tilde{\mathbf{G}}_{0,\ep}(x) \triangleq \kappa_1\left(\frac{|x|}{\ep}\right) \kappa_2\left(\ep |x|\right) \mathbf{G}_0(x).
\ea\ee
Then
\be\la{caid37}\ba
\operatorname{supp} \tilde{\mathbf{G}}_{0,\ep} \subset \left\{ x \in \mathbb{R}^N \mid 2\ep \le |x| \le \frac{2}{\ep} \right\}.
\ea\ee
Since
\be\la{caid37a}\ba
\kappa_1\left(\frac{|x|}{\ep}\right) \kappa_2\left(\ep |x|\right) \to 1 \text{ as } \ep \to 0, \quad \text{ for } x \ne 0,
\ea\ee
the dominated convergence theorem and (\ref{caid33}) imply
\be\la{caid38}\ba
\tilde{\mathbf{G}}_{0,\ep} \to \mathbf{G}_0  \text{ in } L^2(\mathbb{R}^N) \cap L^{p_1}(\mathbb{R}^N).
\ea\ee
Let $\eta_\ep$ be as in (\ref{caid13}).
Choose $0<\tau_\ep<\frac{\ep}{4}$ such that
\be\la{caid39}\ba
\| \eta_{\tau_\ep} * \tilde{\mathbf{G}}_{0,\ep} - \tilde{\mathbf{G}}_{0,\ep} \|_{L^2(\mathbb{R}^N) \cap L^{p_1}(\mathbb{R}^N)} \le \ep.
\ea\ee
Define
\be\la{caid40}\ba
\mathbf{G}_{0,\ep} \triangleq \eta_{\tau_\ep} * \tilde{\mathbf{G}}_{0,\ep}.
\ea\ee
Since $\eta$ and $\tilde{\mathbf{G}}_{0,\ep}$ are spherically symmetric, $\mathbf{G}_{0,\ep}$ is also spherically symmetric.
Moreover, from (\ref{caid37}), (\ref{caid38}), and (\ref{caid39}), we deduce that
\be\la{caid41}\ba
\mathbf{G}_{0,\ep} \in C_c^\infty(\mathbb{R}^N;\mathbb{R}^N), \quad \mathbf{G}_{0,\ep} = 0 \text{ in } \ol{B_\ep}, \quad \mathbf{G}_{0,\ep} \to \mathbf{G}_0  \text{ in } L^2(\mathbb{R}^N) \cap L^{p_1}(\mathbb{R}^N).
\ea\ee
Set
\be\la{caid42}\ba
\mathbf{m}_{0,\ep} \triangleq \n_{0,\ep}^{1-\frac{1}{p_1}} \mathbf{G}_{0,\ep}, \quad \mathbf{u}_{0,\ep} \triangleq \n_{0,\ep}^{-\frac{1}{p_1}} \mathbf{G}_{0,\ep}.
\ea\ee

\noindent\textit{Step 4. Strong convergence of the approximate momentum.}

By (\ref{caid32}), (\ref{caid41}) and (\ref{caid42}), one has
\be\la{caid43}\ba
\n_{0,\ep}^{-p_1+1} |\mathbf{m}_{0,\ep}|^{p_1} = |\mathbf{G}_{0,\ep}|^{p_1} \to |\mathbf{G}_{0}|^{p_1} = \n_{0}^{-p_1+1} |\mathbf{m}_{0}|^{p_1} \text{ in } L^1(\mathbb{R}^N).
\ea\ee
Since $p_1>2$, it follows from H\"older's inequality that
\be\la{caid44}\ba
& \| \n_{0,\ep}^{-1} |\mathbf{m}_{0,\ep}|^{2} - \n_{0}^{-1} |\mathbf{m}_{0}|^{2} \|_{L^1(\mathbb{R}^N)} \\
& = \| \n_{0,\ep}^{1-\frac{2}{p_1}} |\mathbf{G}_{0,\ep}|^{2} - \n_{0}^{1-\frac{2}{p_1}} |\mathbf{G}_{0}|^{2} \|_{L^1(\mathbb{R}^N)} \\
& \le \| \n_{0,\ep}^{1-\frac{2}{p_1}} \|_{L^\infty(\mathbb{R}^N)}
\| |\mathbf{G}_{0,\ep}|^{2} - |\mathbf{G}_{0}|^{2} \|_{L^1(\mathbb{R}^N)}
+ \| \n_{0,\ep}^{1-\frac{2}{p_1}} - \n_{0}^{1-\frac{2}{p_1}} \|_{L^\infty(\mathbb{R}^N)}
\| \mathbf{G}_{0} \|^2_{L^2(\mathbb{R}^N)},
\ea\ee
which together with (\ref{caid41}) and (\ref{caid21}) gives
\be\la{caid45}\ba
\n_{0,\ep}^{-1} |\mathbf{m}_{0,\ep}|^{2} \to \n_{0}^{-1} |\mathbf{m}_{0}|^{2} \text{ in } L^1(\mathbb{R}^N).
\ea\ee
Similarly,
\be\la{caid46}\ba
& \| \mathbf{m}_{0,\ep} - \mathbf{m}_{0} \|_{L^2(\mathbb{R}^N)} \\
& = \| \n_{0,\ep}^{1-\frac{1}{p_1}} \mathbf{G}_{0,\ep} - \n_{0}^{1-\frac{1}{p_1}} \mathbf{G}_{0} \|_{L^2(\mathbb{R}^N)} \\
& \le \| \n_{0,\ep}^{1-\frac{1}{p_1}} \|_{L^\infty(\mathbb{R}^N)}
\| \mathbf{G}_{0,\ep} - \mathbf{G}_{0} \|_{L^2(\mathbb{R}^N)}
+ \| \n_{0,\ep}^{1-\frac{1}{p_1}} - \n_{0}^{1-\frac{1}{p_1}} \|_{L^\infty(\mathbb{R}^N)}
\| \mathbf{G}_{0} \|_{L^2(\mathbb{R}^N)},
\ea\ee
which along with (\ref{caid41}) and (\ref{caid21}) implies
\be\la{caid47}\ba
\mathbf{m}_{0,\ep} \to \mathbf{m}_{0} \text{ in } L^2(\mathbb{R}^N).
\ea\ee
Finally, we conclude from (\ref{caid15}), (\ref{caid16}), (\ref{caid19}), (\ref{caid41}), and (\ref{caid42}) that
\be\nonumber\ba
(\n_{0,\ep},\mathbf{u}_{0,\ep})(x) \to (1,0) \ \textnormal{ as } |x| \to \infty.
\ea\ee
This proves all the properties stated in Proposition \ref{caid}.
The proof of Proposition \ref{caid} is complete.
\end{proof}

\section*{Acknowledgments}
The research of Z. Liang was  partially  supported by the NNSFC of China grant 12531010 and the Sichuan Science and Technology Program grant 2025ZYD0157.

\bigskip

\noindent\textbf{Data availability.} No data was used for the research described in the article.

\bigskip

\noindent\textbf{Conflict of interest.} The authors declare that they have no conflict of interest.

\end{document}